\documentclass[11pt,reqno]{amsart}

\usepackage[american]{babel}
\usepackage{comment}
\usepackage{amsmath}
\usepackage{dsfont,mathtools,amssymb,tikz}
\usetikzlibrary{decorations.pathreplacing}
\usepackage{color}
\usepackage{color}
\usepackage[hidelinks]{hyperref}
\mathtoolsset{showonlyrefs,showmanualtags}

\newtheorem{theorem}{Theorem}[section]
\newtheorem{proposition}[theorem]{Proposition}
\newtheorem{lemma}[theorem]{Lemma}
\newtheorem{corollary}[theorem]{Corollary}
\newtheorem{definition}[theorem]{Definition}

\theoremstyle{definition}
\newtheorem{remark}[theorem]{Remark}

\numberwithin{equation}{section}
\numberwithin{figure}{section}

\newcommand{\E}{\mathbb{E}}
\newcommand{\Prob}{\mathbf{P}}

\renewcommand{\rho}{\varrho}
\renewcommand{\phi}{\varphi}

\newcommand{\R}{\mathbb{R}}
\newcommand{\N}{\mathbb{N}}

\newcommand{\dd}{{\rm d}}
\newcommand{\ee}{\mathrm{e}}

\newcommand{\Ex}{\textsf{E}}
\renewcommand{\Pr}{\mathsf{P}}

\DeclareMathOperator{\Exp}{Exponential}

\DeclarePairedDelimiter\floor{\lfloor}{\rfloor}
\DeclareMathOperator{\tx}{\textnormal{\texttt{tx}}}

\newcommand{\Expect}{\mathbf{E}}

\newcommand{\cA}{\ensuremath{\mathcal A}} 
\newcommand{\cB}{\ensuremath{\mathcal B}} 
 
\newcommand{\cD}{\ensuremath{\mathcal D}} 
\newcommand{\cE}{\ensuremath{\mathcal E}} 
\newcommand{\cF}{\ensuremath{\mathcal F}} 
\newcommand{\cG}{\ensuremath{\mathcal G}} 
\newcommand{\cH}{\ensuremath{\mathcal H}} 
 
\newcommand{\cJ}{\ensuremath{\mathcal J}}

\newcommand{\cO}{\ensuremath{\mathcal O}} 
\newcommand{\cP}{\ensuremath{\mathcal P}} 
\newcommand{\cQ}{\ensuremath{\mathcal Q}} 
\newcommand{\cR}{\ensuremath{\mathcal R}} 
\newcommand{\cS}{\ensuremath{\mathcal S}} 
\newcommand{\cT}{\ensuremath{\mathcal T}}

\newcommand{\cW}{\ensuremath{\mathcal W}} 
\newcommand{\cX}{\ensuremath{\mathcal X}} 
 
\newcommand{\cZ}{\ensuremath{\mathcal Z}} 

\newcommand{\bbT}{{\ensuremath{\mathbb T}} }

\renewcommand{\P}{{\ensuremath{\mathbb P}} }

\newcommand{\ind}{\mathds{1}}

\def\({\left(}
\def\){\right)}
\def\[{\left[}
\def\]{\right]}

\renewcommand{\complement}{\texttt{C}}

\allowdisplaybreaks
\begin{document}
	
\date{\today}
	
\title[Voter model on dynamic random graphs]{The voter model on random regular graphs\\ with random rewiring}
	
\author[Avena, Baldasso, Hazra, den Hollander, Quattropani]{Luca Avena, Rangel Baldasso, Rajat Subhra Hazra,\\ Frank den Hollander, Matteo Quattropani}
	
\address{Luca Avena, Dipartimento di Matematica e Informatica `Ulisse Dini', Viale Morgagni 67/a, 50134, Firenze, Italy}
\email{luca.avena@unifi.it}
	
\address{Rangel Baldasso, PUC-Rio, Rua Marqu\^es de S\~ao Vicente, 225, G\'avea, 22451-900 Rio de Janeiro, RJ - Brasil}
\email{rangel@puc-rio.br}
	
\address{Rajat Subhra Hazra, Mathematical Institute, Leiden University, Einsteinweg 55, 2333 CC Leiden, The Netherlands}
\email{r.s.hazra@math.leidenuniv.nl}
	
\address{Frank den Hollander, Mathematical Institute, Leiden University, Einsteinweg 55, 2333 CC Leiden, The Netherlands}
\email{denholla@math.leidenuniv.nl}
	
\address{Matteo Quattropani, Dipartimento di Matematica e Fisica, Universit\`a degli Studi Roma Tre, Via della Vasca Navale 84, 00146, Roma, Italy}
\email{matteo.quattropani@uniroma3.it}
	
\begin{abstract}
We consider the voter model with binary opinions on a random regular graph with $n$ vertices of degree $d \geq 3$, subject to a rewiring dynamics in which pairs of edges are rewired, i.e., broken into four half-edges and subsequently reconnected at random. A parameter $\nu \in (0,\infty)$ regulates the frequency at which the rewirings take place, in such a way that any given edge is rewired exponentially at a rate $\nu$ in the limit as $n\to\infty$. We show that, under the joint law of the random rewiring dynamics and the random opinion dynamics, the fraction of vertices with either one of the two opinions converges on time scale $n$ to the Fisher-Wright diffusion with an explicit diffusion constant $\vartheta_{d,\nu}$ in the limit as $n\to\infty$. In particular, we identify $\vartheta_{d,\nu}$ in terms of a continued-fraction expansion and analyse its dependence on $d$ and $\nu$. A key role in our analysis is played by the set of discordant edges, which constitutes the boundary between the sets of vertices carrying the two opinions.   
		
		
		
\medskip\noindent
\emph{Acknowledgement.} 
The research in this paper was supported by the Netherlands Organisation for Scientific Research (NWO) through the NETWORKS Gravitation programme under grant agreement no.\ 024.002.003. RB has counted on the support of the Conselho Nacional de Desenvolvimento Científico e Tecnológico - CNPq' grants Projeto Universal (402952/2023-5) and Produtividade em Pesquisa (308018/2022-2), MQ on the support support of the European Union's Horizon 2020 research and innovation programme under the Marie Sk\l odowska-Curie grant agreement no.\ 101034253. MQ is a member of GNAMPA INdAM and acknowledges partial support through the GNAMPA INdAM project ``Redistibution models on networks''. FdH and MQ are grateful to the Simons Institute for the Theory of Computing in Berkeley, California, USA for hospitality during a one-month visit in the Fall of 2022. 		
\end{abstract}
	
\maketitle
	
\small
\tableofcontents
\normalsize
	
%%%%%%%%%%%% SECTION 1 %%%%%%%%%%%%%%%%%%%%%%%%%%%%%%%%%
	
\section{Introduction}
\label{sec:model}

Random processes on \emph{static} random graphs has been an important topic of research in network science in recent years. By now their behaviour is relatively well understood, and a number of mathematical techniques that were developed for their analysis have been progressively refined \cite{D10,RvdH17,RvdH24}. However, analysing random processes on \emph{dynamic} random graphs remains a major challenge. The latter are natural models for systems in which the underlying geometry evolves alongside the process that evolves on it. Models fall into two classes:
\begin{enumerate}
	\item \emph{One-way feedback}: The random graph affects the random process, but itself evolves autonomously.
	\item \emph{Two-way feedback}: The random process and the random graph mutually affect each other.
\end{enumerate}
The model considered in the present paper belongs to the first class. In the second class, where the two dynamics are in \emph{co-evolution}, there are so far only a handful of examples for which mathematical progress has been made. In recent years, new techniques have emerged that allow for an extension of the results obtained for static random graphs to dynamic random graphs. Recent key references for the analysis of random processes on \emph{dynamic} random graphs include: simple random walk \cite{AGHH18,AGHH19,ST20,AGHHN22,PSS15,PSS20,HS20,M24,CQ21}, the contact process \cite{JM17,SOV21pr,SV23pr,JLM24pr}, and the voter model \cite{HN06,DGLMSSSV12,BDZ15,BS17,BdHM22pr,Fernley24}. Nonetheless, the overall picture remains fragmented.

Within the realm of static random graphs, \emph{random regular graphs} represent one of the most extensively studied examples, because the \emph{local structure} of such graphs converges, as their size tends to infinity, to a simple deterministic graph, namely, the infinite $d$-regular tree. The most studied processes on random regular graphs include simple random walk \cite{LS10,CF05}, the voter model \cite{CFR10,ABHdHQ24,O13}, the contact process \cite{MV16}, and the stochastic Ising model \cite{DHJN17,HHK21}. For extensions to directed random graphs, see \cite{BCS18,ACHQpr,C24pr}. In the present paper we use random regular graphs as a building block for the study of the voter model on dynamic random graphs. 

%%%

\subsection{Voter model on static random graphs}

The voter model \cite{HL75} is a classical interacting particle system that can be described as follows. Initially, each vertex of the graph is equipped with a binary opinion. At the arrival times of a Poisson process, a vertex selects one of its neighbours uniformly at random and copies its opinion. Clearly, if the underlying graph is connected and finite, then the system eventually gets absorbed into one of the two \emph{consensus configurations} where all the vertices share the same opinion.  

As pointed out in \cite{AF02,A12,D10}, it is heuristically clear that \emph{universal behaviour} emerges for the voter model in the scaling limit when the underlying graph is \emph{locally transient}, i.e., when it converges to a transient graph in the local-weak sense. More precisely, focusing on the density of one of the two opinions and scaling time appropriately, convergence is expected towards the solution of the classical SDE for the \emph{Fisher-Wright diffusion} $(\mathfrak{B}_t)_{t \geq 0}$, which is given by
$$
{\rm d}\mathfrak{B}_t = \sqrt{2\vartheta \mathfrak{B}_t(1-\mathfrak{B}_t)} \, {\rm d}\mathfrak{W}_t \,,
$$
where $(\mathfrak{W}_t)_{t \geq 0}$ is the standard Brownian motion and $\vartheta$ is a positive real number referred to as the \emph{diffusion constant}. This idea has been made rigorous for the complete graph \cite{D08} (where $\vartheta = 1$), and for the $d$-dimensional torus \cite{C89} (where $\vartheta$ is related to the Green function of simple random walk on the infinite lattice). More recently, in \cite{CCC16} (see also \cite{O13}) sufficient conditions are provided that ensure the above convergence, which is part of a more general program initiated in \cite{CG90} referred to as the \emph{finite-systems scheme}. In particular, the approach developed in \cite{CCC16} is capable of predicting that the limiting diffusion constant must be given by a proper rescaling of the expected meeting time of two stationary independent simple random walks evolving on the same underlying random graph. Exploiting this perspective, the limiting diffusion constant has been successfully identified for a number of \emph{static} random graph models \cite{CFR10,C21,ABHdHQ24,ACHQpr,QS23}.  

%%%

\subsection{Voter model on dynamic random graphs}

In the present paper, we consider the two-opinion voter model evolving on a \emph{dynamic random regular graph}, where the dynamics arises from \emph{random rewiring}. More precisely, our (autonomous) graph dynamics is initialised with a random regular graph and evolves in a stationary way by \emph{randomly rewiring} edges at exponential times. More precisely, each stub (half-edge) is equipped with a Poisson process and, at each arrival time, selects another stub uniformly at random and matches with it, i.e., the two half-edges are linked together to form an edge. A parameter $\nu \in (0,\infty)$ controls the rate of the Poisson process, slowing down or speeding up the graph dynamics, with $\nu = 0$ corresponding to the static graph. This model of a dynamic random graph has been used in \cite{SOV21pr,SV23pr} as the underlying geometry for the contact process. 

The goal of the present paper is to extend the work in \cite{CCC16} from a static to a dynamic environment, thereby proving pathwise convergence of the opinion densities to a Fisher-Wright diffusion. Moreover, by studying the meeting time of two independent random walks on our dynamic random graph, we identify the diffusion constant, $\vartheta_{d,\nu}$ as a function of the degree $d$ and the rewiring rate $\nu$. In particular, we provide an explicit expression for the diffusion constant in terms of a {\em continued-fraction expansion}, from which its finer properties can be deduced.  As a byproduct of the analysis of this expansion, we derive relevant takeaways for applications: To what extent does the rewiring dynamics enhance the speed of consensus? What is the interplay between the speed of the dynamics and the degree? Moreover, we recover the mean-field setting in the limit as $\nu \to \infty$ or $d \to \infty$, and the static setting in the limit as $\nu \downarrow 0$.  

To the best of our knowledge, the results presented in this paper provide the first instance of a \emph{voter model in a dynamic random environment} for which the diffusion constant is computed explicitly (in accordance with the program set out in \cite{AF02,D10,A12}). Our results may also serve as a benchmark for different graph models.  

\subsection{Organisation of the paper}  
The precise statement of our main result, Theorem \ref{th:CCC}, is postponed to Section \ref{sec:notation}, following a detailed description of the model. In the same section, we also provide a thorough account of the proof strategy. Section \ref{sec:tools} presents additional preliminaries, while the core of the paper is covered in Sections \ref{sec:RW}--\ref{sec:theta}. A more detailed road map of the paper can be found in Section \ref{sec:outline}.

%%%%%%%%%%%% SECTION 2 %%%%%%%%%%%%%%%%%%%%%%%%%%%%%%%%%
		
\section{Notation and results}
\label{sec:notation}

We first provide a formal definition of the dynamic random environment and the random process therein that we analyse in this paper. In Section \ref{sec:grdyn} we recall the definition of the random walk and the voter model on static random regular graphs. In Section \ref{suse:dyn} we introduce the rewiring dynamics of the graph. In Section \ref{suse:RW-vot-dyn} we define the model of interest, i.e., the voter model on rewiring random regular graphs. Section \ref{sec:voter} is devoted to presenting the main result of this paper (Theorem \ref{th:CCC}), which builds on two other results of independent interest (Theorems \ref{th:meeting} and \ref{thm:disc}). A detailed discussion of the latter two results, which constitute the major novelty of this work, and of the derivation of Theorem \ref{th:CCC} is postponed to Section \ref{sec:explanation}.

%%%
	
\subsection{Random walks and the voter model on static random graphs}
\label{sec:grdyn}
	
Fix $n \in \N$ and $d \geq 3$ such that $dn$ is even. Let $[n]$ represent the set of (labeled) vertices of our graph. Attach $d$ {\em stubs} (or {\rm half-edges}) to each $x \in [n]$, which will be labeled as $\sigma_{x,1}, \dots, \sigma_{x,d}$. In this language, a \emph{graph} $G$ is a matching on the set of stubs $\cup_{1 \leq i \leq d}\cup_{x \in [n]} \{ \sigma_{x,i} \}$, and we call $\mathcal{G}_n(d)$ the set of all possible graphs. We call an \emph{edge} of $G$ a matched pair of stubs and write $x \sim_G y$ to mean that there exists an edge (possibly more) between $x$ and $y$. More precisely, it will be convenient to write $\sigma_{x,i} \leftrightarrow_G \sigma_{y,j}$ to mean that $\sigma_{x,i}$ and $\sigma_{y,j}$ are matched in $G$, and ${\rm v}_G(x,i) \in [n]$ to denote the vertex to which the stub matched in $G$ to $\sigma_{x,i}$ belongs. We use the expression ``$G$ is sampled according to the \emph{Configuration Model}'' to mean that $G$ is obtained by taking a matching uniformly at random on the set of stubs, i.e., $G$ is uniformly distributed on $\mathcal{G}_n(d)$, and we formally write it as $G=^{(\rm d)} \mu_{d}^{(n)}$. 
	
We will also consider the simple random walk evolving on $G \in \cG_n(d)$, namely, the continuous-time Markov process on $[n]$ with generator 
	\begin{equation}
	\label{RWgenerator}
	(L_{G}^{\rm RW}f)(x) \coloneqq \frac{1}{d} \sum_{1 \leq i \leq d} [f({\rm v}_G(x,i))-f(x)], \qquad f\colon\,[n] \to \R.
	\end{equation}
It is worth pointing out that the uniform distribution on $[n]$, which will be denoted by $\pi^{(n)}$, is stationary for the process $L_G^{\rm RW}$ regardless of the choice of $G \in \cG_n(d)$. 
	
The voter model on a given $G \in \cG_n(d)$ is the continuous-time Markov process $(\eta_t)_{t \geq 0}$ on $\{0,1\}^n$ with generator
	\begin{equation}
	\label{VoterGen}
	(L^{\rm voter}_G f)(\eta) \coloneqq \sum_{x \in [n]} \frac{1}{d} \sum_{1 \leq i \leq d} 
	[f(\eta^{x,{\rm v}_{G}(x,i)})-f(\eta)], \qquad f\colon\, \{0,1\}^n \to \R,
	\end{equation}
where $\eta^{x,y}$ is the configuration obtained from $\eta$ by setting
	\begin{equation}
	\label{eq:voter-step}
	\eta^{x,y}(z) \coloneqq \begin{cases}
	\eta(y), & \text{ if } z=x, \\
	\eta(z), & \text{ otherwise.}
	\end{cases}
	\end{equation}
We will write ``$x$ has opinion 0 (respectively, 1) at time $t$'' when $\eta_t(x)=0$ (respectively, $\eta_t(x)=1$). Define $ \mathbf{0}$ (respectively,  $\mathbf{1}$) as the opinion configuration in which every vertex has opinion 0 (respectively, 1), and 
	\begin{equation}
	\label{def-tau-cons}
	\tau_{\rm cons} \coloneqq \inf\{t \geq 0 \colon \eta_t = \mathbf{0} \text{ or } \eta_t = \mathbf{1}\} 
	\end{equation}
to be the {\em consensus time}, i.e., the first time when all the vertices agree on the same opinion. Note that, as soon as $G$ is connected and regardless of the initial configuration, $\tau_{\rm cons} < \infty$ with probability one.
	
Often, to lighten notation, we suppress the dependence on $n$ and write $\cG(d)$ in place of $\cG_n(d)$, $\mu_d$ in place of $\mu_{d}^{(n)}$, and $\pi$ in place of $\pi^{(n)}$.

%%%
	
\subsection{Graph dynamics}
\label{suse:dyn}
	
Next consider a dynamic graph, i.e., a continuous-time Markov process $(G_t)_{t \geq 0}$ on $\cG_n(d)$. Starting at an initial graph $G_0 \in \mathcal{G}_n(d)$, we consider a \emph{rewiring dynamics} in which pairs of edges are rewired, i.e., broken into four stubs that are subsequently reconnected. More precisely, each stub decides to rewire at rate $\frac{\nu}{4}$ in order to get matched to another stub uniformly at random, and the two stubs  that are left alone by this procedure will then be matched among themselves. With this choice of parametrisation, a given edge of the graph has a survival time which is exponential of rate $\nu \frac{dn-2}{dn-1} \sim \nu$. Indeed, an edge will be rewired as soon as one of the two stubs which form the edge decides to rewire with another stub, or when another stub in the graph decides to rewire with one of the stubs forming the edge.
	
We will use the shortcuts $x \sim_t y$, $\sigma_{x,i} \leftrightarrow_t \sigma_{y,j}$ and ${\rm v}_t(x,i)$ to intend  $x \sim_{G_t} y$, $\sigma_{x,i} \leftrightarrow_{G_t} \sigma_{y,j}$ and ${\rm v}_{G_t}(x,i)$, respectively. Formally, $(G_t)_{t \geq 0}$ is the Markov process with generator $L^{\rm dyn}_{d,\nu}$  acting on functions $f: \cG_n(d) \to \R$ as
	\begin{equation}
	\label{GenGraphDyn}
	(L^{\rm dyn}_{d,\nu} f)(G) \coloneqq \sum_{\substack{x \in [n]}}\sum_{1 \leq i \leq d}
	\frac{\nu}{4}\frac1{dn-1} \sum_{y \in [n]}\sum_{1 \leq j \leq d}  \ind_{\sigma_{x,i} \neq \sigma_{y,j}}
	\ind_{\sigma_{x,i} {\not\leftrightarrow}_G \sigma_{y,j}} [f(\bar G^{x,y}_{i,j}) - f(G)],
	\end{equation}
where, for any $G \in \cG_n(d)$, $x, y\in [n]$ and $1 \leq i,j \leq d$ such that $\sigma_{x,i} \neq \sigma_{y,j}$ and $\sigma_{x,i}{\not\leftrightarrow}_G \sigma_{y,j}$, $\sigma_{z,k}$ and $\sigma_{v,\ell}$ are the stubs matched to $\sigma_{x,i}$ and $\sigma_{y,j}$ in $G$, and $\bar{G}^{x,y}_{i,j}$ is the graph obtained from $G$ by replacing the matchings $\sigma_{x,i} \leftrightarrow \sigma_{z,k}$ and $\sigma_{y,j} \leftrightarrow \sigma_{v,\ell}$ with  $\sigma_{x,i} \leftrightarrow \sigma_{y,j}$ and $\sigma_{v,\ell} \leftrightarrow \sigma_{z,k}$. An explicit \emph{graphical construction} of the process $(G_t)_{t \geq 0}$ will be given in Section \ref{sec:tools}.
	
It is worth noting that the measure $\mu_d$ is stationary (actually, reversible) for the rewiring dynamics, in the sense that $G_t =^{(\rm d)} \mu_{d}$ for all $t \geq 0$ as soon as $G_0 =^{(\rm d)} \mu_{d}$.

%%%
	
\subsection{Random walks and the voter model on dynamic random graphs}
\label{suse:RW-vot-dyn}
	
In analogy with the static case, we now give a formal description of the simple random walk and of the voter model on our dynamic graph set-up. By \emph{(simple) random walk on the dynamic graph} we mean the continuous-time Markov process $(G_t,X_t)_{t \geq 0}$ on $\cG_n(d) \times [n]$ with generator $L^{\rm dRW}_{d,\nu}$ acting on functions $f\colon \mathcal{G}_{n}(d) \times [n] \to \R$ as
	\begin{equation}
	\label{GenDynRW}
	(L^{\rm dRW}_{d,\nu} f)(G,x) \coloneqq (L^{\rm dyn}_{d,\nu} f(\cdot,x))(G) + (L^{\rm RW}_{G}f(G,\cdot))(x).
	\end{equation}
Note that the process above is reversible with respect to the product measure $\mu_{d} \otimes \pi$. We will often consider two independent random walks on the same underlying dynamic graph, i.e., the Markov process $(G_t, X_t, Y_t)_{t \geq 0}$ on $\cG_n(d) \times [n]^2$ with generator $L^{\rm dRW(2)}_{d,\nu}$ acting on functions $g \colon \mathcal{G}_{n}(d) \times [n]^2 \to \R$ as
	\begin{equation}
	(L^{\rm dRW(2)}_{d,\nu} g)(G,x,y) \coloneqq (L^{\rm dyn}_{d,\nu} g(\cdot,x,y))(G) + (L^{\rm RW}_{G}g(G,\cdot,y))(x)
	+ (L^{\rm RW}_{G}g(G,x,\cdot))(y).
	\end{equation}
Once again, this process is reversible with respect to $\mu_{d}\otimes \pi \otimes \pi$.
	
The main object in the present paper is the \emph{voter model on the dynamic graph $(G_t)_{t \geq 0}$}, namely, the continuous-time Markov process $(G_t,\eta_t)_{t \geq 0}$ on $\cG_n(d) \times \{0,1\}^n$ with generator $L_{d,\nu}^{\rm voter}$ acting on functions $h \colon \mathcal{G}_{n}(d) \times \{0,1\}^{n} \to \R$ as
	\begin{equation}
	\label{GenDynVot}
	(L_{d,\nu}^{\rm voter} f)(G,\eta) \coloneqq (L^{\rm dyn}_{d,\nu} f(\cdot,\eta))(G) + (L^{\rm voter}_{G}f(G,\cdot))(\eta).
	\end{equation}
For any $G \in \cG_n(d)$ and $\xi \in \{0,1\}^n$,  denote by $\P_{G,\xi} = \P^{(n)}_{G,\xi}$ the law of $(G_t, \eta_t)_{t \geq 0}$ starting from $(G,\xi)$. When the initial distribution is of product form with $G_0 =^{(\rm d)} \mu_d$ and $\eta_0 =^{(\rm d)} \otimes_{x\in[n]}{\rm Bernoulli}(u)$ for some $u \in [0,1]$, we use the shortcut $\P_{\mu_d, u}$ to refer to the law of the joint process with such an initial condition. Also in the dynamic set-up we define the \emph{consensus time} as in \eqref{def-tau-cons}. Clearly, $(G_t,\eta_t)_{t \geq 0}$ has two absorbing sets, $\cG_n(d) \times \{\mathbf{0} \}$ and $\cG_n(d) \times  \{\mathbf{1} \}$, and the process will be absorbed in one of them in finite time, i.e., $\P_{\mu_d,u}(\tau_{\rm cons} < \infty) =1$ for all $u \in [0,1]$.
	
%%%
	
\subsection{Main results} 
\label{sec:voter}
	
Throughout the sequel we assume $d \geq 3$ and $\nu \in (0,\infty)$ are fixed, and we are interested in deriving asymptotic results in the regime $n \to \infty$. Before presenting our results it is worth defining the following object, which plays a major role in our work.

\begin{definition}{\rm \bf [The diffusion constant $\vartheta_{d,\nu}$]}
	\label{def:key}
	For $d\ge 2$ and $\nu\in[0,\infty)$, define
	\begin{equation}
		\label{def-theta}
		\vartheta_{d,\nu}\coloneqq1-\frac{\Delta_{d,\nu}}{\beta_d},
	\end{equation} 
	with
	\begin{equation}
		\label{def-beta}
		\beta_d \coloneqq \sqrt{d-1},
	\end{equation}
	$\Delta_{d,\nu}$ the inhomogeneous continued fraction (written in the Pringsheim notation)
	\begin{equation}
		\label{def-Delta}
		\Delta_{d,\nu} \coloneqq \frac{1|}{|\frac{2+\nu}{\rho_d}} - \frac{1|}{|\frac{2+2\nu}{\rho_d}} 
		- \frac{1|}{|\frac{2+3\nu}{\rho_d}} - \dots,
	\end{equation}
	and 
	\begin{equation}
		\label{def-rho}
		\rho_{d} \coloneqq \frac{2\sqrt{d-1}}{d}.
	\end{equation}
\end{definition}

\begin{remark}
	\label{cfconv}
	According to the Pringsheim criterion \cite[p.254]{P1913}, for all $d \ge 2$ and $\nu\in[0,\infty)$ the continued fraction $\Delta_{d,\nu}$ in \eqref{def-Delta} is convergent because $\frac{2+\nu i}{\rho_d} \geq 2$ for all $i\in\N$, and the convergence is uniform in $\nu$. Moreover, $\nu \mapsto \Delta_{d,\nu}$ is analytic on $[0,\infty)$. See Figure~\ref{fig:theta} for a numerical computation.
\end{remark}

\begin{remark}
	It is worth noting that $\rho_d$ coincides with the smallest non-zero eigenvalue of the generator of a random walk on the infinite $d$-regular tree $\bbT_d$. We also recall that, for a static random graph $G =^{\rm (d)}\mu_d$, w.h.p.\ the smallest eigenvalue of the generator $L^{\rm RW}_G$ (defined in \eqref{RWgenerator}) is $\rho_d+o(1)$ (see e.g.\ \cite{F03,B20}).
\end{remark}

%%%%%%%%%% FIGURE %%%%%%%%%%%%%%%%%%%%%%%%%%%%%%%%%
\begin{figure}[htbp]
	\centering
	\includegraphics[width=4.5cm]{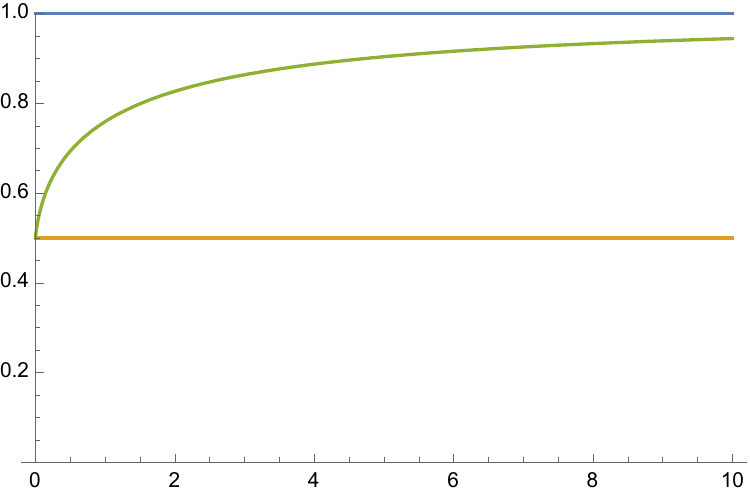} \quad
	\includegraphics[width=4.5cm]{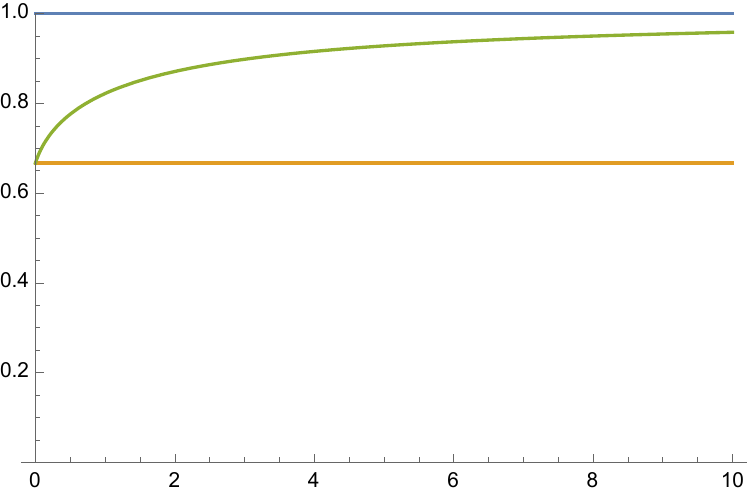} \quad  	
	\includegraphics[width=4.5cm]{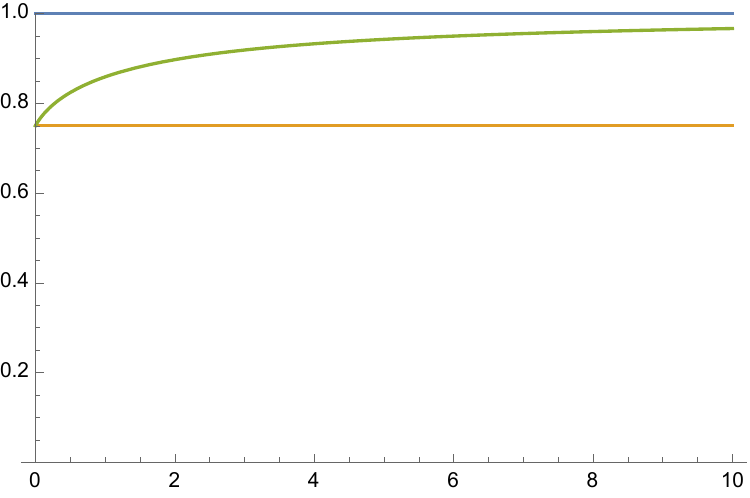}
	\caption{Numerical computation of $\nu\mapsto\vartheta_{d,\nu}$ for $d=3,4,5$ (from left to right). The blue line has height $1$, the orange line has height $\vartheta_{d,0}=(d-2)/(d-1)$. The green line is the numerical approximation of $\nu\mapsto \vartheta_{d,\nu}$ obtained after retaining 10 terms in the continued fraction.}
	\label{fig:theta}
\end{figure}
%%%%%%%%%%%%%%%%%%%%%%%%%%%%%%%%%%%%%%%%%%%

\subsubsection{Meeting times of stationary random walks}
\label{sec:rw}
As pointed out in the introduction, thanks to duality, the key ingredient to understand the  scaling limit of the voter model on a certain environment is a control on the meeting time of two independent stationary walks. In our case, the latter amounts to considering the Markov process $(G_t,X_t,Y_t)_{t \geq 0}$ on $\cG_n(d) \times [n]^2$ with $(G_0,X_0,Y_0) =^{(\rm d)} \mu_{d} \otimes \pi \otimes \pi$, and examining $\tau_{{\rm meet}}^{\pi\otimes\pi}$, the first time $t \geq 0$ such that $X_t=Y_t$.

The main technical contribution of the present paper lies in showing that, in the limit as $n$ grows with high probability, $\tau_{\rm meet}^{\pi\otimes\pi}$ is well approximated by an exponential random variable whose mean scales as $n/2\vartheta_{d,\nu}$. See Figure~\ref{fig:exp} for a simulation.

%%%%%%%%%%% FIGURE %%%%%%%%%%%%%%%%%%%%%%%%%
\begin{figure}[htbp]
	\centering
	\includegraphics[width=7.5cm]{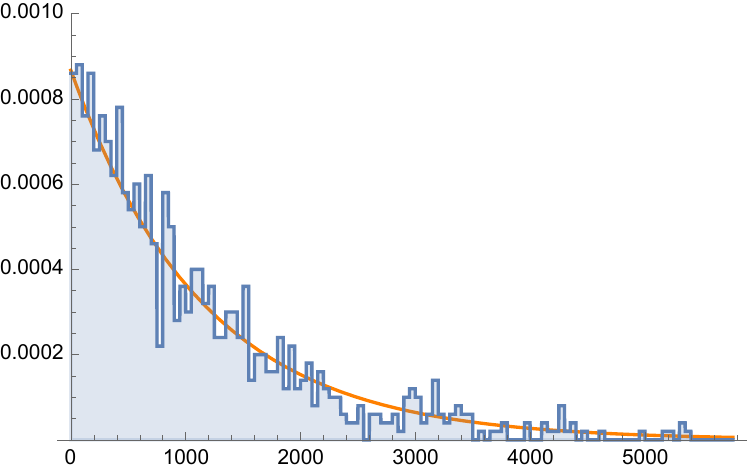}
	\caption{Simulation of the meeting time of two independent random walks in the case $d=3$ and with rewiring rate $\nu = 0.3$. The size of the graph is $n = 1500$. The empirical histogram in blue is obtained by running $1000$ independent simulations. The orange curve is the probability density function of ${\rm Exp}(2\vartheta_{3,0.3}/n)$.}
	\label{fig:exp}
\end{figure}
%%%%%%%%%%%%%%%%%%%%%%%%%%%%%%%%%%%%%%%%%%

\begin{theorem}{\rm \bf [Exponential limit law for $\tau_{\rm meet}^{\pi\otimes\pi}$]}
	\label{th:meeting}
	For every $s \geq 0$,
	\begin{equation}
		\label{bound-prob}
		\lim_{n\to\infty} \left| \P_{\mu_d}\left({\tau_{\rm meet}^{\pi\otimes\pi}} > sn\right) 
		- \ee^{-2\vartheta_{d,\nu} s} \right| = 0,
	\end{equation} 
	and 
	\begin{equation}
		\label{bound-exp}
		\lim_{n\to\infty} \frac{\E_{\mu_d}\left[\tau_{\rm meet}^{\pi\otimes\pi}\right]}{n} 
		= \frac{1}{2\vartheta_{d,\nu}}.
	\end{equation} 
\end{theorem}

The proof of Theorem~\ref{th:meeting} is articulated in two main steps. In Section~\ref{sec:RW} we introduce an idealised model in which two random walks evolve on fragmenting and coalescing infinite trees, and we prove the analogue of  Theorem \ref{th:meeting} in this simplified set-up. In Section \ref{sec:meeting} we show that, via a two-step approximation procedure, the original process can be coupled to the idealised model at a small total variation cost.

\subsubsection{Scaling limit of the voter model}

As mentioned in the Introduction, the aim of this work is to understand the long-time behaviour of the voter model in our dynamic random environment. More precisely, we are interested in the scaling limit of the following observable of the process with generator $L_{d,\nu}^{\rm voter}$ (defined in \eqref{GenDynVot}), 
	\begin{align}
	\label{def-O} \cO_t^{(n)} &= \cO_t \coloneqq \frac1n \sum_{x \in [n]} \eta_t(x) \in [0,1],
	\end{align}
i.e., the fraction of \emph{vertices with opinion} $1$. Our main result shows that, after rescaling time by a factor $n$, the quantity $\cO_t$ evolves according to a \emph{Fisher-Wright diffusion with diffusion constant} $\vartheta_{d,\nu}$ as in Definition \ref{def:key}.

\begin{theorem}{\rm \bf [Convergence to Fisher-Wright diffusion of the opinion density]}
\label{th:CCC}
For all $T\ge 0$,
	\begin{equation}
	(\cO_{sn})_{s \in [0,T]} \overset{\P^{(J_1)}}{\longrightarrow} (\mathfrak{B}_s)_{s\in[0,T]},
	\end{equation}
where $\overset{\P^{(J_1)}}{\longrightarrow}$ stands for convergence in distribution on path space in the Skorohod $J_1$-metric under the joint law $\P_{\mu_d,u}$ as $n \to \infty$, and
	\begin{equation}
	\label{FW}
	\begin{cases}
	\dd \mathfrak{B}_s = \sqrt{2\vartheta_{d,\nu} \mathfrak{B}_s(1-\mathfrak{B}_s) }\,
	\dd \mathfrak{W}_s,\\
	\mathfrak{B}_0=u,
	\end{cases}
	\end{equation}
with $(\mathfrak{W}_s)_{s \geq 0}$ denoting the standard Brownian motion, is the Fisher-Wright diffusion with diffusion constant $\vartheta_{d,\nu}$ identified in Definition \ref{def:key}.
\end{theorem}

\begin{remark}
Our result does not immediately imply the convergence in distribution of the consensus time to the absorption time of corresponding diffusion process. Indeed, hitting times are in general not continuous in the Skorohod $J_1$-topology. Nevertheless, in the spirit of \cite[Proposition 16]{CCC16} and \cite{O13}, it is natural to expect that weak convergence of the consensus time can be obtained by a sharpening of our techniques, leading in particular to the conclusion that
			\begin{equation}\label{eq:oli}
			\lim_{n\to\infty} \frac{\E_{\mu_d,u}\left[\tau_{\rm cons}\right]}{n} 
			= \frac{2H(u)}{\vartheta_{d,\nu}} \qquad \forall\, u \in (0,1),
			\end{equation} 
			with $H(u) = -(1-u)\log(1-u)-u\log u$ the entropy of the initial density. See \cite[Theorem 1.3]{O13} for a proof of \eqref{eq:oli} in the static setup, and \cite[Theorem 1.2]{O13} and \cite[Proposition 2.6]{CCC16} for the analogous result concerning a more general system of coalescing random walks. 
\end{remark}

%%%

\subsubsection{Analysis of the diffusion constant}
	
	Even though the Fisher-Wright diffusion is the natural scaling limit of our density process (as an expert reader could guess via a heuristic argument in the spirit of \cite{AF02,D08}), the identification of the diffusion constant $\vartheta_{d,\nu}$ is more challenging. A novel feature of our work lies in the characterisation of this constant as a function of the degree $d$ of the graph and the speed $\nu$ of the rewiring dynamics.
	
Given the explicit definition of the quantity $\vartheta_{d, \nu}$, it is natural to study its limit behaviour when the speed of the dynamics vanishes or grows to infinity, and its dependence on the degree of the graph. In particular, as the next result shows, we prove that, regardless of the degree, the effect of the rewiring is a \emph{speed up} of the consensus time. Symmetrically, regardless of the speed of the rewiring dynamics, there is a \emph{speed up} on the consensus time as the connectivity increases.

\begin{proposition}{\rm \bf [Limits of $\vartheta_{d,\nu}$]}
	\label{lem:limit-theta}
	The functions $\nu \mapsto \vartheta_{d,\nu}$ and  $d \mapsto \vartheta_{d,\nu}$ are strictly increasing.
	Moreover, for every $d \geq 3$,
	\begin{equation}
		\label{eq:limit-theta-a}
		\lim_{\nu \downarrow 0} \vartheta_{d,\nu} = \vartheta_{d,0} = \frac{d-2}{d-1},
		\quad \lim_{\nu\to \infty}\vartheta_{d,\nu} = 1,
	\end{equation}
	while for every $\nu>0$,
	\begin{equation}
		\label{eq:limit-theta-b}
		\lim_{d \to \infty} \vartheta_{d,\nu} = 1.
	\end{equation}
	The first limit corresponds to the static $d$-regular graph, the second and third limits to the mean-field setting.
\end{proposition}

To conclude this section, we look at the first-order behaviour of the diffusion constant for vanishing or diverging rewiring rate, respectively, diverging degree.

\begin{proposition}{\bf [Perturbations of $\vartheta_{d,\nu}$]}
	\label{pr:der-theta-0}
	For every $d \geq 3$, 
	\begin{equation}
		\label{eq:per-theta-a}
		\lim_{\nu \downarrow 0} \frac{1}{\nu} \big(\vartheta_{d,\nu} - \vartheta_{d,0}\big) = \frac{d}{2(d-2)^2},
		\quad \lim_{\nu \to \infty}{\nu} \big(1-\vartheta_{d,\nu}\big) = \frac{2}{d},
	\end{equation}
	while for every $\nu>0$,
	\begin{equation}
		\label{eq:per-theta-b}
		\lim_{d \to \infty}{d} \big(1-\vartheta_{d,\nu}\big) = \frac{2}{\nu+2}.
	\end{equation}
\end{proposition}

\noindent
See Remark \ref{rmk:heuristic} below for a heuristic explanation of the scalings in Proposition \ref{pr:der-theta-0}.

%%%

\subsection{Homogenisation of discordances}
\label{sec:explanation}

The key object to deduce the convergence in Theorem \ref{th:CCC} is an analysis of the time evolution of the \emph{density of discordances}, defined as
\begin{align}
	\label{def-D} \cD_t &= \cD_t^{(n)} \coloneqq \frac{2}{dn} \sum_{x \in [n]} \eta_{t}(x) 
	\sum_{1 \leq i \leq d} \big(1- \eta_{t}({\rm v}_t(x,i))\big) \in [0,1].
\end{align}
Indeed, in the static set-up the density of opinions $(\cO_t)_{t \geq 0}$ in \eqref{def-O} is well-known to be a martingale with predictable quadratic variation given by $(\frac{2}{n} \int_0^t \cD_s{\rm d}s)_{t \geq 0}$. As the next result shows, this is still true in our dynamic set-up.

\begin{lemma}{\rm \bf [Key martingales]}
	\label{prop:martingales}
	Let $(\mathcal{F}_t)_{t \geq 0}$ be the natural filtration of the joint process $(G_{s}, \eta_{s})_{s \in [0,t]}$. The following two processes are martingales with respect to this filtration:
	\begin{enumerate}
		\item $(\mathcal{O}_{t})_{t \geq 0}$.
		\item $\left(\mathcal{O}^2_{t} - \frac{1}{n} \int_{0}^{t} \mathcal{D}_{s} \, \dd s\right)_{t \geq 0}$.
	\end{enumerate}
\end{lemma}

\begin{proof}
	Consider the Dynkin martingale
	\begin{equation}
		M_{t}(f) \coloneqq f(G_{t}, \eta_{t}) - \int_{0}^{t} (L^{\rm voter}_{d,\nu}f)(G_{s}, \eta_{s}) \, \dd s
	\end{equation}
	for arbitrary $f \colon \mathcal{G}_{n}(d) \times \{0,1\}^{n} \to \R$. To check the two statements, it suffices to pick $f=h$ and $f=g$ with $h(G, \eta) = \frac{1}{n}\sum_{x \in [n]} \eta(x)$ and $g(G, \eta) = \big(\frac{1}{n} \sum_{x \in [n]} \eta(x) \big)^{2}$, and note that, by \eqref{GenDynVot},
	\begin{equation}
		(L^{\rm voter}_{d,\nu}h)(G,\eta) = 0, \qquad (L^{\rm voter}_{d,\nu}g)(G,\eta) 
		= \frac{1}{dn^2}\sum_{x\in[n]}\sum_{1\le i\le d}
		\ind_{\{\eta(x) \neq \eta({\rm v}_{G}(x,i))\}}.
	\end{equation}
\end{proof}

For the simplest case of mean-field interaction, i.e., when the underlying static graph is the complete graph, we have
\begin{equation}
	\cD_t = 2\,\frac{n-1}{n}\,\cO_t(1-\cO_t).
\end{equation}
Because the solution to the SDE in \eqref{FW} with $1$ in place of $\vartheta_{d, \nu}$ can be characterised as the unique continuous martingale having predictable quadratic variation
\begin{equation}
	\langle \mathfrak{B} \rangle_t = 2 \int_0^t \mathfrak{B}_s(1-\mathfrak{B}_s)\,{\rm d}s,
\end{equation}
it is not hard to deduce the classical convergence of the mean-field voter model to the standard Fisher-Wright diffusion when time is rescaled linearly in the volume of the graph.

In \cite{CCC16} it is shown that, beyond the complete graph, the latter condition continues to hold under certain assumptions on the random walk on the underlying graph, which can be essentially summarised by a fast-mixing condition and an anti-concentration property of the stationary distribution (see \cite[Theorem 2.2]{CCC16}). For obvious reasons the latter are referred to as \emph{mean-field conditions}. In particular, convergence to the standard Fisher-Wright diffusion emerges when time is rescaled by the expected meeting time of two independent random walks evolving on the same graph and initialised at stationarity.

In view of this situation, our approach in proving Theorem \ref{th:CCC} is based on a two-step argument. First. we identify the first-order approximation of the expected meeting time, leading to Theorem \ref{th:meeting}. Second, we generalise the results in \cite{CCC16} to our dynamic set-up. The first step is a major novelty in the present paper, and requires delicate coupling arguments as well as an analysis of certain recursive equations. We provide the proof of and the underlying heuristics behind Theorem \ref{th:meeting} in Section \ref{sec:RW}. As to the second step, we follow the strategy developed in \cite{CCC16}. The following \emph{convergence criterion}, which is stated in \cite[Theorem 2.1]{CCC16} for the static case, is based on the analogue of Lemma \ref{prop:martingales} and standard tightness criteria. The reader may check that the proof presented in \cite[Section 5]{CCC16} extends to our set-up without any change.

\begin{proposition}{\bf [Convergence criterion]}
	\label{prop:gammanWF-new}
	Fix $d \geq 3$, $\nu>0$ and $u \in (0,1)$, and assume that there exists a positive sequence $(\gamma_n)_{n \in \N}$ such that, for any $T >0$, 
	\begin{equation}
		\label{mfc-new}
		\frac{\gamma_{n}}{n} \int_{0}^{T} \mathcal{D}_{\gamma_{n}s} \, \dd s 
		- \int_{0}^{T} \mathcal{O}_{\gamma_{n}s} \big(1-\mathcal{O}_{\gamma_{n}s}\big) \, 
		\dd s \overset{\P}{\longrightarrow} 0,
	\end{equation}
	where $\overset{\P}{\longrightarrow}$ stands for convergence in probability under the measure $\P_{\mu_d,u}$. Then, for any $T>0$, the process $(\mathcal{O}_{\gamma_{n} t})_{t \in [0,T]}$ converges to the standard Fisher-Wright diffusion, i.e., the solution to the SDE in \eqref{FW} with $1$ in place of $2\vartheta_{d,\nu}$.
\end{proposition}

In light of the above discussion, our main task for the second step, spelled out in detail in Section \ref{sec:CCC}, amounts to proving the following \emph{homogenisation} property for the density of discordant edges of the voter model under the edge-rewiring dynamics.

\begin{theorem}\label{thm:disc}{\bf [Homogenisation of the discordances]}
	Fix $d \geq 3$, $\nu>0$ and $u \in (0,1)$. For any $T >0$, the convergence in \eqref{mfc-new} holds for the choice $\gamma_n = \frac{n}{2\vartheta_{d, \nu}}$.
\end{theorem}

\noindent
As will become clear in the proof in Section \ref{sec:CCC}, to derive the above homogenisation property we must adapt the arguments in \cite{CCC16} to the dynamic set-up. 

Theorem \ref{th:CCC} is a straightforward consequence of Theorem \ref{thm:disc}.

\begin{proof}[Proof of Theorem~\ref{th:CCC}]
	Let $\gamma_n=\frac{n}{2\vartheta_{d, \nu}}$. Proposition~\ref{prop:gammanWF-new} and Theorem \ref{thm:disc} imply that $(\mathcal{O}_{\gamma_{n} t})_{t \in [0,T]}$ converges in distribution to the standard Fisher-Wright diffusion $(\bar{\mathfrak{B}}_{t})_{t \in [0,T]}$, the solution of the stochastic differential equation
	\begin{equation}
		\begin{cases}
			\dd \bar{\mathfrak{B}}_s = \sqrt{ \bar{\mathfrak{B}}_s(1-\bar{\mathfrak{B}}_s) }\, \dd \mathfrak{W}_s, \\
			\bar{\mathfrak{B}}_0=u.
		\end{cases}
	\end{equation}
	It therefore suffices to perform two time changes. First, $(\mathcal{O}_{n t})_{t \in [0,T]}$ converges in distribution to $(\bar{\mathfrak{B}}_{2\vartheta_{d,\nu} t})_{t \in [0,T]}$, by the very definition of the scale parameter $\gamma_{n}$. Second, $(\bar{\mathfrak{B}}_{2\vartheta_{d,\nu} t})_{t \in [0,T]}$ has the same distribution as $(\mathfrak{B}_{t})_{t \in [0,T]}$, which is the solution of \eqref{FW}.
\end{proof}

\begin{remark}
(1) In \cite{ABHdHQ24} we considered the model \emph{without rewiring} and identified the behaviour of $\cD_t$ also on moderate time scales, i.e., when $1\ll t\ll n$. In particular, we showed that, when the initial density of opinions is $u$, $\cD_t$ is well approximated (in a strong sense) by the deterministic quantity $2u(1-u)\vartheta_{d,0}$ (see \cite[Theorem 1.3]{ABHdHQ24}). In Proposition \ref{separate2-new} we show that, for $1 \ll t \ll n$, $\E[\cD_t] = (1+o(1))\,2u(1-u)\vartheta_{d,\nu}$. We expect that a similar control on $\E[\cD_t^2]$ can be obtained by extending the coupling arguments developed in Section \ref{sec:meeting}, implying a concentration property.\\
(2) We expect that on time scale $n$ the result in \eqref{mfc-new} can be extended to convergence of $(\mathcal{D}_{ns})_{s \in [0,T]}$ to $(2\mathcal{O}_{ns}(1-\mathcal{O}_{ns}))_{s \in [0,T]} $ in the Skorohod topology. A proof of this would require a refinement of the coupling argument in Section~\ref{sec:meeting}.
\end{remark}

%%%

\subsection{Organization of the proof}
\label{sec:outline}

The remainder of this paper is organised as follows. In Section~\ref{sec:tools} we formally introduce the probability space under investigation, presenting the graphical construction of the joint process and a number of basic tools, including duality of the voter model with a system of coalescing random walks. In Section~\ref{sec:RW} we consider two \emph{idealized models}, in which two random walks evolve on certain coalescing and fragmenting trees, which turn out to be the key approximation tools. In Section~\ref{sec:meeting} we couple the idealised models with the actual random walks evolving on the finite dynamic graph, leading to the proof of the exponential law of their meeting time in Theorem \ref{th:meeting}. Section~\ref{sec:CCC} is devoted to the proof of Theorem \ref{thm:disc}, which shows the {\em convergence criterion} for the Fisher-Wright  approximation is fulfilled in the case of random regular graphs with rewiring. The results presented in this section have been developed in a static set-up in \cite{CCC16}, so our main technical contribution in here is to show how to lift the recipe of \cite{CCC16} to our edge-dynamic setting. Finally, in Section~\ref{sec:theta} we provide a complete analysis of the diffusion constant, thereby settling Propositions \ref{lem:limit-theta}--\ref{pr:der-theta-0}.

%%%%%%%%%%%% SECTION 3 %%%%%%%%%%%%%%%%%%%%%%%%%%
	
\section{Duality and meeting times}
\label{sec:tools}

%%%
	
\subsection{Duality and graphical construction}
\label{ss.notprop}
	
It will be convenient to consider the following construction of the probability space we are interested in, which is usually referred to as the \emph{graphical construction}:
	\begin{itemize}
	\item[(1)] 
	Attach to each stub $\sigma_{x,i}$, $1 \leq i \leq d$, of each vertex $x \in [n]$ a Poisson process of rate $1/d$, and call $(\mathfrak{T}^{\rm rw}_{x,i})_{x \in [n], 1 \leq i \leq d}$ the collection of these processes.
	\item[(2)] 
	Attach to each  stub $\sigma_{x,i}$, $1 \leq i \leq d$, of each vertex $x \in [n]$ a Poisson process of rate $\frac{\nu}{4}$, and mark each arrival of such a process with a uniformly chosen independent stub $\sigma_{y,j}$, $y \in [n]$, $1 \leq j \leq d$, with $(x,i) \neq (y,j)$. Call $(\mathfrak{T}^{\rm dyn}_{x,i})_{x \in [n], 1 \leq i \leq d}$ the collection of these marked processes.
	\end{itemize}
We further assume that all the Poisson processes above are independent. With this source of randomness the graph dynamics, starting from an initial graph $G=G_0$, can be obtained as follows:
	\begin{itemize}
	\item 
	Let $t>0$ be the first arrival among the Poisson processes in (2). Assume that this time arrives on the process associated to $\sigma_{x,i}$ and is marked by $\sigma_{y, j}$ (with $(x,i) \neq (y,j)$).
	\begin{itemize}
	\item
	If $\sigma_{x,i}$ and $\sigma_{y, j}$ are matched in $G_0$, then nothing changes: set $G_s=G_0$ for all $s \in [0,t]$.
	\item 
	If $\sigma_{x,i}$ and $\sigma_{y, j}$ are not matched in $G_0$, then call $\sigma_{z,k}$ and $\sigma_{v,\ell}$ the stubs matched to $\sigma_{x,i}$ and $\sigma_{y,j}$, respectively, and:
	\begin{itemize}
	\item 
	let $G_s=G_0$ for all $0 < s < t$;
	\item 
	let $G_t$ be obtained from $G_0$ by erasing the edges $\sigma_{x,i} \leftrightarrow \sigma_{z,k}$  and $\sigma_{y,j} \leftrightarrow \sigma_{v,\ell}$ and creating the edges $\sigma_{x,i} \leftrightarrow \sigma_{y,j}$ and $\sigma_{z,k} \leftrightarrow \sigma_{v,\ell}$.
	\end{itemize}
	\end{itemize}
	\end{itemize}
For any initial configuration, the voter dynamics on the underlying dynamic graph, i.e., $(G_t, \eta_t)_{t \geq 0}$, can be obtained by means of the same source of randomness as follows:
	\begin{itemize}
	\item 
	Consider the rewiring  dynamics $(G_t)_{t \geq 0}$ constructed above starting at some initial graph $G_0$, and let $\eta_0 = \xi$ for some $\xi \in \{0,1\}^n$. 
	\item  
	Let $t>0$ be the first arrival among the Poisson processes in (1). Assume that this time arrives on the process associated to $\sigma_{x,i}$, let $\eta_s = \eta_0$ for all $0 < s < t$, and define $\eta_t = \eta_0^{x,{{\rm v}_t(x,i)}}$ (recall \eqref{eq:voter-step}).
	\end{itemize}
Now, fix an initial graph $G_0 = G$ and a time horizon $t>0$. We will consider two coalescing random walks evolving on $(G_s)_{0 \leq s \leq t}$ \emph{backward in time}. More precisely, let $x,y$ be two vertices and consider a process $(\hat X^x_{s,t}, \hat X_{s,t}^y)_{s \in [0,t]}$ on $[n]^2$ constructed as follows:
	\begin{itemize}
	\item[(a)] 
	Set $\hat X^x_{0,t} = x$ and $\hat X^y_{0,t} = y$. 
	\item[(b)] 
	Look at the \emph{last} arrival time $r$ in the interval $[0,t]$ among the processes defined in (1). Say that this corresponds to an arrival for the process associated to $\sigma_{z,i}$ for some $z \in [n]$ and $1 \leq i \leq d$. Set $(\hat X_{s,t}^x, \hat X_{s,t}^y)=(x,y)$ for all $s \in [t-r,t)$. Moreover, 
	\begin{itemize}
	\item
	if $z \notin \{x,y\}$, then set $\hat X^x_{t-r,t} = x$ and $\hat X^y_{t-r,t} = y$;
	\item 
	if $z \in \{x,y\}$, then set $\hat X^z_{t-r,t}$ to be the vertex connected to $x$ (respectively, $y$) through $\sigma_{z,i}$ in $G_r$, i.e., $\hat X_{t-r,t}^x = {\rm v}_r(\sigma_{z,i})$.
	\end{itemize}
	\end{itemize} 

The above construction can be extended to $n$ coalescing random walks $(\hat X^x_{s,t})_{s \in [0,t], x \in [n]}$ running backwards in time, each starting from a different vertex. Moreover, for any $t>0$ we can also consider the \emph{forward-in-time} coalescing random walks $(X_s^x)_{s \in [0,t], x \in [n]}$, which are defined in the same fashion as $(\hat X_{s,t}^x)_{s \in [0,t], x \in [n]}$, with the only difference that in (b) the word \emph{last} is replaced by the word \emph{first}. Note that when we are looking at \emph{first} arrivals such a forward-in-time process does not need to be defined on finite time intervals only, so that we can also consider $(X_s^x)_{s \geq 0,x \in [n]}$.
	
In what follows we employ the symbol $\P$ to refer to the probability space associated to the Poisson processes in (1) and (2), and write $\P_{G,\xi}$ to intend that the initial state of the process is $(G_0,\eta_0) = (G,\xi)$. Similarly, we will write $\P_{\mu_d,u}$ to intend the same probability space enriched with the (independent) randomness needed to construct the initial state $(G_0,\eta_0) =^{(\rm d)} \mu_d \otimes_{x \in [n]} {\rm Bernoulli}(u)$.
	
The graphical construction above reveals the following duality relation, which generalises the classical duality between the voter model and coalescent random walks on static graphs: for any initial configuration $(G,\xi) \in \cG_n(d) \times \{0,1\}^n$, any sets of vertices $A,A' \subset [n]$ and any time $t>0$,
	\begin{equation}
	\label{Duality}
	\prod_{x\in A} \eta_{t}(x) \prod_{y\in A'} (1-\eta_t(y))
	= \prod_{x\in A} \xi(\hat X_{t,t}^x) \prod_{y\in A'} (1-\xi(\hat X_{t,t}^y) ) \qquad \P_{G, \xi} \text{-a.s.}
	\end{equation}
For $x, y \in [n]$ and $t>0$, consider 
	\begin{equation}
	\label{def:hat-tau-meet}
	\hat \tau_{{\rm meet},t}^{x,y} \coloneqq \inf\{s \in [0,t] \mid \hat X_{s,t}^x = \hat X_{s,t}^y \},
	\end{equation}
and observe that \eqref{Duality} gives the relation
	\begin{equation}
	\label{Dual2}
	\P_{G,\xi} \big( \eta_{t}(x) \neq \eta_{t}(y) \big) 
	= \P_{G,\xi}\big( \hat \tau_{{\rm meet},t}^{x,y} > t, \xi(\hat X_{t,t} ^x)\neq\xi(\hat X_{t,t} ^y) \big),
	\end{equation}
which, when the system starts from i.i.d.\ Bernoulli opinions of parameter $u \in [0,1]$, reduces to 
	\begin{equation}
	\label{Dual3}
	\P_{G,u}\big( \eta_{t}(x) \neq \eta_{t}(y) \big) 
	= 2u(1-u)\, \P_{G} \big(\hat\tau_{{\rm meet},t}^{x,y} > t\big),
	\end{equation}
for any $x, y \in [n]$, $G \in \cG_n(d)$ and $u \in [0,1]$, where the dependence on $u$ in the probability on the right-hand side can be dropped due to the independence of the backward random walks and the initial configuration of the voter model. In particular, letting $G =^{(\rm d)} \mu_d$, we obtain
	\begin{equation}
	\label{Dual4}
	\P_{\mu_d,u} \big(\eta_{t}(x) \neq \eta_{t}(y)\big) 
	= 2u(1-u)\, \P_{\mu_d} \big(\hat\tau_{{\rm meet},t}^{x,y} > t\big).
	\end{equation}
At this point it is worth noting that, since the graph dynamics is reversible with respect to $\mu_{d}$ and the Poisson processes in (1) are i.i.d., we have, for any $t>0$ and $x, y \in [n]$,
	\begin{equation}
	\label{Dual5}
	\P_{\mu_d} \big(\hat\tau_{{\rm meet},t}^{x,y} >t\big) = \P_{\mu_d} \big(\tau_{\rm meet}^{x,y} > t\big),
	\end{equation}
where, similarly to \eqref{def:hat-tau-meet}, the meeting time $\tau_{\rm meet}^{x,y}$ is defined as 
	\begin{equation}
	\label{def-tau-meet}
	\tau_{\rm meet}^{x,y} \coloneqq \inf\{s > 0 \mid X_{s}^x = X_{s}^y \}.
	\end{equation}
Combining \eqref{Dual4} and \eqref{Dual5}, we get
	\begin{equation}
	\label{Dual-final}
	\P_{\mu_d,u} \big(\eta_{t}(x) \neq \eta_{t}(y) \big) 
	= 2u(1-u)\, \P_{\mu_d} \big(\tau_{\rm meet}^{x,y} > t\big).
	\end{equation}

%%%
		
\subsection{Consequences of duality}

The identity in \eqref{Dual-final} unveils the relation between discordances of opinions for the voter model and meeting times of associated random walks. We next show more explicitly how the expectation of the observables we are interested in (i.e., those in \eqref{def-O} and \eqref{def-D}) are linked to the meeting time. To do so, it is useful to generalise the definition of meeting time in \eqref{def-tau-meet} to allow for the two initial positions $x$ and $y$ to be random, possibly depending on the initial graph $G_0$. In particular, we will consider the product case $(X_0,Y_0) =^{(\rm d)} \pi \otimes \pi$, and the case in which the two random walks start at the extremes of an edge of $G_0$ sampled uniformly at random. In order to indicate the meeting time of the two random walks in the latter two scenarios, we will write $\tau_{\rm meet}^{\pi\otimes\pi}$ and $\tau_{\rm meet}^{\rm edge}$, respectively.
		
With this notation at hand, and with the duality in \eqref{Duality}, it is not hard to prove the following lemma.
\begin{lemma}
\label{lemma:conseq-duality}{\bf [Representation in terms of meeting time]}
Recall the definition of $(\cO_t)_{t \geq 0}$ and $(\cD_t)_{t \geq 0}$ in \eqref{def-O} and \eqref{def-D}, respectively. For any $t \geq 0$,
		\begin{equation}
		\label{exp-O-u}
		\E_{\mu_d,u}[\cO_t(1-\cO_t)] = u(1-u)\,\P_{\mu_d}(\tau^{\pi\otimes\pi}_{\rm meet}>t),
		\end{equation}
and
		\begin{equation}
		\label{exp-D-u}
		\E_{\mu_d,u}[\cD_t] = 2u(1-u)\,\P_{\mu_d}(\tau^{\rm edge}_{\rm meet}>t).
		\end{equation}
Moreover,
		\begin{equation}
		\label{exp-D-sup}
				\max_{\xi \in \{0,1\}^n} \E_{\mu_d,\xi}[\cD_t] \leq 2 \P_{\mu_d}(\tau^{\rm edge}_{\rm meet}>t).
		\end{equation}
\end{lemma}
		
\begin{proof}
We start by showing \eqref{exp-O-u}. By \eqref{Duality}, the definition of $\cO_t$, and the fact that the distribution of initial opinions is of product form, we can write
			\begin{equation}
			\begin{split}
			\E_{G,u}[\cO_t(1-\cO_t)]
			& = \sum_{x \in [n]} \sum_{y \in [n]} \frac{1}{n^2}
			\E_{G,u} [\eta_t(x)(1-\eta_t(y))] \\
			& = \sum_{x \in [n]} \sum_{y \in [n]} \frac{1}{n^2} \E_{G,u}[\xi(\hat X^x_{t,t})(1-\xi(\hat X_{t,t}^y))
			\ind_{\hat{\tau}_{{\rm meet},t}^{x,y}>t}] \\
			& = u(1-u) \sum_{x \in [n]} \sum_{y \in [n]\setminus\{x\}} \frac{1}{n^2}
			\P_{G}(\hat{\tau}_{{\rm meet},t}^{x,y}>t).
			\end{split}
			\end{equation}
Next, we exploit that the initial graph is distributed according to $\mu_d$ and use \eqref{Dual5}, to get
			\begin{equation}
			\begin{split}
			\E_{\mu_d,u}[\cO_t(1-\cO_t)]
			& = u(1-u) \sum_{x \in [n]} \sum_{y \in [n]\setminus\{x\}} \frac{1}{n^2}
			\P_{\mu_d}({\tau}_{\rm meet}^{x,y}>t)
			= u(1-u)\,\P_{\mu_d}({\tau}_{\rm meet}^{\pi\otimes\pi}>t),
			\end{split}
			\end{equation}
where for the last equality we use the definition of $\tau_{\rm meet}^{\pi\otimes\pi}$. To prove \eqref{exp-D-u} and \eqref{exp-D-sup}, we proceed similarly. By \eqref{def-D},
		\begin{equation}
		\label{d0-}
		\begin{split}
		\E_{\mu_d,\xi} [\cD_t] & = \frac{2}{dn} \sum_{G \in \cG(d)} \mu_d(G) \sum_{x \in [n]}
		\sum_{1 \leq i \leq d} \sum_{y \in [n]\setminus\{x\}} \sum_{1\leq j \leq d} \E_{G,\xi}
		\left[ \eta_t(x)(1-\eta_t(y))\ind_{\sigma_{x,i} \leftrightarrow_t \sigma_{y,j}} \right].
		\end{split}
		\end{equation}
Using \eqref{Duality}, we get that, for all $x,y \in [n]$, $G \in \cG(d)$, $\xi \in \{0,1\}^n$ and $t \geq 0$,
		\begin{equation}
		\label{d1}
		\begin{split}
		& \E_{G,\xi} \left[ \eta_t(x)(1-\eta_t(y))\ind_{\sigma_{x,i} \leftrightarrow_t \sigma_{y,j}} \right]
		= \E_{G,\xi} \left[ \xi(\hat X^x_{t,t})(1-\xi(\hat X^y_{t,t}))\ind_{\sigma_{x,i} \leftrightarrow_{t}
		\sigma_{y,j}} \ind_{\hat\tau^{x,y}_{{\rm meet},t}>t} \right].
		\end{split}
		\end{equation}
In particular, for any $\xi \in \{0,1\}^n$,
		\begin{equation}
		\label{d2}
		\begin{split}
		\E_{G,\xi} \left[ \eta_t(x)(1-\eta_t(y))\ind_{\sigma_{x,i} \leftrightarrow_t \sigma_{y,j}} \right]
		& \leq \E_{G} \left[ \ind_{\sigma_{x,i} \leftrightarrow_{t} \sigma_{y,j}}
		\ind_{\hat\tau^{x,y}_{{\rm meet},t}>t} \right],
		\end{split}
		\end{equation}
while for the product initial condition with parameter $u \in (0,1)$,
		\begin{equation}
		\label{d2bis}
		\begin{split}
		\E_{G, u} \left[ \eta_t(x)(1-\eta_t(y)) \ind_{\sigma_{x,i} \leftrightarrow_t \sigma_{y,j}} \right]
		& = u(1-u) \E_{G} \left[ \ind_{\sigma_{x,i} \leftrightarrow_{t} \sigma_{y,j}}
		\ind_{\hat\tau^{x,y}_{{\rm meet},t}>t} \right].
		\end{split}
		\end{equation}
A crucial observation is that, because the graph dynamics is reversible with respect to the uniform measure $\mu_d$ on $\cG(d)$, we have
	\begin{equation}
	\label{d3}
	\begin{split}
	\E_{G} \left[ \ind_{G_t = G'} \ind_{\sigma_{x,i} \leftrightarrow_{G'} \sigma_{y,j}}
	\ind_{\hat\tau^{x,y}_{{\rm meet},t}>t} \right] = \E_{G'}
	\left[ \ind_{G_t = G} \ind_{\sigma_{x,i} \leftrightarrow_{0} \sigma_{y,j}} \ind_{\tau^{x,y}_{\rm meet} > t} \right].
	\end{split}
	\end{equation}
Substituting \eqref{d1},~\eqref{d2bis} and~\eqref{d3} into~\eqref{d0-} with $\xi =^{\rm (d)} \otimes_{x \in [n]} {\rm Bern}(u)$, and using that $\mu_d$ is uniform over $\cG(d)$, we obtain
	\begin{equation}
	\label{d4}
	\begin{split}
	\E_{\mu_d,u} [\cD_t] & = \frac{2u(1-u)}{dn} \sum_{G, G' \in \cG(d)} \mu_d(G')
	\sum_{\substack{x,y \in [n] \\ y \neq x}} \sum_{1 \leq i,j \leq d} \E_{G'} \left[ \ind_{G_t = G}
	\ind_{\sigma_{x,i} \leftrightarrow_{0} \sigma_{y,j}} \ind_{\tau^{x,y}_{\rm meet}>t} \right] \\
	& = \frac{2u(1-u)}{dn} \sum_{\substack{x,y \in [n] \\ y \neq x}} \sum_{1 \leq i,j \leq d} \P_{\mu_d}
	\left( \{\sigma_{x,i} \leftrightarrow_{0} \sigma_{y,j}\} \cap \{\tau^{x,y}_{\rm meet}>t\} \right) \\
	& = 2u(1-u)\,\P_{\mu_d}(\tau_{\rm meet}^{\rm edge}>t),
	\end{split}
\end{equation}
where in the last equality we only use the definition of  $\tau_{\rm meet}^{\rm edge}$. This concludes the proof of \eqref{exp-D-u}. The proof of~\eqref{exp-D-sup} follows the same argument, after~\eqref{d2bis} is replaced by~\eqref{d2}.
\end{proof}

%%%%%%%%%%%% SECTION 4 %%%%%%%%%%%%%%%%%%%%%%%%%%
	
\section{Toy model: Meeting time on dynamic trees}
\label{sec:RW}
	
This section analyses an idealised model that approximates the meeting time of two stationary random walks on a dynamic random graph. We begin with a heuristic argument to guide the reader through Sections \ref{sec:RW}--\ref{sec:meeting}. The idealised model is introduced in Section~\ref{sec:RW}, while the coupling and its relation to the original model are discussed in Section \ref{sec:meeting}.

\medskip\noindent
{\bf Meeting time of random walks and hitting time of the diagonal set.} We begin by recalling the argument that was used to prove the analogue of Theorem \ref{th:meeting} in the static set-up. It is well known that a $d$-regular random graph \emph{locally} looks like a $d$-regular tree, in the sense that most of vertices have a neighbourhood of diverging radius that is a tree. This similarity is expedient to study observables of random walks such as mixing, hitting, meeting, and cover times, \cite{CF05,CFR10,LS10,ABHdHQ24}, as well as the existence of a phase transition for the contact process \cite{MV16}. 
	
In general, the meeting time of two random walks on a graph $G$ can be easily rephrased in terms of the hitting time of the diagonal set $\{(x,x) \colon x \in [n]\}$ for a random walk on the Cartesian product $G \times G$. For regular undirected graphs, the stationary distribution of the diagonal vanishes as the size of the graph increases, so that this set can be thought of as a \emph{small set of states}. Moreover, if the original graph locally resembles a transient graph, then the product graph, when viewed from the diagonal, also locally resembles a transient graph. Additionally, as a consequence of local transience, one can typically deduce a form of \emph{rapid mixing}, meaning that the expected hitting time of a small target set of states $\cX$ is much longer than the time it takes for the process to reach its equilibrium distribution $\pi$.

With the above in mind, the meeting time problem amounts to understanding the distribution of the hitting time of a small target set $\cX$ by a rapidly mixing stationary random walk. Due to the rapid mixing property, the hitting time of $\cX$ is well approximated by an exponential random variable. Furthermore, stationarity allows the hitting time to be viewed as the first success in a series of independent attempts, leading to exponential-like behaviour. This phenomenon has been rigorously studied under various names, such as Poisson clumping heuristic \cite{A13}, exponential approximation of hitting times \cite{A82,AB92,AB93}, First Visit Time Lemma \cite{CF05,MQS21}, and mean-field conditions \cite{O13,CCC16}. The rate of the exponential random variable is at most the stationary value of $\pi(\cX)$, discounted by the local time spent in $\cX$ before equilibrium. Given the local transience and the fact that the two random walks move independently, dividing this discount factor by two and taking the inverse we get the probability that a random walk starting at $\cX$ reaches equilibrium before returning to $\cX$.

%%%

\medskip\noindent
{\bf Meeting time on static regular random graphs.} To make the above heuristic concrete, we return to the static $d$-regular random graph. In this case, the stationary distribution of the diagonal set is $1/n$, while the mixing time of two random walks is of order $\log n$, which quantifies the ``rapid mixing" criterion mentioned above. Moreover, approximating the neighbourhood of a vertex by an infinite $d$-regular tree, we see that the probability that the process starting from the diagonal (i.e., two random walks starting at the same vertex) does not revisit the diagonal before mixing can be approximated by a one-dimensional problem. More precisely, for two random walks on a tree (both starting at the root), the distance between them evolves as a biased random walk on $\mathbb{N}_0$ (starting at $0$) that moves to the right at rate $2\frac{d-1}{d}$ and to the left at rate $\frac{2}{d}$. (When the biased random walk is at $0$, it moves to the right at rate $2$.) Therefore, we are left with computing the escape probability, i.e., the probability that such a random walk does not revisit $0$, which is known to be $\vartheta_{d,0} = \frac{d-2}{d-1}$ (which also coincides with the inverse of the expected local time at $0$, i.e., the Green's function). In \cite{ABHdHQ24, CFR10}, the analogue of Theorem \ref{th:meeting} for this setting is obtained by making the above heuristic rigorous.

%%%

\medskip\noindent
{\bf Meeting time on dynamic regular graphs and related toy models.}  Returning to our dynamic set-up, the (small) target set is $\cG_n(d) \times \{(x,x) \colon\, x \in [n]\}$ (with mass $1/n$ according to the stationary measure), and the process is rapidly mixing by Proposition \ref{pr:mixing} below. Although a slight modification of Proposition \ref{pr:mixing} is needed to account for the second random walk, this generalisation is straightforward and is not pursued here, since it is not needed for proving Theorem \ref{th:meeting}. The main requirements for our heuristic approach are satisfied, with the only remaining task being constructing a proper \emph{local approximation} of the dynamic environment. To this end, we proceed in stages, by considering two different toy models.

In the \emph{first toy model}, two random walks evolve independently on a $d$-regular tree, both starting at the root. To model the rewiring mechanism, we allow edges in the unique path (if any) connecting the two random walks to disappear at rate $\nu$. If any such edge disappears, then the two random walks will be doomed to stay apart forever. By removing only the edges in the path between the two random walks we model the fact that the rewiring of other edges does not affect their local mutual perspective. In Section \ref{sec:locdiag} we compute the expected time that the two random walks spend together before becoming permanently separated. If the heuristic described above is correct, this local time should provide a good estimate for the quantity $\vartheta_{d,\nu}$. This is indeed verified in Proposition \ref{lemma-R}, where we prove that the expected local time in the first model aligns with $\frac{1}{2\vartheta_{d,\nu}}$. However, rigorously establishing this fact in our dynamic set-up is far from straightforward.
	
In Section \ref{sec:meet} we consider the \emph{second toy model}, which mimics more directly the behaviour of the two random walks on the dynamic graph. We will consider two ``large'' $d$-regular trees, representing the graph as seen {\em locally} by the random walks (which are represented by the roots of the trees), and consider a process articulated into two phases: 
\begin{itemize}
\item
In the \emph{first phase} (the ``unmerged phase''), the trees are distinct and at exponential times pairs of edges (one in each of the two trees) are rewired, leading to a merging of the two trees. 
\item
In the \emph{second phase} (the ``merged phase''), the two random walks evolve independently on a single tree starting at a certain distance (depending on the pairs of edges that have been rewired during the first phase), but the edges along the unique path joining them disappear at a rate $\nu$: if the two random walks meet before the disappearance of the path joining them, then the process stops (representing the meeting of the two random walks); otherwise the process is reinitialised to the first phase. This phase can be coupled with the first toy model, which will provide a direct approximation tool.
\end{itemize}

At this stage, it should be clear that the second toy model offers a more direct approximation of the original process. Consequently, Theorem \ref{th:meeting} can be established by using Proposition \ref{lem:exponential-tau} below in conjunction with a coupling argument, which will be detailed in Section \ref{sec:meeting}. Finally, let us mention that Proposition \ref{lem:exponential-tau} below states that the meeting time in this idealised model is indeed exponentially distributed with rate $2\vartheta_{d,\nu}$, confirming the heuristic prediction within our dynamic framework.

%%%

\subsection{First toy model}
\label{sec:locdiag}
	
Consider two independent continuous-time random walks $X=(X_t)_{t \geq 0}$ and $Y=(Y_t)_{t \geq 0}$ on an infinite $d$-regular tree, both jumping along edges at rate $1$. Call $Z=(Z_t)_{t\ge 0}$ their distance process in the tree and note that this is a continuous-time Markov chain on $\N_0$ with transition rates
	\begin{equation}
	r_d(i,j) = \begin{cases}
	2\frac{d-1}{d}, & 0 < i = j-1, \\
	\frac{2}{d}, & i = j+1, \\
	2, & i=0 \text{ and } j=1, \\
	0, & \text{otherwise}.
	\end{cases}
	\end{equation} 
Suppose that the edges in the unique path in the tree joining $X_t$ and $Y_t$ (of length $Z_t$) \emph{disappear} at rate $\nu$ independently of each other. We consider the modified distance process $\hat Z = (\hat Z_t)_{t\ge 0}$ on $\N_0\cup\{\dagger\}$ with transition rates
	\begin{equation}
	\hat r_{d,\nu}(i,j) = \begin{cases}
	2\frac{d-1}{d}, & 0 < i = j-1, \\
	\frac{2}{d}, & i = j+1, \\
	2, & i=0\,,j=1\,, \\
	\nu i, & j=\dagger \text{ and } i \neq \dagger, \\
	0, & \text{otherwise}.
	\end{cases}
	\end{equation} 
In words, $\dagger$ represents the absorbing states of the modified distance process, representing the disappearance of an edge along the path joining the two random walks. Clearly the process $(\hat Z_t)_{t\ge 0}$ will be absorbed in $\dagger$ a.s.\ in a finite time, regardless of the initial distance $\hat Z_0$. Denote the law of this process by $\Pr_{d,\nu}$, and let
	\begin{equation}
	\label{def-Ri}
	R_{d,\nu}(i) \coloneqq \Ex_{d,\nu}\left[\int_0^\infty \ind_{\hat Z_t=0}\, \dd t\:\bigg\rvert\: \hat Z_0=i\right]
	\end{equation}
be the average total time $X$ and $Y$ spend together, also called the \emph{collision local time}. Finally, let
	\begin{equation}
	\label{def-q}
	q_{d,\nu}(i)\coloneqq\Pr_{d,\nu}(\exists\, t\ge 0\colon \hat Z_t=0\mid \hat Z_0=i)
	\end{equation}
be the probability that $X$ and $Y$ ever meet when starting at distance $i$.
	
\begin{proposition}{\rm \bf [Average collision local time]}
\label{lemma-R} 
For every $\nu>0$ and $d\ge 3$, 
		\begin{equation}
		R_{d,\nu}(0) = \frac{1}{2\vartheta_{d,\nu}}\,,
		\end{equation}
where $\vartheta_{d, \nu}$ is defined in \eqref{def-theta}.
\end{proposition} 
	
\begin{proof}
We will suppress the dependence on $d$ and $\nu$ to improve readability. By looking at the first transition of $\hat Z$, we get the recursion relations
		\begin{equation}
		\label{eq:recursion-R}   
		\begin{aligned}
		i \in \N \colon \,\, & R(i) = 2\,\frac{d-1}{d}\frac{1}{2+\nu i}\, R(i+1) 
		+ \frac{2}{d}\frac{1}{2+\nu i}\, R(i-1), \\
		i = 0 \colon \,\, & R(0) = \frac12 + R(1),
		\end{aligned}
		\end{equation}   
where we use that the total rate of a transition at $i$ equals $2+\nu i$. Recall the definition of $\beta$ in~\eqref{def-beta} and set
		\begin{equation}
		\label{def-Delta-i}
		\Delta(i) \coloneqq \frac{R(i+1)}{R(i)} \beta, \qquad i \in \N_0.
		\end{equation}
The first recursion relation can be written as the forward recursion
		\begin{equation}
		\label{eq:recursion-Delta}
		\Delta(i) = \left(\frac{2+\nu (i+1)}{\rho} - \Delta(i+1)\right)^{-1}, \qquad i \in \N_0,
		\end{equation}
where $\rho=\rho_d$ is given by~\eqref{def-rho}.
Iteration gives
		\begin{equation}
		\Delta(0) 
		= \frac{1|}{|\frac{2+\nu}{\rho}} - \frac{1|}{|\frac{2+2\nu}{\rho}} - \frac{1|}{|\frac{2+3\nu}{\rho}} - \dots = \Delta_{d, \nu}.
		\end{equation}
The second recursion relation can be written as
		\begin{equation}
		\frac{\beta R(1)}{\Delta(0)} = \frac{1}{2}+R(1),
		\end{equation}
from which we get 
		\begin{equation}
		R(1) = \frac12\left(\frac{\beta}{\Delta(0)}-1\right)^{-1}
		\end{equation}
and hence
		\begin{equation}
		R(0) = \frac12\frac{\beta}{\beta-\Delta(0)} = \frac{1}{2\vartheta_{d,\nu}},
		\end{equation}
concluding the proof.
\end{proof}

%%%
	
\subsection{Second toy model}
\label{sec:meet}

It is convenient to introduce finite trees having heights depending on a parameter $n$, which will be related to the size of the dynamic graph in the forthcoming Section \ref{sec:meeting}. To this aim, fix $n\in\N$, $\delta\in(0,\frac1{48})$, and define
	\begin{equation}
	\label{def-sigma}
	\hslash=\hslash_n = \floor{\delta \log_{d} n}\,.
	\end{equation}
The process will be divided in two phases that alternate in a sequential way:
	\begin{itemize}
	\item 
	\emph{First phase}: Consider two infinite $d$-regular trees, call $X$ and $Y$ their roots, and call $\cT^X$ and $\cT^Y$ the subtrees of height $\hslash$ starting at each of the roots. It will be convenient to see an edge of the tree as a matching between two stubs. In particular, for every vertex $z\in \cT^W$ (with $W\in\{X,Y\}$) we let $(\sigma^W_{z,k})_{1\le k\le d}$ be the collection of its stubs and, if $z$ is not the root, we assume that $\sigma^W_{z,1}$ is the unique stub pointing towards the root. Note that $\sum_{i=0}^{\hslash-1}d(d-1)^i$ coincides with the number of edges in each of the two trees (i.e., half of the number of stubs in each tree). Attach to each stub in each tree an independent Poisson process of rate $\frac{\nu}{4}\,\frac{2}{dn-1} \sum_{i=0}^{\hslash-1}d(d-1)^i$, and independently mark each arrival of the Poisson process by a uniformly chosen stub in the other tree. Call 
\begin{equation}
({\mathfrak{T}}_{z,k}^W)_{W\in\{X,Y\},\,z\in\cT^W,\,1\le k\le d}
\end{equation}
the collection of these marked processes.
		
For simplicity, in what follows we call \emph{nice} a pair of process-marks enjoying one of the following properties:
		\begin{itemize}
		\item 
		the arrival is from a process ${\mathfrak{T}}_{z,1}^W$ for some $W \in \{X,Y\}$, $z \neq W$, and the attached mark is $\sigma^{W'}_{z',k'}$ with $W'\neq W$ and $k'\neq 1$;
		\item 
		the arrival is from a process ${\mathfrak{T}}_{z,1}^W$ for some $W \in \{X,Y\}$, $z \neq W$, and the attached mark is $\sigma^{W'}_{W,1}$ with $W' \neq W$;
		\item 
		the arrival is from a process ${\mathfrak{T}}_{z,k}^W$ for some $W \in \{X,Y\}$, $z \neq W$ and $k>1$, and the attached mark is $\sigma^{W'}_{z',1}$ with $W' \neq W$.
		\end{itemize}
Note that a \emph{nice} pair is such that, when rewired, produces a graph in which $X$ and $Y$ are in the same connected component. At the first arrival of a \emph{nice} pair enter the second phase, and call $\{\sigma_{z,k}^X,\sigma_{z',k'}^Y \}$ the pair associated to the first arrival having the above mentioned features (regardless of which of the members of the pair ``generated'' the arrival and which was just the mark). Set
		\begin{equation}
		\label{ij}
		i = {\rm dist}_{\cT^X}(X,z) - \ind_{k>1} \ind_{z\neq X}, \qquad 
		j = {\rm dist}_{\cT^Y}(Y,z') - \ind_{k'>1} \ind_{z'\neq Y},
		\end{equation}
and note that, with this notation, when the selected nice pair is rewired, the distance between $X$ and $Y$ is $i+j+1$.
		\item 
		\emph{Second phase}: 
Consider two random walks evolving on the same infinite $d$-regular tree $\cT^{\rm joint}$ (obtained from the first phase) starting at distance $i+j+1$ (note that, by transitivity, we are free to specify the exact initial positions), and let the edges along the path joining them disappear at rate $\nu$. In other words, consider the process $\hat Z=(\hat Z_t)_{t\ge 0}$ introduced in Section \ref{sec:RW}. If, for some $t>0$, $\hat Z_t=0$, then stop the whole process. Otherwise, if the process hits $\dagger$ before $0$, then stop the second phase and restart from the first phase.
		\end{itemize}
We will call $N$ the number of times in which the two phases are executed up to the end of the process. Moreover, we call $(\tau^\iota_{\rm 1^{st}},\tau^\iota_{\rm 2^{nd}})$, $1\le \iota\le N$, the duration of the first and the second phases in each iteration, and set
	\begin{equation}
	\label{eq:def-tau-final}
	\tau^{\rm tot}_{\rm 1^{st}}=\sum_{\iota=1}^N\tau^\iota_{\rm 1^{st}}\,,\qquad	\tau^{\rm tot}_{\rm 2^{nd}}
	= \sum_{\iota=1}^N\tau^\iota_{\rm 2^{nd}}\,,\qquad \tau_{\rm final}=\tau^{\rm tot}_{\rm 1^{st}}
	+\tau^{\rm tot}_{\rm 2^{nd}}\,,
	\end{equation}
where $\tau_{\rm final}$ clearly represents the total duration of the process. We will use the symbol $\Prob_{d,\nu}=\Prob_{d,\nu}^{(n)}$ (respectively, $\Expect_{d,\nu}=\Expect^{(n)}_{d,\nu}$) to refer to the law (respectively, the expectation) of this process, and note that this law depends on the underlying parameter $n$.

The aim of this section is to prove the following asymptotic result.

\begin{proposition}{\rm \bf [Exponential distribution of $\tau_{\rm final}$]}
\label{lem:exponential-tau}
For every $\nu>0$ and $d\ge 3$,
		\begin{equation}
		\label{eq:meeting1}
		\text{under } \Prob_{d,\nu},\, \frac{\tau_{\rm final}}{n} \text{ converges in distribution to } \Exp(2\vartheta_{d,\nu})
		\text{ as } n\to\infty\,,
		\end{equation}
and
		\begin{equation}
		\label{eq:meeting2}
		\lim_{n\to\infty}\frac{\Expect_{d,\nu}[\tau_{\rm final}]}{n} = \frac1{2\vartheta_{d,\nu}}\,,
		\end{equation}
where $\vartheta_{d,\nu}$ is given by \eqref{def-theta}.
\end{proposition}

We split the proof of Proposition \ref{lem:exponential-tau} into two parts. In Section \ref{suse:tau-first} we prove that~\eqref{eq:meeting1} is satisfied for $\tau^{\rm tot}_{\rm 1^{st}}$. In Section \ref{sec:addtime}, we show that $\tau^{\rm tot}_{\rm 2^{nd}}$ does not affect the distribution of $\tau_{\rm final}$ (nor its expectation) at first order. 
	
%%%

\subsubsection{Control of the first phase}
\label{suse:tau-first}
	
Recall the definition of $q_{d, \nu}$ from~\eqref{def-q}. Note that $\tau_{\rm 1^{st}}^{\rm tot}$ can be sampled as follows: at the end of each iteration $\iota\ge 1$ of the first phase, given the identity of  the \emph{nice} pair of stubs realizing the first arrival, say $\{\sigma_{z,k}^X,\sigma_{z',k'}^Y \}$, sample a Bernoulli random variable of parameter $q_{d,\nu}(i+j+1)$, where $i$ and $j$ are as in \eqref{ij}. Indeed, the latter coincides with the probability that the forthcoming second phase ends with $\hat Z$ hitting $0$, and hence concludes the whole process. If the Bernoulli variable results in a success, the process stops and $N=\iota$, otherwise proceed to sample $\tau_{\rm 1^{st}}^{\iota+1}$ independently. 	
	
By construction, the fact that $\tau_{\rm 1^{st}}^{\rm tot}$ has an exponential law is immediate from the definition: it can be thought of as the first occurrence of independent exponentials associated to the \emph{nice} pairs of stubs, after a thinning of the Poisson processes by a factor $q_{d,\nu}$ depending on the distances of the two vertices in the \emph{nice} pair from the corresponding roots. 

It will be convenient to make explicit the construction we are going to use: For all $\iota \geq 1$ we sample $\tau_{\rm 1^{st}}^\iota$ by taking an independent collection of exponential random variables (one for each nice pair) of rate $\frac\nu2\frac{1}{dn-1}\sum_{i=0}^{\hslash-1}d(d-1)^i$ and recording the first arrival. Indeed, note that each nice pair might occur in two (independent) ways, depending on which element is associated to the process and which to the mark. If the corresponding arrival is at a pair $\{\sigma^X_{z,k}, \sigma^Y_{z',k'}\}$ and $i$ and $j$ given by \eqref{ij}, then we toss a coin with success probability $q_{d,\nu}(i+j+1)$, while if it is a success, then we set $N=\iota$ and stop the procedure. As an outcome of this procedure with get the random vector $(N,\tau_{\rm 1^{st}}^1,\dots,\tau_{\rm 2^{nd}}^N, g_1,\dots, g_N)$, where $g_\iota$ is the pair of edges that achieves the first arrival at the $\iota$-th iteration (regardless which of the two edges is associated to the arrival of the Poisson process and which to the mark).
	
The next lemma shows that the exponential law of $\tau_{\rm 1^{st}}^{\rm tot}$ is preserved in the asymptotic regime $n\to\infty$, and that its rate converges to $2\vartheta_{d,\nu}$.
	
\begin{lemma}{\rm \bf [Exponential distribution of $\tau_{\rm 1^{st}}^{\rm tot}$]}
\label{le:exponential-tau}
For every $\nu>0$ and $d \geq 3$,
		\begin{equation}\label{eq:tau-first}
		\lim_{n\to\infty} \sup_{s> 0} \left| \frac{\log \Prob_{d,\nu} \left(\tau_{\rm 1^{st}}^{\rm tot} > sn \right)}{s} 
		+ 2\vartheta_{d,\nu} \right| = 0.
		\end{equation}
In particular,
		\begin{equation}\label{exp-tau-first}
		\lim_{n\to\infty} \frac{\Expect_{d,\nu}[\tau_{\rm 1^{st}}^{\rm tot}]}{n} = \frac{1}{2\vartheta_{d,\nu}}.
		\end{equation}
\end{lemma}
	
\begin{proof}
We suppress the dependence of $d$ and $\nu$ to improve readability, but we keep the dependence on $n$ to indicate the asymptotic role of this parameter. We already know that $\tau_{\rm 1^{st}}^{\rm tot} \sim \Exp(\gamma_{n})$, for some parameter $\gamma_n$. In order to verify~\eqref{eq:tau-first}, it suffices to prove that $\frac{\gamma_{n}}{n} \to 2\vartheta$. The uniform convergence is then an immediate consequence of the converge in distribution of $\frac{\tau_{\rm 1^{st}}^{\rm tot}}{n}$ together with Dini's Theorem. The proof is articulated in four steps.
		
\medskip\noindent
{\bf 1.} Recall that for each \emph{nice pair} there is a \emph{dual nice pair}, namely, the rewiring of the latter matches the former. Therefore, we can simply double the rates and consider only nice pairs in which the stubs are oriented from the root to the leaves. Consequently, recalling \eqref{def-q}, we have
		\begin{equation}
		\label{eq:tilde-theta}
		\begin{split}
		\gamma_n 
		& = \nu \times \frac{1}{dn-1} \sum_{i=0}^{\hslash_n-1}\sum_{j=0}^{\hslash_n-1}d(d-1)^{i} 
		d (d-1)^{j}q(i+j+1) \\
		& = \nu \times \frac{1}{dn-1}\sum_{\ell=1}^{2\hslash_n-1}\sum_{i=0}^{\ell-1} 
		d(d-1)^{i} d (d-1)^{\ell-i-1}q(\ell) \\
		& = \nu \times \frac{1}{dn-1} \times \frac{d^2}{d-1} \sum_{\ell=1}^{2\hslash_n-1}\ell (d-1)^{\ell} q(\ell) \\
		& =\nu \times \frac{1}{dn-1} \times \frac{d^2}{d-1} \sum_{\ell=1}^{2\hslash_n-1}\ell \beta^{2\ell} q(\ell) \\
		& \sim \frac{1}{n} \times \nu \times \frac{d}{d-1} \sum_{\ell=1}^{2\hslash_n-1}\ell \beta^{2\ell}\,q(\ell).
		\end{split}
		\end{equation} 
Moreover, $q(\ell)$ satisfies the same recursion relations as in \eqref{eq:recursion-R} for $R$, but with initial value $1$:
		\begin{equation}
		\label{eq:recursion-q}   
		\begin{aligned}
		i\in\N\colon 
		&\quad q(i) = 2 \frac{d-1}{d} \frac{1}{2+\nu i} q(i+1) + \frac{2}{d} \frac{1}{2+\nu i} q(i-1), \\[0.2cm]
		i=0 \colon 
		&\quad q(0) = 1,
		\end{aligned}
		\end{equation}  
from which we obtain
		\begin{equation}
		\label{eq:q-final}
		q(\ell) = \prod_{i=0}^{\ell-1} \frac{\Delta(i)}{\beta}, \qquad \ell \in \N,
		\end{equation}
with $\Delta(i)$ defined in \eqref{eq:recursion-Delta}. Inserting \eqref{eq:q-final} into \eqref{eq:tilde-theta}, we obtain
		\begin{equation}
		\label{eq:tilde-theta2}
		\gamma\coloneqq\lim_{n\to\infty} n\gamma_n
		= \frac{d\nu}{d-1}\sum_{\ell \in \N} \ell \beta^{\ell} \prod_{i=0}^{\ell-1} \Delta(i).
		\end{equation}
Thus, it remains to show that $\gamma =2\vartheta$, i.e.,
		\begin{equation}
		\label{eq:left-to-show}
		\frac{d\nu}{d-1}\sum_{\ell\in\N} \ell \beta^\ell\prod_{i=0}^{\ell-1}\Delta(i) 
		= 2\left( 1 - \frac{\Delta(0)}{\beta} \right) = 2\vartheta.
		\end{equation}
		
\medskip\noindent
{\bf 2.}	
Put $\kappa(i)\coloneqq\beta \Delta(i)$, $i \in \N_0$, and reverse the recursion in \eqref{eq:recursion-Delta}, to get
		\begin{equation}
		\label{eq:recursion-eta}
		\kappa(i+1) = \psi_{i+1} - \frac{d-1}{\kappa(i)}, \qquad i \in \N_0,
		\end{equation}
with 
		\begin{equation}
		\psi_i \coloneqq d + \frac{d\nu}{2}i, \qquad i \in \N_0.
		\end{equation}
Next, put 
		\begin{equation}
		\label{eq:chi}
		\chi_i = \prod_{j=0}^i \kappa(j), \qquad i \in \N_0.
		\end{equation} 
Using the recursion in \eqref{eq:recursion-eta}, we get 
		\begin{equation}
		\label{eq:recursion-chi}
		\chi_{i+1} = \psi_{i+1} \chi_i - (d-1)\chi_{i-1}, \qquad i \in \N.
		\end{equation}
Abbreviate
		\begin{equation}
		S \coloneqq \sum_{\ell\in\N} \ell \chi_{\ell-1}.
		\end{equation}
With this notation, \eqref{eq:left-to-show} amounts to showing that
		\begin{equation}
		\label{eq:left-to-show-new} 
		\frac{d\nu}{d-1}S = 2\left(1-\frac{\Delta(0)}{\beta} \right).
		\end{equation}
		
\medskip\noindent
{\bf 3.}
Define the generating function
		\begin{equation}	
		\Phi(z) = \sum_{\ell\in\N} z^\ell \chi_{\ell-1}, \qquad z \in \R.
		\end{equation}
Note that $S = \Phi'(1)$. We derive a differential equation for $\Phi$ with the help of the recursion in \eqref{eq:recursion-chi}. To that end we write
		\begin{equation}
		\label{eq:rew1}
		\begin{aligned}
		&\Phi(z) = z\chi_0 + z^2 \chi_1 + {\rm I} -{\rm  II},
		\end{aligned}
		\end{equation} 
where
		\begin{equation}
		{\rm I} \coloneqq \sum_{\ell=3}^\infty z^\ell \psi_{\ell-1}\chi_{\ell-2}, 
		\qquad {\rm II} \coloneqq \sum_{\ell=3}^\infty z^\ell (d-1) \chi_{\ell-3}.
		\end{equation}
Note that
		\begin{equation}
		\label{eq:rew2}
		{\rm I} = \sum_{\ell\in\N} z^{\ell+1} \left[d+\frac{d\nu}{2}\ell\right] \chi_{\ell-1} - z^2 \psi_1\chi_0
		= d z\,\Phi(z) + \frac{d\nu}{2}\,z^2\,\Phi'(z) - z^2\,\psi_1\chi_0
		\end{equation}   
and 
		\begin{equation}
		\label{eq:rew3}
		{\rm II} = \sum_{\ell\in\N} z^{\ell+2}\,(d-1) \chi_{\ell-1} = (d-1) z^2 \Phi(z).
		\end{equation}
Combining \eqref{eq:rew1}--\eqref{eq:rew3}, we obtain
		\begin{equation}
		\label{eq:diffeq1}
		\alpha(z) \Phi'(z) = \varsigma(z) \Phi(z) + \upsilon(z)
		\end{equation}
with
		\begin{equation}
		\label{eq:diffeq2}
		\alpha(z) \coloneqq\frac{d\nu}{2}z^2, \quad \varsigma(z) 
		\coloneqq (d-1)z^2 - dz + 1, \quad \upsilon(z) \coloneqq (d-1)z^2-\kappa(0)z,
		\end{equation}
where we use \eqref{eq:recursion-eta} for $i=0$ to get $\psi_1\chi_0-\chi_1 = d-1$. Pick now $z=1$ in \eqref{eq:diffeq1}--\eqref{eq:diffeq2} and note that $\varsigma(1)=0$. This gives
		\begin{equation}
		\frac{d\nu}{2}\Phi'(1) = \upsilon(1),
		\end{equation}
provided $\Phi(1)<\infty$. From the fact that $\Phi'(1)=S$, it follows that 
		\begin{equation}
		\frac{d\nu}{d-1} S = \frac{2}{d-1} \upsilon(1) = \frac{2}{d-1}\,\big[(d-1) - \beta\Delta(0)\big]
		= 2\left(1-\frac{\beta\Delta(0)}{d-1}\right) = 2\left(1-\frac{\Delta(0)}{\beta}\right),
		\end{equation}
which proves \eqref{eq:left-to-show-new}. 
		
\medskip\noindent
{\bf 4.}
To conclude the proof of the lemma, we show that $\Phi(1)<\infty$. Recall that $\kappa(i)\coloneqq\beta\Delta(i)$ and that $\Delta(i)$ can be expressed in terms of a continued fraction as 
		\begin{equation}
		\Delta(i) =  \frac{1|}{|\frac{2+(i+1)\nu}{\rho}} - \frac{1|}{|\frac{2+(i+2)\nu}{\rho}} 
		- \frac{1|}{|\frac{2+(i+3)\nu}{\rho}} - \dots,
		\end{equation}
which is immediate from \eqref{eq:recursion-Delta}. The latter shows that $\lim_{i \to \infty} \Delta(i) = 0$ (recall Remark~\ref{cfconv}). Hence $\lim_{i \to \infty} \kappa(i) = 0$ and, via~\eqref{eq:chi}, also $\lim_{i \to \infty} \frac{1}{i} \log \chi_i = - \infty$. Consequently, $\Phi(z) < \infty$ for all $z \in \R$. This concludes the proof of the lemma.
\end{proof}
	
\begin{remark}
The differential equation in~\eqref{eq:diffeq1}--\eqref{eq:diffeq2} can be solved explicitly, namely, 
\begin{equation}
\Phi(z) = \exp\left[\int_{1}^z \dd y\,\frac{\varsigma(y)}{\alpha(y)}\right]
\left\{\Phi(1)+ \int_{(\cdot)}^z \dd y\,\frac{\upsilon(y)}{\alpha(y)} 
\exp\left[-\int_{(\cdot)}^y \dd x\,\frac{\varsigma(x)}{\alpha(x)} \right]\right\}\,,
\end{equation}
where
\begin{equation}
\begin{aligned}
\int_{1}^y \dd x\,\frac{\varsigma(x)}{\alpha(x)} 
&= \int_{1}^y \dd x\,\left[\frac{2(d-1)}{d\nu} - \frac{2}{\nu x} + \frac{2}{d\nu x^2}\right]
= \frac{2(d-1)}{d\nu}\,y - \frac{2}{\nu x}\,\log y - \frac{2}{d\nu}\,\frac{1}{y}
\end{aligned}
\end{equation}	
and
\begin{equation}
\frac{\upsilon(y)}{\alpha(y)} = \frac{2(d-1)}{d\nu} - \frac{2\kappa(0)}{d\nu z}\,.	
\end{equation}
In particular, we may pick $(\cdot) = 1$ and $\upsilon(\cdot) = \Phi(1)$ to get a closed formula.
\end{remark}
	
%%%

\subsubsection{Control of the second phase}
\label{sec:addtime}
	
In what follows it will be useful to consider the same model as above, with the only difference that after $\tau_{\rm final}$ the process restarts at the first phase. We will call this modification the \emph{renewal version}. We are interested in the total amount of time spent by this process in the second phase up to some time $t$ depending on the parameter $n$. To this aim, for $t> 0$ we let $\bar \tau_{\rm 2^{nd}}(t)$ be the total amount of time spent by the \emph{renewal version} of the process in the second phase until time $t$. 
	
We consider the following explicit construction of $\bar \tau_{\rm 2^{nd}}(t)$. We first sample the first phase of the renewal version of the process up to time $t$, and call $\bar N_t$ the total number of times in which the process passes from the first to the second phase before having spent a time $t$ in the first phase. We use the notation $\bar \tau_{\rm 2^{nd}}^\iota$ to denote the length of the $\iota$-th iteration of the second phase for $1\le \iota\le \bar N_t$. For any $1 \leq \iota \leq \bar  N_t$, $\tau^\iota_{\rm 2^{nd}}$ depends on $\tau^\iota_{\rm 1^{st}}$ only though the identity of the \emph{nice} pair of edges that ended the first phase. More precisely, assume that the $\iota$-th iteration of the first phase ended with the clock associated to the \emph{nice pair}  $g_\iota=(\sigma^X_{z,k},\sigma^{Y}_{z',k'})$ ringing. In what follows we will use the notation ${\rm dist}(g_\iota)= \ell$ to mean that $g_\iota$ is of the form $(\sigma^X_{z,k},\sigma^{Y}_{z',k'})$ for some $i$ and $j$ as in \eqref{ij} such that $i+j+1=\ell$. The distribution of $\bar\tau^\iota_{\rm 2^{nd}}$ is the distribution of the first hitting of $\{0,\dagger\}$ by the process $\hat Z=(\hat Z_{t})_{t\ge 0}$ starting at $\hat Z_0=\ell$. More precisely, defining
\begin{equation}
 	H_A \coloneqq \inf\{t \geq 0 \mid \hat Z_t\in A \}, \qquad A \subset \N_0 \cup \{\dagger\},
\end{equation}
we have
\begin{equation}
\label{eq:tau2-cond-dagger-new}
	\Prob_{d,\nu}(\bar \tau^\iota_{\rm 2^{nd}}> s \mid \iota < \bar N_t, {\rm dist}(g_\iota)=\ell)
	=\Pr(H_{\{0,\dagger\}} > s\mid \hat Z_0=\ell), \quad s,t> 0.
\end{equation}
For $1 \leq \ell \leq 2\hslash_{n} - 1$, let $K_\ell$ be the random variable with law $\Pr(H_{\{0,\dagger\}}\in \cdot\mid \hat Z_0=\ell)$. Then
	\begin{equation}
	\label{id}
	\Prob_{d,\nu}\big(\bar\tau_{\rm 2^{nd}}(t)> s\big) \leq \Prob_{d,\nu} \bigg(\sum_{\iota=1}^{\bar N_t} 
	K^{(\iota)}_{{\rm dist}(g_\iota)}> s \bigg),
	\end{equation}
where $(K^{(\iota)}_{\ell})_{\iota\ge 1}$ are independent copies of $K_\ell$, and we assume that these random variables are independent for different $\ell$ and mutually independent.
	
\begin{lemma}
\label{lemma:2nd-phase-new}{\rm \bf [Control on the time spent in the second phase]}
Recall the definition of $\delta$ in \eqref{def-sigma}. For any $t>0$ and $\bar{\tau}_{\rm 2^{nd}}(t)$ defined as above
		\begin{equation}
		\lim_{n \to \infty} \Prob_{d,\nu} ( \bar \tau_{\rm 2^{nd}}(t) > t n^{-1+4\delta}) = 0.
		\end{equation}
\end{lemma}
	
\begin{proof}
For any $m \in \N$, the Markov inequality yields
		\begin{equation}
		\label{eq:expect-tau-add0-neww}
		\begin{split}
		\Prob_{d,\nu} \left( \sum_{\iota=1}^{\bar N_t} K^{(\iota)}_{{\rm dist}(g_\iota)} > s \right)				
		& \leq \Prob_{d,\nu} \left( \sum_{\iota=1}^{m} K^{(\iota)}_{{\rm dist}(g_\iota)} > s \right) + \Prob_{d,\nu} (\bar N_t > m) \\
		& \leq \frac{m \Expect_{d,\nu} [K_{{\rm dist}(g_1)}]}{s} + \frac{\Expect_{d,\nu} [\bar N_t]}{m} \\
		& \leq \frac{m \max_{1\le\ell\le 2\hslash_{n}-1} \Expect_{d,\nu} [K_\ell]}{s} + \frac{\Expect_{d,\nu} [\bar N_t]}{m}.
		\end{split}
		\end{equation}
Note that the maximum in the last display can be bounded by a constant depending only on $\nu$. Indeed, regardless of the value of $\hat Z_s\not\in\{0, \dagger\}$ the transition to $\dagger$ occurs at rate at least $\nu$, which yields the bound $\Expect_{d,\nu} [K_{\ell}] \leq \frac{1}{\nu}$. Moreover, $\Expect_{d,\nu} [\bar N_t]$ is the expected number of arrivals before time $t$ of the Poisson processes associated to the collection of \emph{nice} pairs. Hence, for all $n$ large enough,
		\begin{equation}
		\label{nt}
		\Expect_{d,\nu} [\bar N_t] \le t \frac{\nu}{dn-1} d^\hslash\le t  n^{-1+2\delta}.
		\end{equation}
The desired conclusion follows from \eqref{eq:expect-tau-add0-neww} and \eqref{nt} by choosing $m=t n^{-1+3\delta}$ and $s=t n^{-1+4\delta}$, namely,
\begin{equation}
\Prob_{d,\nu} ( \bar \tau_{\rm 2^{nd}}(t) > t n^{-1+4\delta}) \leq \frac{m}{s \nu} + \frac{tn^{-1+2\delta}}{m} \to 0,
\end{equation}
as $n$ tends to infinity, which concludes the proof of the lemma.
\end{proof}
	
Next, we control the expectation of $\tau^{\rm tot}_{\rm 2^{nd}}$. Similarly to what was done above, assume that the $\iota$-th iteration of the first phase ended with the clock associated to the \emph{nice pair} $g_\iota=(\sigma^X_{z,k},\sigma^{Y}_{z',k'})$ ringing. We saw that the distribution of $\tau^\iota_{\rm 2^{nd}}$ is the distribution of the first hitting of $\{0,\dagger\}$ by the process $\hat Z=(\hat Z_{t})_{t\ge 0}$ starting at $\hat Z_0=\ell$. Nevertheless, if $1\le\iota< N$, then $\tau^\iota_{\rm 2^{nd}}$ is distributed as follows
		\begin{equation}
		\label{eq:tau2-cond-dagger-a}
		\Prob_{d,\nu} ( \tau^\iota_{\rm 2^{nd}} \leq t \mid \iota<N, {\rm dist}(g_\iota)=\ell)
		=\Pr(H_\dagger \leq t \mid \hat Z_0 = \ell, H_\dagger<H_0), \quad t \geq 0,
		\end{equation}
while if $\iota=N$, then
		\begin{equation}
		\label{eq:tau2-cond-0}
		\Prob_{d,\nu} (\tau^\iota_{\rm 2^{nd}} \leq t \mid \iota=N, {\rm dist}(g_\iota)=\ell)
		=\Pr(H_0 \leq t \mid \hat Z_0=\ell, H_\dagger>H_0), \quad t \geq 0.
		\end{equation}
For $1 \leq \ell \leq 2\hslash_n - 1$ let $W_\ell$ with law $\Pr(H_\dagger \in \cdot \mid \hat Z_0 = \ell, H_\dagger < H_0)$ and $F_\ell$ with law $\Pr(H_0 \in \cdot \mid \hat Z_0 = \ell, H_\dagger > H_0)$ (recall \eqref{eq:tau2-cond-dagger-a} and \eqref{eq:tau2-cond-0}), and assume that these random variables are independent for different $\ell$ and mutually independent. With this notation we have
		\begin{equation}
		\Prob_{d,\nu} \big(\tau_{\rm 2^{nd}}^{\rm tot}\le t\big) = \Prob_{d,\nu} \bigg(F_{{\rm dist}(g_N)}
		+\sum_{\iota=1}^{N-1} W^{(\iota)}_{{\rm dist}(g_\iota)}\le t \bigg)\,,
		\end{equation}
where $(W^{(\iota)}_{\ell})_{\iota\in\N}$ are independent copies of $W_{\ell}$.
		
\begin{lemma}{\rm \bf [Control on $\tau_{\rm 2^{nd}}^{\rm tot}$]}
\label{le:bound-err-expo}
For all $d\ge 3$ and $\nu>0$,
			\begin{equation}
			\label{eq:bound-err-expo}
			\lim_{n\to\infty}\frac{\Expect_{d,\nu} \big[ \tau_{\rm 2^{nd}}^{\rm tot} \big] }{n} = 0.
			\end{equation}
\end{lemma}
		
\begin{proof}
The proof that follows seems long for a statement as simple as \eqref{eq:bound-err-expo}, but there are intricate dependencies that need to be controlled.
			
We start by bounding the expectation of $\tau_{\rm 2^{nd}}^{\rm tot}$. Recall the definition of $q$ from~\eqref{def-q} and define
\begin{equation}
	\label{def-Z0}
	\cZ_\dagger = \sum_{\ell = 1}^{2\hslash_n - 1} \Prob_{d,\nu} ( {\rm dist}(g_1) = \ell) [1-q(\ell)],
	\qquad \cZ_0 = \sum_{\ell = 1}^{2\hslash_n - 1} \Prob_{d,\nu} ( {\rm dist}(g_1) = \ell) q(\ell).
\end{equation}
	Given $N$, the collection of random variables $({\rm dist}(g_\iota))_{\iota< N}$ are i.i.d.\ with distribution
			$$
			\Prob_{d,\nu} ( {\rm dist}(g_\iota) = \ell \mid \iota < N )
			=\frac{\Prob_{d,\nu} ({\rm dist}(g_1)=\ell)[1-q(\ell)]}{\cZ_\dagger}, \qquad 1 \leq \ell \leq 2\hslash_{n} -1,
			$$
while ${\rm dist}(g_N)$ is independent of the previous random variables and has distribution
			$$
			\Prob_{d,\nu} ( {\rm dist}(g_\iota) = \ell \mid \iota = N) = \frac{\Prob_{d,\nu} ({\rm dist}(g_1)=\ell)
			q(\ell)}{\cZ_0}, \qquad 1 \leq \ell \leq 2\hslash_{n}-1.
			$$			
The proof is articulated in six steps.
			
\medskip\noindent
{\bf 1.} 
Observe first the simple bound
\begin{equation}
	\label{eq:expect-tau-add0}
	\begin{split}
	\Expect_{d,\nu} [\tau_{\rm 2^{nd}}^{\rm tot}] 
	& = \Expect_{d,\nu} [N -1] \sum_{\ell = 1}^{2\hslash_n - 1} \frac{\Prob_{d,\nu} ( {\rm dist}(g_1) = \ell)[1-q(\ell)]}{\cZ_\dagger}\Expect_{d,\nu} [ W_\ell ] \\
	& \qquad +\sum_{\ell = 1}^{2\hslash_n - 1} \frac{\Prob_{d,\nu} ( {\rm dist}(g_1) = \ell ) q(\ell)}{\cZ_0} \Expect_{d,\nu} [ F_\ell ] \\
	& \leq \Expect_{d,\nu} [N] \max_{\ell \leq 2\hslash_n - 1} \Expect_{d,\nu} [ W_\ell ] + \sum_{\ell = 1}^{2\hslash_n - 1} \frac{\Prob_{d,\nu} ( {\rm dist}(g_1) = \ell ) q(\ell)}{\cZ_0} \Expect_{d,\nu} [F_\ell].
	\end{split}
\end{equation}
			
\medskip\noindent
{\bf 2.}  
We now prove that, for $n$ large enough and uniformly over $d \geq 3$,
			\begin{equation}
			\label{eq:bound-expect-T}
			\max_{\ell \le 2\hslash_n-1}\Expect_{d,\nu} [W_\ell] \le \frac{2+\nu}{\nu^2}\,. 
			\end{equation}
To do so, we start by examining at the expectation of $W_\ell$ for some $1 \leq \ell \leq 2\hslash_{n} - 1$. We bound
\begin{equation}
	\label{eq:bound1}
	\begin{split}
	\Prob_{d,\nu} ( W_\ell > t ) 
	& = \Pr( H_\dagger > t \mid \hat Z_0 = \ell, H_\dagger<H_0) \\
	& = \frac{\Pr( H_\dagger > t, H_\dagger < H_0 \mid \hat Z_0 = \ell )}{\Pr( H_\dagger < H_0 \mid \hat Z_0 = \ell )}
	\leq \frac{\Pr( H_{\{0,\dagger\}} > t \mid \hat Z_0 = \ell )}{\Pr( H_\dagger < H_0 \mid \hat Z_0 = \ell )}.
	\end{split}
\end{equation}
Note that, for $n$ large enough and uniformly over $d \geq 3$ and $1 \leq \ell \leq 2 \hslash_{n}-1$,
			\begin{equation}
			\label{claim1}
			1-q(\ell) = \Pr( H_\dagger < H_0 \mid \hat Z_0 = \ell ) > \frac{\nu}{2+\nu}.
			\end{equation}
Indeed, since $\ell \geq 1$, the probability that the biased random walk hits $\dagger$ before hitting $0$ is at least the probability that, at the very next jump, the process $\hat Z$ jumps to $\dagger$. Similarly, for $n$ large enough and uniformly over $d \geq 3$ and $1 \leq \ell \leq 2\hslash_{n} - 1$,
			\begin{equation}
			\label{claim2}
			\Prob_{d,\nu} ( H_{\{0,\dagger\}} > t \mid \hat Z_0 = \ell ) \leq \ee^{-\nu t},
			\end{equation}
since up to time $H_\dagger \wedge H_0$ the rate to hit $\dagger$ is at least $\nu$. As a consequence of \eqref{eq:bound1}--\eqref{claim2}, we deduce that
			\begin{equation}
			\Expect_{d,\nu} [W_\ell]\le \frac{2+\nu}{\nu}\int_0^\infty \ee^{-\nu t}\, {\rm d}t = \frac{2+\nu}{\nu^{2}},
			\end{equation}
for every $1 \leq \ell \leq 2\hslash_{n}-1$, from which~\eqref{eq:bound-expect-T} follows.
			
\medskip\noindent
{\bf 3.} 
To control the last quantity on the right-hand side of \eqref{eq:expect-tau-add0}, we note that
			\begin{equation}
			\label{Fell}
			\begin{split}
			\Expect_{d,\nu} [F_\ell] = \int_0^\infty\Prob_{d,\nu} (F_\ell>t)\, \dd t 
			= \int_0^\infty \frac{\Pr(H_{\{0,\dagger\}}>t\mid \hat Z_0=\ell)}
			{\Pr(H_\dagger>H_0 \mid \hat Z_0=\ell)}\, \dd t.
			\end{split}
			\end{equation}
Moreover, by definition the denominator in the last display equals $q(\ell)$ (recall \eqref{def-q}). Hence, by \eqref{Fell} and \eqref{claim2},
			\begin{align}
			\nonumber\sum_{\ell=1}^{2\hslash_n-1} \Prob_{d,\nu} & ( {\rm dist}(g_1) = \ell) q(\ell) \Expect_{d,\nu} [F_\ell] \\
			& \leq \sum_{\ell=1}^{2\hslash_n-1} \Prob_{d,\nu} ( {\rm dist}(g_1) = \ell )
			\int_0^\infty \Prob_{d,\nu} ( H_{\{0,\dagger\}} > t \mid \hat Z_0 = \ell)\,\dd t \\
			\label{Fell2}	
			& \leq \sum_{\ell = 1}^{2\hslash_n-1} \Prob_{d,\nu} ( {\rm dist}(g_1) = \ell )
			\int_0^\infty e^{-\nu t} \, \dd t = \frac{1}{\nu}.
			\end{align}
			
\medskip\noindent
{\bf 4.} 
We next show how to control the partition function $\cZ_0$ defined in \eqref{def-Z0}. Note that we can roughly bound
			\begin{equation}
			\cZ_0 > {\Prob_{d,\nu} ( {\rm dist}(g_1) = 1 ) q(1)}.
			\end{equation}
To control the first factor on the right-hand side, note that, for every $1 \leq \ell \leq 2\hslash_n-1$,
			\begin{equation}
			\Prob_{d,\nu} ( {\rm dist}(g_1) = \ell ) \propto \sum_{i=0}^{\hslash_n-1} \sum_{j=0}^{\hslash_n-1} d^2(d-1)^{i+j} \ind_{i+j+1=\ell} = \frac{d^2}{d-1} \ell (d-1)^\ell.
			\end{equation}
It remains to find the proportionality constant, which can be derived by the following asymptotic relation, valid as $n$ tends to infinity:
			\begin{equation}
			\label{use}
			\sum_{\ell=1}^{2\hslash_n-1} \ell(d-1)^\ell \sim \frac{2}{d^{2}}(d-2) (d-1)^{2\hslash_n} \hslash_n.
			\end{equation}
The two previous equations lead to
			\begin{equation}
			\label{est1}
			\Prob_{d,\nu} ( {\rm dist}(g_1) = \ell) \sim \frac{d^{4}}{2(d-2)}\frac{\ell(d-1)^\ell}{(d-1)^{2\hslash_n-1} \hslash_n}.
			\end{equation}
In particular,
			\begin{equation}
			\label{est1bis}
			\Prob_{d,\nu} ( {\rm dist}(g_1) = 1 ) \sim \frac{d^{4}}{2(d-2)}\frac{1}{(d-1)^{2\hslash_n-2}\hslash_n}. 
			\end{equation}
On the other hand,
			\begin{equation}
			\label{est2bis}
			q(1) \geq \frac{\frac{2}{d}}{2+\nu},
			\end{equation}
since it suffices that one of the two random walk traverses the unique edge taking them apart before doing anything else. Therefore, by \eqref{est1bis} and \eqref{est2bis},
			\begin{equation}
			\label{est-const}
			\cZ_0 > \frac{d^{3}}{(2+\nu)(d-2)} \frac{1}{(d-1)^{2\hslash_n-2}\hslash_n}.
			\end{equation}

\medskip\noindent
{\bf 5.} 
We are left to bound $\Expect_{d,\nu} [N]$. Note that
			\begin{equation}
			\begin{split}
			\Expect_{d,\nu} [\tau_{\rm 1^{st}}^{\rm tot}]
			= \Expect_{d,\nu} [N-1] \Expect_{d,\nu} & [\tau_{\rm 1^{st}}^1 \mid N> 1] \\
			& + \Expect_{d,\nu} [\tau_{\rm 1^{st}}^1 \mid N=1] 
			\geq \Expect_{d,\nu} [N-1] \Expect_{d,\nu}[\tau_{\rm 1^{st}}^1 \mid N > 1],
			\end{split}			
			\end{equation}
and therefore
			\begin{equation}
			\label{est3}
			\Expect_{d,\nu} [N] \leq 1 + \frac{\Expect_{d,\nu} [\tau_{\rm 1^{st}}^{\rm tot}]}{\Expect_{d,\nu} [\tau_{\rm 1^{st}}^1 \mid N> 1]}.
			\end{equation}
To bound from below the expectation in the denominator on the right-hand side, we use the tower property and exploit the fact that, conditional on the value of ${\rm dist}(g_1)$, $\tau_{\rm 1^{st}}^1$ is independent of $N$, i.e.,
			\begin{equation}
			\label{est2}
			\begin{split}
			\Expect_{d,\nu} [\tau_{\rm 1^{st}}^1 \mid N>1] & = \sum_{\ell=1}^{2\hslash_n-1} 
			\Prob_{d,\nu} ({\rm dist}(g_1) = \ell \mid N > 1) \Expect_{d,\nu} [\tau_{\rm 1^{st}}^1 \mid {\rm dist}(g_1) = \ell] \\
			& = \sum_{\ell=1}^{2\hslash_n-1} \frac{\Prob_{d,\nu} ( {\rm dist}(g_1) = \ell ) [1-q(\ell)]}{\cZ_\dagger} 
			\Expect_{d,\nu} [\tau_{\rm 1^{st}}^1 \mid {\rm dist}(g_1) = \ell] \\
			& = \sum_{\ell=1}^{2\hslash_n-1} \frac{1-q(\ell)}{\cZ_\dagger} 
			\Expect_{d,\nu} [\tau_{\rm 1^{st}}^1 \ind_{{\rm dist}(g_1) = \ell} ].
			\end{split}
			\end{equation}
Note that, for any fixed $\ell$, the last expectation can be bounded from below by the expected value of the minimum of $\sum_{\ell' \leq 2\hslash_{n}-1}\ell'(d-1)^{\ell'}$ exponential random variables of rate $\frac{\nu}{dn-1}$, and therefore
			\begin{equation}
			\label{estBlue}
			\begin{split}
			\Expect_{d,\nu} [\tau_{\rm 1^{st}}^1 \mid N > 1] & \geq \frac{dn-1}{\nu} \times 
			\frac{1}{3d(d-1)^{2\hslash_n-1} \hslash_n} \times \sum_{\ell=1}^{2\hslash_n-1} \frac{1-q(\ell)}{\cZ_\dagger},
			\end{split}
			\end{equation}
where we use \eqref{use}. Next, we bound from below the last factor on the right-hand side, arguing similarly as in step \textbf{4}. Indeed,
			\begin{equation}
			\label{est5}
			\sum_{\ell=1}^{2\hslash_n-1} \frac{1-q(\ell)} {\cZ_\dagger} \geq \frac{1-q(1)}{\cZ_\dagger} 
			\geq 1-q(1) \geq \frac{\nu}{2+\nu},
			\end{equation}
where for the second inequality we use that the partition function $\cZ_\dagger$ can be bounded from above by $1$ (simply by neglecting the terms $[1-q(\ell)]$ in \eqref{def-Z0}), and for the third we use \eqref{claim1}. In conclusion, combining \eqref{exp-tau-first}, \eqref{est3}, \eqref{est2}, \eqref{estBlue} and \eqref{est5}, we obtain
			\begin{equation}
			\label{est4}
			\Expect_{d,\nu} [N] \leq C_1 \hslash_n (d-1)^{2\hslash_n},
			\end{equation}
for some constant $C_1$ depending only on $d$ and $\nu$.
			
\medskip\noindent
{\bf 6.} 
To conclude, simply substitute~\eqref{eq:bound-expect-T},~\eqref{Fell2},~\eqref{est-const},~and \eqref{est4} into~\eqref{eq:expect-tau-add0}, and recall the definition of $\hslash_n$ in~\eqref{def-sigma} and the fact that $\delta<\tfrac{1}{2}$.
\end{proof}

\begin{proof}[Proof of Proposition \ref{lem:exponential-tau}]
Observe first that Lemma~\ref{le:exponential-tau} yields
		\begin{equation}
		\Prob( \tau_{\rm final} \geq sn) \geq \Prob( \tau_{\rm 1^{st}}^{\rm tot} \geq sn)
		\to e^{-2\vartheta_{d, \nu} s}.
		\end{equation}
On the other hand, for any $t>0$ we have $\tau_{2^{\rm nd}}^{\rm tot} \ind_{\tau_{1^{\rm st}}^{\rm tot} \leq t} \leq \bar \tau_{2^{\rm nd}}(t)$. Therefore, for any $\varepsilon \in (0,s)$,
		\begin{align}
		\Prob( \tau_{\rm final}>sn) & \leq \Prob \big( \tau_{\rm 1^{st}} \geq (s-\varepsilon) n \big) + 
		\Prob \big( \tau_{\rm 1^{st}}^{\rm tot} \leq (s-\varepsilon)n , \tau_{2^{\rm nd}}^{\rm tot}>\varepsilon n) \\
		& \leq  \Prob \big( \tau_{\rm 1^{st}} \geq (s-\varepsilon) n \big)+ \Prob \big( \bar \tau_{2^{\rm nd}}\big( (s-\varepsilon) n \big) > \varepsilon n \big).
		\end{align}
The desired result follows by noting that the first term in the last display converges to $\ee^{-2\vartheta_{d, \nu} (s-\varepsilon)}$ thanks to Lemma \ref{le:exponential-tau}, while the second term converges to $0$ thanks to Lemma \ref{lemma:2nd-phase-new}. By letting $\varepsilon \to 0$ we immediately obtain~\eqref{eq:meeting1}. The limit in~\eqref{eq:meeting2} follows directly from the combination of~\eqref{exp-tau-first} and~\eqref{eq:bound-err-expo}.
\end{proof}

%%%%%%% SECTION 5 %%%%%%%%%%%%%%%%%%%%%%%%%%%%%%%%
	
\section{Proof of the exponential law of the meeting time}
\label{sec:meeting}
	
We begin by introducing the relevant notation. Let $(G_t,X_t,Y_t)_{t \ge 0}$ be the Markov process on $\cG_n(d) \times [n]^2$ with generator $L^{\rm (2)dRW}$ defined in Section~\ref{suse:RW-vot-dyn}, with $(G_0,X_0,Y_0) \overset{d}{=} \mu_d \otimes \pi \otimes \pi$. We write ${\rm dist}_{G_t} (X_t,Y_t)$ for the graph distance between $X_t, Y_t \in [n]$ in $G_t$, and for $x \in [n]$ denote by $\cB_{\hslash,t}(x)$ the ball of radius $\hslash$ (defined in~\eqref{def-sigma}) around $x$ in $G_t$. Moreover, we write $\tx(\cB_{\hslash,t}(x))$ for the tree excess of such a ball, i.e., the difference between the number of vertices in $\cB_{\hslash,t}(x)$ and the number of edges in $\cB_{\hslash,t}(x)$, minus $1$. In this way $\tx(\cB_{\hslash,t}(x)) = 0$ if and only if $\cB_{\hslash,t}(x)$ is a tree. Since we are only interested in the analysis of the process $(G_t,X_t,Y_t)_{t \geq 0}$ up to time $\tau_{\rm meet}^{\pi \otimes \pi}$, for any given initial state $(G,x,y) \in \cG_n(d) \times [n]^2$ we can assume that the process is constructed explicitly by using the graphical construction introduced in Section~\ref{ss.notprop}. Moreover, to ease the reading, in what follows we will suppress the dependence on the initial condition and write $\P,\E$ in place of $\P_{\mu_d, \nu},\E_{\mu_d, \nu}$, and always assume that $(X_0,Y_0) \overset{d}{=} \pi \otimes \pi$, unless specified otherwise.
	
The remainder of this section is devoted to the proof Theorem \ref{th:meeting}. We split the argument into three steps, each discussed in a different subsection. 
\begin{itemize}
\item
In Section~\ref{ss.typev} we define~\emph{typical events}, i.e., time-dependent events of the process $(G_t,X_t,Y_t)_{t \geq 0}$ which occur for all $t\le n^{3/2}$ with high probability.  Since we are interested in studying the process $(G_t,X_t,Y_t)_{t \geq 0}$ up to the meeting time of the two random walks, which we prove to be w.h.p.\ of order $n$, estimates can be performed under the realisation of any finite collection of typical events at a small cost in probability (see Corollary~\ref{coro}).
\item
In Section~\ref{ss.expl} we introduce and analyse an \emph{exploration process}, in which the local graph dynamics is constructed together with the random walks. Such construction exploits the fact that the random walks evolve by means of local rules, making information about the details of the graph far away from their current location essentially irrelevant. A similar approach has already been used successfully in \cite{SOV21pr,SV23pr} for the analysis of the {\em contact process} on dynamic $d$-regular graphs.
\item
In Section~\ref{ss.coup} we construct an explicit coupling between the exploration process and the toy model introduced in Section~\ref{sec:meet}. The proof of Theorem \ref{th:meeting}, which is postponed to the end of the section, uses the analysis in Sections~\ref{ss.typev}--\ref{ss.expl} to show that, in the coupled probability space introduced in Section~\ref{ss.coup}, we have $\tau_{\rm meet}^{\pi\otimes\pi}=\tau_{\rm final}$ with high probability.
\end{itemize}

%%%
	
\subsection{Typical events}
\label{ss.typev}
	
For $t \geq 0$, define the events 
	\begin{equation}
	\label{events}
	\cE^{\rm trees}_t \coloneqq \{ \tx(\cB_{\hslash,t}(X_t)) = \tx(\cB_{\hslash,t}(Y_t)) = 0 \}, 
	\qquad \text{and} \qquad
	\cE_t^{\rm far} \coloneqq \{ {\rm dist}_{G_t}(X_t,Y_t) > 2\hslash \}.
	\end{equation}
Thanks to stationarity of the process and the symmetry between the two random walks, the probabilities of the events above do not depend on $t$. 

It is useful to define the rate at which the balls around $X_t$ and $Y_t$ evolve. By definition, $\cB_{\hslash,t}(X_t) \cup \cB_{\hslash,t}(Y_t)$ evolves when either one (or a pair of) stub therein rewires, or when one of the two random walks moves. Since each stub is selected for a rewiring at rate asymptotically $\nu/2$, and each of the random walks moves at rate $one$, the total rate of change of $\cB_{\hslash,t}(X_t) \cup \cB_{\hslash,t}(Y_t)$ is bounded from above, uniformly in $t \geq 0$, by
	\begin{equation}
	\mathfrak r \coloneqq 2+\frac{\nu}{2} \times \frac1{dn-1} \times 2 \sum_{\ell=0}^{\hslash-1} d(d-1)^{\ell} \times dn \leq C_1 d^{\hslash},
	\end{equation}
for some constant $C_1 = C_1(d,\nu) > 0$ and $n$ large enough. The next results shows that {\em typically} the random walks are far away if one of them has a neighbourhood which is not a tree.

\begin{proposition}{\bf [Typicality of the events $\cE^{\rm trees}_t$ and $\cE^{\rm far}_t$]}
\label{prop:typical}
For any $d \geq 3$ and $\nu > 0$,
		\label{prop:nice-events}
		\begin{equation}\label{eq:EH}
		\lim_{n \to \infty} \P\big( \cE^{\rm trees}_t \cup \cE^{\rm far}_t, \text{ for all } 0 \leq t \leq n^{3/2} \big) = 1.
		\end{equation}
\end{proposition}

\begin{proof}
Define the event
		\begin{equation}
		\cS_t \coloneqq \left\{ \cB_{\hslash,s}(X_s) \cup \cB_{\hslash,s}(Y_s) 
		= \cB_{\hslash,t}(X_t) \cup \cB_{\hslash,t}(Y_t), \text{ for all } s \in [t,t+\mathfrak{r}^{-1}] \right\}.
		\end{equation}
For $t \geq 0$, we bound
		\begin{equation}
		\P ( \cS_t ) \geq \P \big( {\rm Exp}(\mathfrak{r}) > \tfrac1{\mathfrak{r}} \big) = \ee^{-1}.
		\end{equation}
Now let $\cA_t = [\cE^{\rm trees}_t \cup \cE^{\rm far}_t]^\complement$ and note that this event is measurable with respect to $\cB_{\hslash,t}(X_t)\cup \cB_{\hslash,t}(Y_t)$. Define
		\begin{equation}
		\tau_{\rm bad} \coloneqq \inf\{ s \geq 0\colon \cA_s \text{ holds} \},
		\end{equation}
so that~\eqref{eq:EH} can be rephrased as
		\begin{equation}\label{eq:EHalt}
		\lim_{n \to \infty} \P\big( \tau_{\rm bad} \geq n^{3/2} \big) = 1.
		\end{equation}
Thanks to the strong Markov property, we have 
		\begin{equation}
		\P( \{\tau_{\rm bad} \leq t\} \cap \cS_{\tau_{\rm bad}} ) \geq \P(\tau_{\rm bad} \leq t) \ee^{-1}.
		\end{equation}
On the other hand, we can use the a.s.\ inequality
		\begin{equation}
		\mathfrak{r}^{-1}\ind_{\{\tau_{\rm bad} \le t\} \cap \cS_{\tau_{\rm bad}}} 
		\le \int_0^{t+\mathfrak{r}^{-1}} \ind_{\cA_s}\,\dd s\,,
		\end{equation}
where we note that, on the event $\cS_{\tau_{\rm bad}}$, $\cA_s$ holds for all $s \in [\tau_{\rm bad}, \tau_{\rm bad}+\mathfrak{r}^{-1}]$. Taking the expectation and using Fubini's theorem together with the stationarity of the graph dynamics, we deduce
		\begin{equation}
		\label{eq:bound-tau}
		\P(\tau_{\rm bad} \leq t) \leq \ee \mathfrak{r} (t+\mathfrak{r}^{-1}) \P(\cA_0)
		\leq C_2 t d^{\hslash}\, \P(\cA_0)
		\end{equation}
for some constant $C_2 = C_2(d,\nu) > 0$. Consequently, recalling the definition of $\hslash$ in \eqref{def-sigma}, in order to settle the claim it suffices to show that, for some $\varepsilon > \delta$ and all $n$ large enough,
\begin{equation}
		\label{eq:real}
		\P\big( [\cE^{\rm trees}_0]^\complement \cap [\cE^{\rm far}_0]^\complement \big) 
		\leq 2 \P\big( \tx(\cB_{\hslash,t}(X_t)) \geq 1 \text{ and } {\rm dist}_0(X_0,Y_0) \leq 2\hslash \big)
		\leq n^{-3/2-\varepsilon}.
\end{equation}
note that the first bound in the probability above follows from union bound and the fact that $(X_0, Y_0) =^{\rm (d)} (Y_0, X_0)$.
 
To prove the latter we proceed by constructing $(G_0, X_0, Y_0)$ as follows:
		\begin{enumerate}
		\item Sample $X_0 =^{\rm (d)} \pi$.
		\item Construct the neighbourhood $\cB_{2\hslash,0}(X_0)$ by matching its stubs following an arbitrary lexicographic order.
		\item Sample $Y_0 =^{\rm (d)} \pi$ independently.
		\item Match the remaining stubs uniformly at random.
		\end{enumerate}

We now bound
\begin{equation}
\E \big[ \textbf{1}_{{\rm dist}_0(X_0,Y_0) \leq 2\hslash} ~\big|~ \cB_{2\hslash,0}(X_0) \big] = \frac{|\cB_{2\hslash,0}(X_0)|}{n} \leq \frac{d^{2\hslash}}{n} \leq n^{-1+2\delta},
\end{equation}
which follows from the independence of $Y_0$ and the uniform bound on the size of the ball of radius $2\hslash$, $|\cB_{2\hslash,0}(X_0)| \leq d^{2\hslash}$.

Second, let us examine the event where $\tx(\cB_{\hslash,t}(X_0)) \geq 1$, whose probability can be bounded following a standard argument that we briefly describe. By the matching-by-matching construction, in order for $\cB_{\hslash,0}(X_0)$ not to be a tree, it is necessary at at least one of the stubs in $\cB_{\hslash,0}(X_0)$ is matched to another stub available during the construction. This observation applied to all the possible stubs used in the construction of this neighbourhood yields the bound
\begin{equation}
\P\big( \tx(\cB_{\hslash,0}(X_0)) \geq 1 \big) \leq \P\Bigg( {\rm Bin} \bigg( d^\hslash,\frac{d}{dn-2d^\hslash} \bigg) \geq 1 \bigg) \leq \frac{d^{\hslash+1}}{dn-2d^{\hslash}}.
\end{equation}
Combining the two estimates above, we get
\begin{equation}
\begin{split}
\P\big( \tx(\cB_{\hslash,0}(X_0)) \geq 1 \text{ and } & {\rm dist}_0(X_0,Y_0) \leq 2\hslash \big) \\
& = \E \big[ \textbf{1}_{\tx(\cB_{\hslash,0}(X_0)) \geq 1} \E \big[ \textbf{1}_{{\rm dist}_0(X_0,Y_0) 
\leq 2\hslash} ~\big|~ \cB_{2\hslash,0}(X_0) \big] \big] \\
& \leq \E \big[ \textbf{1}_{\tx(\cB_{\hslash,0}(X_0)) \geq 1} n^{-1+2\delta} \big] \\
& \leq \frac{d^{\hslash+1}}{dn-2d^{\hslash}}n^{-1+2\delta} \leq  C_{3} n^{-2+3\delta},
\end{split}
\end{equation}
for some constant $C_3(d)>0$, provided $n$ is large enough. The proof is complete by noting that $\delta < \tfrac{1}{48}$.
\end{proof}

\begin{remark}
We are interested in excluding w.h.p.\ the complement of a typical event on the time scale of the meeting time, i.e., $\Theta(n)$. Hence the choice of the exponent $\tfrac{3}{2}$ is expedient only to avoid the introduction of further constants. All the statements above are true when $\tfrac{3}{2}$ is replaced by $2-C_2\delta$ for some $C_2=C_2(d,\nu)>0$ with $\delta$ as in \eqref{def-sigma}.
\end{remark}
	
\begin{corollary}{\bf [Restriction to typicality]}
\label{coro}
For all $0 \le t \le n^{3/2}$,
\begin{equation}
		\lim_{n \to \infty} \P\big( \{\tau_{\rm meet}^{\pi \otimes \pi} > t\} \cap \{ \cE^{\rm trees}_s \cup \cE^{\rm far}_s,
		\text{ for all } 0 \leq s \leq t\}^\complement \, \big) = 0.
\end{equation}
\end{corollary}
	
%%%
	
\subsection{Exploration process}
\label{ss.expl}
	
We next consider an exploration process similar in spirit to the one in \cite[Section 4.2]{SV23pr}. Let
	\begin{equation}
	\mathfrak{P}_{n}(d) \coloneqq \big\{ \{\sigma_{x,i},\sigma_{y,j}\} \colon x,y\in [n], 	1 \leq  i,j \leq d, (x,i) \neq (y,j) \big\}
	\end{equation}
be the set of all potential edges of the random graph. A set $E \subset \mathfrak{P}_{n}(d)$ is called \emph{a partial matching} when no stub is present in $E$ more than once. With this notation
	\begin{equation}
	{\cP}_n(d) \coloneqq \{E \subset \mathfrak{P}_n(d) \colon E \text{ is a partial matching} \}
	\end{equation}
represents the set of all the possible {partial matchings} of the graph, so that in particular $\cG_n(d) \subset {\cP}_n(d)$. Using the same source of randomness used to construct $(G_t,X_t,Y_t)_{t \geq 0}$, we will define a process $(E_t)_{t \geq 0}$ taking values in $\mathfrak{P}_n(d)$. In words, $E_t$ will represent the collection of edges that are known to be part of $G_t$ by either of the two random walks. 
	
Without risk of confusion, in this section we use the symbol $\sigma \leftrightarrow_t \sigma'$ to mean that the stubs $\sigma$ and $\sigma'$ are matched in $E_t$ (similarly for ${\rm v}_s (\sigma)$), and write ${\rm dist}_t(x,y)$ for the (graph) distance in $E_t$. We will sometimes abuse notation by letting $\sigma \in E_t$ mean that the stub $\sigma$ is matched in $E_t$. We denote by $\bar E_t$ the collection of stubs that are matched in $E_t$, together with the stubs that are unmatched in $E_t$ whose corresponding vertex has at least one other stub matched in $E_t$.
	
Before explicitly defining the process, we claim that it will enjoy the following properties:
	\begin{itemize}
	\item[(P1)] 
	The triple $(E_t,X_t,Y_t)_{t \geq 0}$ is a Markov chain.
	\item[(P2)] 
	For any $t \geq 0$, the distribution of $(G_t,X_t,Y_t)$ initialized at $(G_0,X_0,Y_0) =^{\rm(d)} \mu_d \otimes \pi \otimes \pi$ coincides with the product distribution of $(E_t,X_t,Y_t)$ and the uniform distribution on the possible matching of the stubs that are unmatched in $E_t$.
	\end{itemize}
In other words, the process $(E_t,X_t,Y_t)_{t\ge0}$ is a Markovian marginal of the original process  $(G_t,X_t,Y_t)_{t\ge 0}$ in which the only randomness used is the one required to know the $\hslash$-neighbourhoods of $(X_t,Y_t)_{t \geq 0}$. Intuitively, the exploration process has to be thought of as follows: the two random walks can ``see'' up to distance $\hslash$, so at any time they are aware of the current matchings that are at most $\hslash$ steps apart from either of them. On the other hand, if, for instance, some rewiring involving the $\hslash$-neighbourhood of the first random walk takes place, then that random walk has knowledge of its new $\hslash$-neighbourhood {\em and} of the matchings involved in the sub-graph that has been moved away from it by the rewiring (the gray edges in Figures \ref{fig:explorationRW}--\ref{fig:explorationDec}). Nevertheless, this last information rapidly becomes negligible, since those edges that are now ``far'' from the walks are rewired at rate $\nu$.
	
Let us formally define the process $(E_t,X_t,Y_t)_{t \geq 0}$. 
	\begin{enumerate}
	\item 
	Sample $(X_0,Y_0) =^{\rm (d)} \pi \otimes \pi$.
	\item 
	Construct $\cB_{\hslash,0}(X_0) \cup \cB_{\hslash,0}(Y_0)$, following the breadth-first construction (see the construction below~\eqref{eq:real}).
	\item 
	Let $E_0$ be the set of edges (= pairs of matched stubs) in $\cB_{\hslash,0}(X_0)\cup \cB_{\hslash,0}(Y_0)$.
	\item 
	Use the same clocks as in Section \ref{ss.notprop}: $(\mathfrak{T}_{z,i}^{\rm rw},\mathfrak{T}_{z,i}^{\rm dyn})_{z\in[n],i\in[d]}$.
	\item 
	For each pair $(z,i) \in [n] \times d$, construct an unmarked Poisson process $\hat{\mathfrak{T}}_{z,i}^{\rm dyn}$ by retaining the arrival times of $\mathfrak{T}_{z,i}^{\rm dyn}$ and the arrival times of $({\mathfrak{T}}_{u,k}^{\rm dyn})_{(u,k)\neq(z,i)}$, which are marked by $\sigma_{z,i}$.
	Note that the collection constructed above is not independent. In fact, every arrival time is associated to two Poisson processes $\hat{\mathfrak{T}}_{z,i}^{\rm dyn}$: one for the stub that originated the arrival time and the other for the stub marked by this arrival.
	\item 
	We are now in position to define the dynamics. For each $s \geq 0$, denote by $t >s$ the first arrival after time $s$ among the Poisson processes $(\mathfrak{T}^{\rm rw}_{z,i})_{z \in \{X_s,Y_s\}, 1 \leq i \leq d}$ and $(\hat{\mathfrak{T}}^{\rm dyn}_{z,i})_{\{ (z,i) \colon \sigma_{z,i} \in E_s \}}$. We set $(E_r, X_r, Y_r) = (E_s, X_s, Y_s) $ for $r \in (s,t)$ and now define $(E_t, X_t, Y_t)$ in the following way.
	\begin{itemize}
	\item[(a)]
	{\bf [Random walk transition]} If $t>s$ is an arrival time of $\mathfrak{T}^{\rm rw}_{X_s,i}$ for some $1 \leq i \leq d$, set $Y_t = Y_s$, $X_t = {\rm v}_s (\sigma_{X_s,i})$, and
	\begin{equation}
	E_t = E_s \cup V^X(X_s,\sigma_{X_s,i}),
	\end{equation}
	where $V^X(X_s,\sigma_{X_s,i})$ is a random variable in $\cP_n(d)$ obtained by randomly matching the stubs in $\cB_{\hslash,t}(X_t)$ (respectively, $\cB_{\hslash,t}(Y_t)$) that are unmatched in $E_s$. See Figure \ref{fig:explorationRW}. The case when $t>s$ is a arrival of $\mathfrak{T}^{\rm rw}_{Y_s,i}$, for some $1 \leq i \leq d$ is defined analogously.
	\item[(b)]
	{\bf [Internal rewiring of an edge]} 
	If  $t>s$ is an arrival time of both $\hat{\mathfrak{T}}_{z,i}^{\rm dyn}$ and $\hat{\mathfrak{T}}_{v,j}^{\rm dyn}$ for stubs $\sigma_{z,i},\sigma_{v,j}\in \bar E_s$, denote by $\sigma_{u,k}$ the stub matched to $\sigma_{z,i}$ in $E_s$ (if any, i.e., if $\sigma_{z,i}\in E_s$) and $\sigma_{u',k'}$ the stub matched to $\sigma_{v,j}$ in $\bar E_s$ (if any). Let
	\begin{equation}
	E_{s}^{-} = E_s\setminus\left(\{\sigma_{z,i},\sigma_{u,k} \}\cup\{\sigma_{v,j},\sigma_{u',k'} \}\right),
	\end{equation}
	and set
	\begin{equation}
		E_t = E_{s}^{-} \cup V^{-}(\sigma_{z,i},\sigma_{v,j},E_s),
	\end{equation}
	where $V^{-}(\sigma_{z,i},\sigma_{v,j},E_s)$ is a random variable in $\cP_n(d)$ obtained by iteratively matching at random the stubs at distance less than $\hslash$ from either $X_t$ or $Y_t$ that are unmatched in $E_{s}^{-}$.	See Figures \ref{fig:explorationV-}, \ref{fig:explorationV0} and \ref{fig:explorationSigma-}.
	\item[(c)]
	{\bf [External rewiring of an edge]} 
		If $t>s$ is an arrival time of a unique $\hat{\mathfrak{T}}_{z,i}^{\rm dyn}$ for some $\sigma_{z,i} \in \bar E_s$ such that $z \in \cB_{\hslash,s}(X_s) \cup \cB_{\hslash,t}(Y_s)$, denote by $\sigma_{u,k}$ the stub matched to $\sigma_{z,i}$ in $\bar E_s$ (if any) and set
	\begin{equation}
		E_t = E_s\setminus\left(\{\sigma_{z,i},\sigma_{u,k} \}\right) \cup V^{+}(\sigma_{z,i},E_s),
	\end{equation} 
	where $V^{+}(\sigma_{z,i},E_s)$ is a random variable in $\cP_n(d)$ obtained by iteratively matching at random the stubs at distance less than $\hslash$ from either $X_t$ or $Y_t$ that are unmatched in $E_s\setminus\{\sigma_{z,i},\sigma_{u,k} \}$. See Figure \ref{fig:explorationV+}.
		\item[(d)] 
	{\bf [External rewiring far from the two random walks]} 
	Finally, if $t>s$ is an arrival time of $\hat{\mathfrak{T}}_{z,i}^{\rm dyn}$ for a unique $\sigma_{z,i} \in E_s$, with $z \not \in \cB_{\hslash,s}(X_s) \cup \cB_{\hslash,t}(Y_s)$, denote by $\sigma_{u,k}$ the stub matched to $\sigma_{z,i}$ in $E_s$ and set
	\begin{equation}
		E_t = E_s\setminus\left( \{\sigma_{z,i},\sigma_{u,k} \}\right).
	\end{equation} 
	See Figure \ref{fig:explorationDec}.
			\end{itemize}
			\end{enumerate}
The reader can check that properties (P1)-(P2) are satisfied by the definition of the process $(E_t,X_t,Y_t)_{t\ge 0}$. Without risk of confusion, we will use the same symbol $\P$ to refer to the probability law associated to the graphical construction just described.

%%%%%%%%%%%%%%%%%%%%%%%%%%%%%%%%%%%%%%%%%%%%%%%
\begin{figure}[!ht]
\centering
\begin{tikzpicture}[scale=0.5]

%Stub
\draw [line width=2, green] (0,4.25) -- (1, 27/8);
\draw [line width=2, green] (6/8,173/56) -- (10/8,205/56);
\draw (6/8,173/56) -- (10/8,205/56);
\node[right] at (1, 27/8) {{\tiny $\sigma_{X_{s}, 3}$}};

%Left branch
\draw (0,4.25) -- (-2,2.5);
\draw (-2,2.5) -- (-2.5,1.25);
\draw (-2,2.5) -- (-1.5,1.25);
\draw (-1.5,1.25) -- (-1.75,0);
\draw (-1.5,1.25) -- (-1.25,0);
\draw (-2.5,1.25) -- (-2.25,0);
\draw (-2.5,1.25) -- (-2.75,0);

%Middle branch
\draw (0,4.25) -- (0,2.5);
\draw (0,2.5) -- (-0.55,1.25);
\draw (0,2.5) -- (0.5,1.25);
\draw (-0.5,1.25) -- (-0.75,0);
\draw (-0.5,1.25) -- (-0.25,0);
\draw (0.5,1.25) -- (0.25,0);
\draw (0.5,1.25) -- (0.75,0);

%Right branch
\draw (0,4.25) -- (2,2.5);
\draw (2,2.5) -- (2.5,1.25);
\draw (2,2.5) -- (1.5,1.25);
\draw (1.5,1.25) -- (1.75,0);
\draw (1.5,1.25) -- (1.25,0);
\draw (2.5,1.25) -- (2.25,0);
\draw (2.5,1.25) -- (2.75,0);

%Root
\draw [fill=blue] (0,4.25) circle (0.15);
\node[above] at (0,4.25) {$X_{s}$}; 

%Left branch
\draw [fill=red] (-2,2.5) circle (0.15);
\draw [fill=white] (-1.5,1.25) circle (0.15);
\draw [fill=white] (-2.5,1.25) circle (0.15);
\draw [fill=white] (-2.75,0) circle (0.15);
\draw [fill=white] (-2.25,0) circle (0.15);
\draw [fill=white] (-1.75,0) circle (0.15);
\draw [fill=white] (-1.25,0) circle (0.15);

%Middle branch
\draw [fill=purple!70!white] (0,2.5) circle (0.15);
\draw [fill=white] (-0.5,1.25) circle (0.15);
\draw [fill=white] (0.5,1.25) circle (0.15);
\draw [fill=white] (0.75,0) circle (0.15);
\draw [fill=white] (0.25,0) circle (0.15);
\draw [fill=white] (-0.25,0) circle (0.15);
\draw [fill=white] (-0.75,0) circle (0.15);

%Right branch
\draw [fill=orange] (2,2.5) circle (0.15);
\draw [fill=pink] (1.5,1.25) circle (0.15);
\draw [fill=yellow] (2.5,1.25) circle (0.15);
\draw [fill=white] (2.75,0) circle (0.15);
\draw [fill=white] (2.25,0) circle (0.15);
\draw [fill=white] (1.75,0) circle (0.15);
\draw [fill=white] (1.25,0) circle (0.15);

\begin{scope}[shift={(6,0)}]

%Left branch
\draw (0,4.25) -- (-2,2.5);
\draw (-2,2.5) -- (-2.5,1.25);
\draw (-2,2.5) -- (-1.5,1.25);
\draw (-1.5,1.25) -- (-1.75,0);
\draw (-1.5,1.25) -- (-1.25,0);
\draw (-2.5,1.25) -- (-2.25,0);
\draw (-2.5,1.25) -- (-2.75,0);

%Middle branch
\draw (0,4.25) -- (0,2.5);
\draw (0,2.5) -- (-0.55,1.25);
\draw (0,2.5) -- (0.5,1.25);
\draw (-0.5,1.25) -- (-0.75,0);
\draw (-0.5,1.25) -- (-0.25,0);
\draw (0.5,1.25) -- (0.25,0);
\draw (0.5,1.25) -- (0.75,0);

%Right branch
\draw (0,4.25) -- (2,2.5);
\draw (2,2.5) -- (2.5,1.25);
\draw (2,2.5) -- (1.5,1.25);
\draw (1.5,1.25) -- (2.5,1.25);
\draw (1.5,1.25) -- (1.25,0);
\draw (2.5,1.25) -- (2.75,0);

%Root
\draw [fill=violet] (0,4.25) circle (0.15);
\node[above] at (0,4.25) {$Y_{s}$}; 

%Left branch
\draw [fill=cyan] (-2,2.5) circle (0.15);
\draw [fill=white] (-1.5,1.25) circle (0.15);
\draw [fill=white] (-2.5,1.25) circle (0.15);
\draw [fill=white] (-2.75,0) circle (0.15);
\draw [fill=white] (-2.25,0) circle (0.15);
\draw [fill=white] (-1.75,0) circle (0.15);
\draw [fill=white] (-1.25,0) circle (0.15);

%Middle branch
\draw [fill=lime] (0,2.5) circle (0.15);
\draw [fill=white] (-0.5,1.25) circle (0.15);
\draw [fill=white] (0.5,1.25) circle (0.15);
\draw [fill=white] (0.75,0) circle (0.15);
\draw [fill=white] (0.25,0) circle (0.15);
\draw [fill=white] (-0.25,0) circle (0.15);
\draw [fill=white] (-0.75,0) circle (0.15);

%Right branch
\draw [fill=olive] (2,2.5) circle (0.15);
\draw [fill=white] (1.5,1.25) circle (0.15);
\draw [fill=white] (2.5,1.25) circle (0.15);
\draw [fill=white] (2.75,0) circle (0.15);
\draw [fill=white] (1.25,0) circle (0.15);

\end{scope}

\begin{scope}[shift={(15,0)}]

%Stub
\draw [line width=2, green] (-2,2.5) -- (-1, 27/8);
\draw [line width=2, green] (-6/8,173/56) -- (-10/8,205/56);
\draw (-6/8,173/56) -- (-10/8,205/56);
\node[left] at (-1, 27/8) {{\tiny $\sigma_{X_{s}, 3}$}};

%Left branch
\draw (0,4.25) -- (-2,2.5);
\draw (-2,2.5) -- (-2.5,1.25);
\draw (-2,2.5) -- (-1.5,1.25);
\draw (-1.5,1.25) -- (-1.75,0);
\draw (-1.5,1.25) -- (-1.25,0);
\draw (-2.5,1.25) -- (-2.25,0);
\draw (-2.5,1.25) -- (-2.75,0);

%Middle branch
\draw (0,4.25) -- (0,2.5);
\draw (0,2.5) -- (-0.55,1.25);
\draw (0,2.5) -- (0.5,1.25);
\draw [color=brown!90!black] (-0.5,1.25) -- (-0.75,0);
\draw [color=brown!90!black] (-0.5,1.25) -- (-0.25,0);
\draw [color=brown!90!black] (0.5,1.25) -- (0.25,0);
\draw [color=brown!90!black] (0.5,1.25) -- (0.75,0);

%Right branch
\draw (0,4.25) -- (2,2.5);
\draw (2,2.5) -- (2.5,1.25);
\draw (2,2.5) -- (1.5,1.25);
\draw [color=brown!90!black] (1.5,1.25) -- (1.75,0);
\draw [color=brown!90!black] (1.5,1.25) -- (1.25,0);
\draw [color=brown!90!black] (2.5,1.25) -- (2.25,0);
\draw [color=brown!90!black] (2.5,1.25) -- (2.75,0);

%Bottom attachment

\draw [color=gray] (-2.75,0) -- (-2.9,-1.5);
\draw [color=gray] (-2.75,0) -- (-3.25,-1.5);
\draw [color=gray] (-2.25,0) -- (-2.55,-1.5);
\draw [color=gray] (-2.25,0) -- (-2.2,-1.5);
\draw [color=gray] (-1.75,0) -- (-1.8,-1.5);
\draw [color=gray] (-1.75,0) -- (-1.45,-1.5);
\draw [color=gray] (-1.25,0) -- (-1.1,-1.5);
\draw [color=gray] (-1.25,0) -- (-0.75,-1.5);

\draw [color=gray, fill=white] (-3.25,-1.5) circle (0.15);
\draw [color=gray, fill=white] (-2.9,-1.5) circle (0.15);
\draw [color=gray, fill=white] (-2.55,-1.5) circle (0.15);
\draw [color=gray, fill=white] (-2.2,-1.5) circle (0.15);
\draw [color=gray, fill=white] (-1.8,-1.5) circle (0.15);
\draw [color=gray, fill=white] (-1.45,-1.5) circle (0.15);
\draw [color=gray, fill=white] (-1.1,-1.5) circle (0.15);
\draw [color=gray, fill=white] (-0.75,-1.5) circle (0.15);

%Root
\draw [fill=orange] (0,4.25) circle (0.15);
\node[above] at (0,4.25) {$X_{t}$}; 

%Left branch
\draw [fill=blue] (-2,2.5) circle (0.15);
\draw [fill=purple!70!white] (-1.5,1.25) circle (0.15);
\draw [fill=red] (-2.5,1.25) circle (0.15);
\draw [fill=white] (-2.75,0) circle (0.15);
\draw [fill=white] (-2.25,0) circle (0.15);
\draw [fill=white] (-1.75,0) circle (0.15);
\draw [fill=white] (-1.25,0) circle (0.15);

%Middle branch
\draw [fill=pink] (0,2.5) circle (0.15);
\draw [fill=white] (-0.5,1.25) circle (0.15);
\draw [fill=white] (0.5,1.25) circle (0.15);
\draw [color=brown!90!black, fill=white] (0.75,0) circle (0.15);
\draw [color=brown!90!black, fill=white] (0.25,0) circle (0.15);
\draw [color=brown!90!black, fill=white] (-0.25,0) circle (0.15);
\draw [color=brown!90!black, fill=white] (-0.75,0) circle (0.15);

%Right branch
\draw [fill=yellow] (2,2.5) circle (0.15);
\draw [fill=white] (1.5,1.25) circle (0.15);
\draw [fill=white] (2.5,1.25) circle (0.15);
\draw [color=brown!90!black, fill=white] (2.75,0) circle (0.15);
\draw [color=brown!90!black, fill=white] (2.25,0) circle (0.15);
\draw [color=brown!90!black, fill=white] (1.75,0) circle (0.15);
\draw [color=brown!90!black, fill=white] (1.25,0) circle (0.15);

\draw [decorate,decoration={brace,amplitude=5pt,mirror}] (-0.9,-0.3) -- (2.9,-0.3) node[midway, below] { {\tiny $V^{X}(X_{s}, \sigma_{X_{s},3})$} };

\begin{scope}[shift={(6,0)}]

%Left branch
\draw (0,4.25) -- (-2,2.5);
\draw (-2,2.5) -- (-2.5,1.25);
\draw (-2,2.5) -- (-1.5,1.25);
\draw (-1.5,1.25) -- (-1.75,0);
\draw (-1.5,1.25) -- (-1.25,0);
\draw (-2.5,1.25) -- (-2.25,0);
\draw (-2.5,1.25) -- (-2.75,0);

%Middle branch
\draw (0,4.25) -- (0,2.5);
\draw (0,2.5) -- (-0.55,1.25);
\draw (0,2.5) -- (0.5,1.25);
\draw (-0.5,1.25) -- (-0.75,0);
\draw (-0.5,1.25) -- (-0.25,0);
\draw (0.5,1.25) -- (0.25,0);
\draw (0.5,1.25) -- (0.75,0);

%Right branch
\draw (0,4.25) -- (2,2.5);
\draw (2,2.5) -- (2.5,1.25);
\draw (2,2.5) -- (1.5,1.25);
\draw (1.5,1.25) -- (2.5,1.25);
\draw (1.5,1.25) -- (1.25,0);
\draw (2.5,1.25) -- (2.75,0);

%Root
\draw [fill=violet] (0,4.25) circle (0.15);
\node[above] at (0,4.25) {$Y_{t}$}; 

%Left branch
\draw [fill=cyan] (-2,2.5) circle (0.15);
\draw [fill=white] (-1.5,1.25) circle (0.15);
\draw [fill=white] (-2.5,1.25) circle (0.15);
\draw [fill=white] (-2.75,0) circle (0.15);
\draw [fill=white] (-2.25,0) circle (0.15);
\draw [fill=white] (-1.75,0) circle (0.15);
\draw [fill=white] (-1.25,0) circle (0.15);

%Middle branch
\draw [fill=lime] (0,2.5) circle (0.15);
\draw [fill=white] (-0.5,1.25) circle (0.15);
\draw [fill=white] (0.5,1.25) circle (0.15);
\draw [fill=white] (0.75,0) circle (0.15);
\draw [fill=white] (0.25,0) circle (0.15);
\draw [fill=white] (-0.25,0) circle (0.15);
\draw [fill=white] (-0.75,0) circle (0.15);

%Right branch
\draw [fill=olive] (2,2.5) circle (0.15);
\draw [fill=white] (1.5,1.25) circle (0.15);
\draw [fill=white] (2.5,1.25) circle (0.15);
\draw [fill=white] (2.75,0) circle (0.15);
\draw [fill=white] (1.25,0) circle (0.15);

\end{scope}

\end{scope}

\end{tikzpicture}

\caption{Visual representation of the transition in step (6-a), where an arrival of $\mathfrak{T}^{\rm rw}_{X_s,3}$ is portrayed, and $\hslash = 3$. Edges represented in gray in the second plot are those that are not part of $\cB_{\hslash,t}(X_t)\cup\cB_{\hslash,t}(Y_t) $.}
\label{fig:explorationRW}
\end{figure}
%%%%%%%%%%%%%%%%%%%%%%%%%%%%%%%%%%%%%%%%%%%%%%%%%%%%%%

%%%%%%%%%%%%%%%%%%%%%%%%%%%%%%%%%%%%%%%%%%%%%%%%%%%%%%%%%%
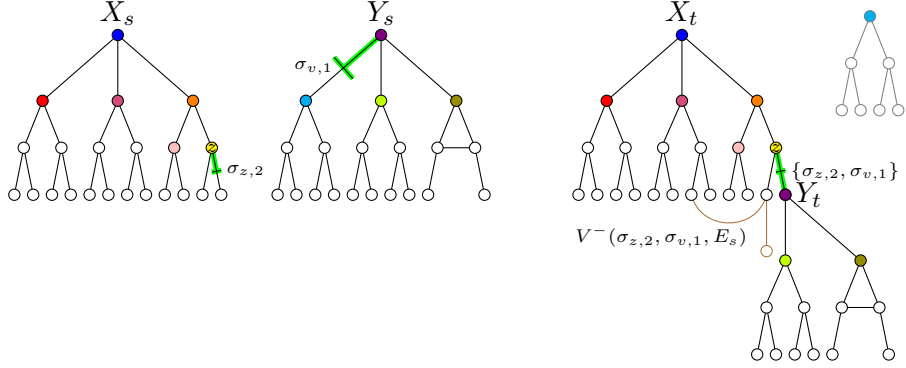
\begin{figure}[!ht]
\centering
\begin{tikzpicture}[scale=0.5]

%Stub
\draw [line width=2, green] (2.5,1.25) -- (21/8, 5/8);
\draw [line width=2, green] (20/8,24/40) -- (22/8,26/40);
\draw (20/8,24/40) -- (22/8,26/40);
\node[right] at (21/8, 5/8) {{\tiny	 $\sigma_{z, 2}$}};

%Left branch
\draw (0,4.25) -- (-2,2.5);
\draw (-2,2.5) -- (-2.5,1.25);
\draw (-2,2.5) -- (-1.5,1.25);
\draw (-1.5,1.25) -- (-1.75,0);
\draw (-1.5,1.25) -- (-1.25,0);
\draw (-2.5,1.25) -- (-2.25,0);
\draw (-2.5,1.25) -- (-2.75,0);

%Middle branch
\draw (0,4.25) -- (0,2.5);
\draw (0,2.5) -- (-0.55,1.25);
\draw (0,2.5) -- (0.5,1.25);
\draw (-0.5,1.25) -- (-0.75,0);
\draw (-0.5,1.25) -- (-0.25,0);
\draw (0.5,1.25) -- (0.25,0);
\draw (0.5,1.25) -- (0.75,0);

%Right branch
\draw (0,4.25) -- (2,2.5);
\draw (2,2.5) -- (2.5,1.25);
\draw (2,2.5) -- (1.5,1.25);
\draw (1.5,1.25) -- (1.75,0);
\draw (1.5,1.25) -- (1.25,0);
\draw (2.5,1.25) -- (2.25,0);
\draw (2.5,1.25) -- (2.75,0);

%Root
\draw [fill=blue] (0,4.25) circle (0.15);
\node[above] at (0,4.25) {$X_{s}$}; 

%Left branch
\draw [fill=red] (-2,2.5) circle (0.15);
\draw [fill=white] (-1.5,1.25) circle (0.15);
\draw [fill=white] (-2.5,1.25) circle (0.15);
\draw [fill=white] (-2.75,0) circle (0.15);
\draw [fill=white] (-2.25,0) circle (0.15);
\draw [fill=white] (-1.75,0) circle (0.15);
\draw [fill=white] (-1.25,0) circle (0.15);

%Middle branch
\draw [fill=purple!70!white] (0,2.5) circle (0.15);
\draw [fill=white] (-0.5,1.25) circle (0.15);
\draw [fill=white] (0.5,1.25) circle (0.15);
\draw [fill=white] (0.75,0) circle (0.15);
\draw [fill=white] (0.25,0) circle (0.15);
\draw [fill=white] (-0.25,0) circle (0.15);
\draw [fill=white] (-0.75,0) circle (0.15);

%Right branch
\draw [fill=orange] (2,2.5) circle (0.15);
\draw [fill=pink] (1.5,1.25) circle (0.15);
\draw [fill=yellow] (2.5,1.25) circle (0.15);
\draw [fill=white] (2.75,0) circle (0.15);
\draw [fill=white] (2.25,0) circle (0.15);
\draw [fill=white] (1.75,0) circle (0.15);
\draw [fill=white] (1.25,0) circle (0.15);

\node at (2.5, 1.25) { {\tiny $z$} };

\begin{scope}[shift={(7,0)}]

%Stub
\draw [line width=2, green] (0,4.25) -- (-1, 27/8);
\draw [line width=2, green] (-6/8,173/56) -- (-10/8,205/56);
\draw (-6/8,173/56) -- (-10/8,205/56);
\node[left] at (-1, 27/8) {{\tiny $\sigma_{v,1}$}};

%Left branch
\draw (0,4.25) -- (-2,2.5);
\draw (-2,2.5) -- (-2.5,1.25);
\draw (-2,2.5) -- (-1.5,1.25);
\draw (-1.5,1.25) -- (-1.75,0);
\draw (-1.5,1.25) -- (-1.25,0);
\draw (-2.5,1.25) -- (-2.25,0);
\draw (-2.5,1.25) -- (-2.75,0);

%Middle branch
\draw (0,4.25) -- (0,2.5);
\draw (0,2.5) -- (-0.55,1.25);
\draw (0,2.5) -- (0.5,1.25);
\draw (-0.5,1.25) -- (-0.75,0);
\draw (-0.5,1.25) -- (-0.25,0);
\draw (0.5,1.25) -- (0.25,0);
\draw (0.5,1.25) -- (0.75,0);

%Right branch
\draw (0,4.25) -- (2,2.5);
\draw (2,2.5) -- (2.5,1.25);
\draw (2,2.5) -- (1.5,1.25);
\draw (1.5,1.25) -- (2.5,1.25);
\draw (1.5,1.25) -- (1.25,0);
\draw (2.5,1.25) -- (2.75,0);

%Root
\draw [fill=violet] (0,4.25) circle (0.15);
\node[above] at (0,4.25) {$Y_{s}$}; 

%Left branch
\draw [fill=cyan] (-2,2.5) circle (0.15);
\draw [fill=white] (-1.5,1.25) circle (0.15);
\draw [fill=white] (-2.5,1.25) circle (0.15);
\draw [fill=white] (-2.75,0) circle (0.15);
\draw [fill=white] (-2.25,0) circle (0.15);
\draw [fill=white] (-1.75,0) circle (0.15);
\draw [fill=white] (-1.25,0) circle (0.15);

%Middle branch
\draw [fill=lime] (0,2.5) circle (0.15);
\draw [fill=white] (-0.5,1.25) circle (0.15);
\draw [fill=white] (0.5,1.25) circle (0.15);
\draw [fill=white] (0.75,0) circle (0.15);
\draw [fill=white] (0.25,0) circle (0.15);
\draw [fill=white] (-0.25,0) circle (0.15);
\draw [fill=white] (-0.75,0) circle (0.15);

%Right branch
\draw [fill=olive] (2,2.5) circle (0.15);
\draw [fill=white] (1.5,1.25) circle (0.15);
\draw [fill=white] (2.5,1.25) circle (0.15);
\draw [fill=white] (2.75,0) circle (0.15);
\draw [fill=white] (1.25,0) circle (0.15);

\end{scope}

\begin{scope}[shift={(15,0)}]

%Stub
\draw [line width=2, green] (2.5,1.25) -- (2.75, 0);
\draw [line width=2, green] (20/8,24/40) -- (22/8,26/40);
\draw (20/8,24/40) -- (22/8,26/40);
\node[right] at (21/8, 5/8) { {\tiny $\{\sigma_{z, 2}, \sigma_{v, 1}\}$} };

%Bottom attachment

\draw [color=brown!90!black] (2.25,0) -- (2.25,-1.5);
\draw [color=brown!90!black] (2.25,0) .. controls (2,-1) and (0.5,-1) .. (0.25, 0);
\draw [color=brown!90!black, fill=white] (2.25,-1.5) circle (0.15);
\node[left] at (2,-1.1) { {\tiny $V^{-}(\sigma_{z,2}, \sigma_{v,1}, E_{s})$} };

%Left branch
\draw (0,4.25) -- (-2,2.5);
\draw (-2,2.5) -- (-2.5,1.25);
\draw (-2,2.5) -- (-1.5,1.25);
\draw (-1.5,1.25) -- (-1.75,0);
\draw (-1.5,1.25) -- (-1.25,0);
\draw (-2.5,1.25) -- (-2.25,0);
\draw (-2.5,1.25) -- (-2.75,0);

%Middle branch
\draw (0,4.25) -- (0,2.5);
\draw (0,2.5) -- (-0.55,1.25);
\draw (0,2.5) -- (0.5,1.25);
\draw (-0.5,1.25) -- (-0.75,0);
\draw (-0.5,1.25) -- (-0.25,0);
\draw (0.5,1.25) -- (0.25,0);
\draw (0.5,1.25) -- (0.75,0);

%Right branch
\draw (0,4.25) -- (1, 27/8);
\draw (1, 27/8) -- (2,2.5);
\draw (2,2.5) -- (2.5,1.25);
\draw (2,2.5) -- (1.5,1.25);
\draw (1.5,1.25) -- (1.75,0);
\draw (1.5,1.25) -- (1.25,0);
\draw (2.5,1.25) -- (2.25,0);
\draw (2.5,1.25) -- (2.75,0);

%Root
\draw [fill=blue] (0,4.25) circle (0.15);
\node[above] at (0,4.25) {$X_{t}$}; 

%Left branch
\draw [fill=red] (-2,2.5) circle (0.15);
\draw [fill=white] (-1.5,1.25) circle (0.15);
\draw [fill=white] (-2.5,1.25) circle (0.15);
\draw [fill=white] (-2.75,0) circle (0.15);
\draw [fill=white] (-2.25,0) circle (0.15);
\draw [fill=white] (-1.75,0) circle (0.15);
\draw [fill=white] (-1.25,0) circle (0.15);

%Middle branch
\draw [fill=purple!70!white] (0,2.5) circle (0.15);
\draw [fill=white] (-0.5,1.25) circle (0.15);
\draw [fill=white] (0.5,1.25) circle (0.15);
\draw [fill=white] (0.75,0) circle (0.15);
\draw [fill=white] (0.25,0) circle (0.15);
\draw [fill=white] (-0.25,0) circle (0.15);
\draw [fill=white] (-0.75,0) circle (0.15);

%Right branch
\draw [fill=orange] (2,2.5) circle (0.15);
\draw [fill=pink] (1.5,1.25) circle (0.15);
\draw [fill=yellow] (2.5,1.25) circle (0.15);
\draw [fill=white] (2.75,0) circle (0.15);
\draw [fill=white] (2.25,0) circle (0.15);
\draw [fill=white] (1.75,0) circle (0.15);
\draw [fill=white] (1.25,0) circle (0.15);

\node at (2.5, 1.25) { {\tiny $z$} };

\begin{scope}[shift={(2.75,-4.25)}]

%Middle branch
\draw (0,4.25) -- (0,2.5);
\draw (0,2.5) -- (-0.55,1.25);
\draw (0,2.5) -- (0.5,1.25);
\draw (-0.5,1.25) -- (-0.75,0);
\draw (-0.5,1.25) -- (-0.25,0);
\draw (0.5,1.25) -- (0.25,0);
\draw (0.5,1.25) -- (0.75,0);

%Right branch
\draw (0,4.25) -- (2,2.5);
\draw (2,2.5) -- (2.5,1.25);
\draw (2,2.5) -- (1.5,1.25);
\draw (1.5,1.25) -- (2.5,1.25);
\draw (1.5,1.25) -- (1.25,0);
\draw (2.5,1.25) -- (2.75,0);

%Root
\draw [fill=violet] (0,4.25) circle (0.15);
\node[right] at (0,4.25) {$Y_{t}$}; 

%Middle branch
\draw [fill=lime] (0,2.5) circle (0.15);
\draw [fill=white] (-0.5,1.25) circle (0.15);
\draw [fill=white] (0.5,1.25) circle (0.15);
\draw [fill=white] (0.75,0) circle (0.15);
\draw [fill=white] (0.25,0) circle (0.15);
\draw [fill=white] (-0.25,0) circle (0.15);
\draw [fill=white] (-0.75,0) circle (0.15);

%Right branch
\draw [fill=olive] (2,2.5) circle (0.15);
\draw [fill=white] (1.5,1.25) circle (0.15);
\draw [fill=white] (2.5,1.25) circle (0.15);
\draw [fill=white] (2.75,0) circle (0.15);
\draw [fill=white] (1.25,0) circle (0.15);

\end{scope}

\begin{scope}[shift = {(5, 02.25)}]

%Left branch
\draw [color=gray] (0,2.5) -- (0.5,1.25);
\draw [color=gray] (0,2.5) -- (-0.5,1.25);
\draw [color=gray] (0.5,1.25) -- (0.75,0);
\draw [color=gray] (0.5,1.25) -- (0.25,0);
\draw [color=gray] (-0.5,1.25) -- (-0.25,0);
\draw [color=gray] (-0.5,1.25) -- (-0.75,0);

%Left branch
\draw [color=gray, fill=cyan] (0,2.5) circle (0.15);
\draw [color=gray, fill=white] (0.5,1.25) circle (0.15);
\draw [color=gray, fill=white] (-0.5,1.25) circle (0.15);
\draw [color=gray, fill=white] (-0.75,0) circle (0.15);
\draw [color=gray, fill=white] (-0.25,0) circle (0.15);
\draw [color=gray, fill=white] (0.75,0) circle (0.15);
\draw [color=gray, fill=white] (0.25,0) circle (0.15);

\end{scope}

\end{scope}

\end{tikzpicture}

\caption{Visual representation of the transition in step (6-b), where an arrival of both  $\hat{\mathfrak{T}}^{\rm dyn}_{z,2}$ and $\hat{\mathfrak{T}}^{\rm dyn}_{v,1}$ is portrayed, with $v=Y_s$. Edges represented in gray in the second plot are those that are not part of $\cB_{\hslash,t}(X_t)\cup\cB_{\hslash,t}(Y_t) $, and $\hslash = 3$. Edges represented in brown are those in $V^-(\sigma_{z,2}, \sigma_{v,1}, E_s)$.}
\label{fig:explorationV-}
\end{figure}
%%%%%%%%%%%%%%%%%%%%%%%%%%%%%%%%%%%%%%%%%%%%%%%%%%%%%%%%%%
	
%%%%%%%%%%%%%%%%%%%%%%%%%%%%%%%%%%%%%%%%%%%%%%%%%%%%%%%%%%
\begin{figure}[!ht]
\centering
\begin{tikzpicture}[scale=0.5]

%Stub
\draw [line width=2, green] (0,4.25) -- (1, 27/8);
\draw [line width=2, green] (6/8,173/56) -- (10/8,205/56);
\draw (6/8,173/56) -- (10/8,205/56);
\node[right] at (1, 27/8) {{\tiny $\sigma_{z, 3}$}};

%Left branch
\draw (0,4.25) -- (-2,2.5);
\draw (-2,2.5) -- (-2.5,1.25);
\draw (-2,2.5) -- (-1.5,1.25);
\draw (-1.5,1.25) -- (-1.75,0);
\draw (-1.5,1.25) -- (-1.25,0);
\draw (-2.5,1.25) -- (-2.25,0);
\draw (-2.5,1.25) -- (-2.75,0);

%Middle branch
\draw (0,4.25) -- (0,2.5);
\draw (0,2.5) -- (-0.55,1.25);
\draw (0,2.5) -- (0.5,1.25);
\draw (-0.5,1.25) -- (-0.75,0);
\draw (-0.5,1.25) -- (-0.25,0);
\draw (0.5,1.25) -- (0.25,0);
\draw (0.5,1.25) -- (0.75,0);

%Right branch
\draw (0,4.25) -- (2,2.5);
\draw (2,2.5) -- (2.5,1.25);
\draw (2,2.5) -- (1.5,1.25);
\draw (1.5,1.25) -- (1.75,0);
\draw (1.5,1.25) -- (1.25,0);
\draw (2.5,1.25) -- (2.25,0);
\draw (2.5,1.25) -- (2.75,0);

%Root
\draw [fill=blue] (0,4.25) circle (0.15);
\node[above] at (0,4.25) {$X_{s}$}; 

%Left branch
\draw [fill=red] (-2,2.5) circle (0.15);
\draw [fill=white] (-1.5,1.25) circle (0.15);
\draw [fill=white] (-2.5,1.25) circle (0.15);
\draw [fill=white] (-2.75,0) circle (0.15);
\draw [fill=white] (-2.25,0) circle (0.15);
\draw [fill=white] (-1.75,0) circle (0.15);
\draw [fill=white] (-1.25,0) circle (0.15);

%Middle branch
\draw [fill=purple!70!white] (0,2.5) circle (0.15);
\draw [fill=white] (-0.5,1.25) circle (0.15);
\draw [fill=white] (0.5,1.25) circle (0.15);
\draw [fill=white] (0.75,0) circle (0.15);
\draw [fill=white] (0.25,0) circle (0.15);
\draw [fill=white] (-0.25,0) circle (0.15);
\draw [fill=white] (-0.75,0) circle (0.15);

%Right branch
\draw [fill=orange] (2,2.5) circle (0.15);
\draw [fill=white] (1.5,1.25) circle (0.15);
\draw [fill=white] (2.5,1.25) circle (0.15);
\draw [fill=white] (2.75,0) circle (0.15);
\draw [fill=white] (2.25,0) circle (0.15);
\draw [fill=white] (1.75,0) circle (0.15);
\draw [fill=white] (1.25,0) circle (0.15);

\begin{scope}[shift={(7,0)}]

%Stub
\draw [line width=2, green] (-2,2.5) -- (-14/8, 15/8);
\draw [line width=2, green] (-12/8,79/40) -- (-16/8,71/40);
\draw (-12/8,79/40) -- (-16/8,71/40);
\node[right] at (-14/8, 15/8) {{\tiny $\sigma_{v,2}$}};

%Left branch
\draw (0,4.25) -- (-2,2.5);
\draw (-2,2.5) -- (-2.5,1.25);
\draw (-2,2.5) -- (-1.5,1.25);
\draw (-1.5,1.25) -- (-1.75,0);
\draw (-1.5,1.25) -- (-1.25,0);
\draw (-2.5,1.25) -- (-2.25,0);
\draw (-2.5,1.25) -- (-2.75,0);

%Middle branch
\draw (0,4.25) -- (0,2.5);
\draw (0,2.5) -- (-0.55,1.25);
\draw (0,2.5) -- (0.5,1.25);
\draw (-0.5,1.25) -- (-0.75,0);
\draw (-0.5,1.25) -- (-0.25,0);
\draw (0.5,1.25) -- (0.25,0);
\draw (0.5,1.25) -- (0.75,0);

%Right branch
\draw (0,4.25) -- (2,2.5);
\draw (2,2.5) -- (2.5,1.25);
\draw (2,2.5) -- (1.5,1.25);
\draw (1.5,1.25) -- (2.5,1.25);
\draw (1.5,1.25) -- (1.25,0);
\draw (2.5,1.25) -- (2.75,0);

%Root
\draw [fill=violet] (0,4.25) circle (0.15);
\node[above] at (0,4.25) {$Y_{s}$}; 

%Left branch
\draw [fill=cyan] (-2,2.5) circle (0.15);
\draw [fill=yellow] (-1.5,1.25) circle (0.15);
\draw [fill=pink] (-2.5,1.25) circle (0.15);
\draw [fill=white] (-2.75,0) circle (0.15);
\draw [fill=white] (-2.25,0) circle (0.15);
\draw [fill=white] (-1.75,0) circle (0.15);
\draw [fill=white] (-1.25,0) circle (0.15);

%Middle branch
\draw [fill=lime] (0,2.5) circle (0.15);
\draw [fill=white] (-0.5,1.25) circle (0.15);
\draw [fill=white] (0.5,1.25) circle (0.15);
\draw [fill=white] (0.75,0) circle (0.15);
\draw [fill=white] (0.25,0) circle (0.15);
\draw [fill=white] (-0.25,0) circle (0.15);
\draw [fill=white] (-0.75,0) circle (0.15);

%Right branch
\draw [fill=olive] (2,2.5) circle (0.15);
\draw [fill=white] (1.5,1.25) circle (0.15);
\draw [fill=white] (2.5,1.25) circle (0.15);
\draw [fill=white] (2.75,0) circle (0.15);
\draw [fill=white] (1.25,0) circle (0.15);

\end{scope}

\begin{scope}[shift={(15,0)}]

%Stub
\draw [line width=2, green] (0,4.25) -- (2,2.5);
\draw [line width=2, green] (6/8,173/56) -- (10/8,205/56);
\draw (6/8,173/56) -- (10/8,205/56);
\node[right] at (1, 27/8) {{\tiny $\{\sigma_{z, 3}, \sigma_{v,2}\}$}};

%Left branch
\draw (0,4.25) -- (-2,2.5);
\draw (-2,2.5) -- (-2.5,1.25);
\draw (-2,2.5) -- (-1.5,1.25);
\draw (-1.5,1.25) -- (-1.75,0);
\draw (-1.5,1.25) -- (-1.25,0);
\draw (-2.5,1.25) -- (-2.25,0);
\draw (-2.5,1.25) -- (-2.75,0);

%Middle branch
\draw (0,4.25) -- (0,2.5);
\draw (0,2.5) -- (-0.55,1.25);
\draw (0,2.5) -- (0.5,1.25);
\draw (-0.5,1.25) -- (-0.75,0);
\draw (-0.5,1.25) -- (-0.25,0);
\draw (0.5,1.25) -- (0.25,0);
\draw (0.5,1.25) -- (0.75,0);

%Right branch
\draw (0,4.25) -- (1, 27/8);
\draw (1, 27/8) -- (2,2.5);
\draw (2,2.5) -- (2.5,1.25);
\draw (2,2.5) -- (1.5,1.25);
\draw (1.5,1.25) -- (1.75,0);
\draw (1.5,1.25) -- (1.25,0);
\draw (2.5,1.25) -- (2.25,0);
\draw (2.5,1.25) -- (2.75,0);

%Bottom attachments
\draw (2.75,0) -- (2.75,-1.5);
\draw (2.75,0) -- (3.25,-1.5);
\draw (3.25,-1.5) -- (2.75,-1.5);
\draw (2.25,0) -- (1.75,-1.5);
\draw (2.25,0) -- (2.25,-1.5);
\draw (2.75,-1.5) -- (3.25,-3);
\draw (3.25,-1.5) -- (3.75,-3);
\draw (1.75,-1.5) -- (1.6,-3);
\draw (1.75,-1.5) -- (1.25,-3);
\draw (2.25,-1.5) -- (2.5,-3);
\draw (2.25,-1.5) -- (2,-3);

\draw [fill=white] (2.75,-1.5) circle (0.15);
\draw [fill=white] (3.25,-1.5) circle (0.15);
\draw [fill=white] (1.75,-1.5) circle (0.15);
\draw [fill=white] (2.25,-1.5) circle (0.15);
\draw [fill=white] (3.75,-3) circle (0.15);
\draw [fill=white] (3.25,-3) circle (0.15);
\draw [fill=white] (2,-3) circle (0.15);
\draw [fill=white] (2.5,-3) circle (0.15);
\draw [fill=white] (1.25,-3) circle (0.15);
\draw [fill=white] (1.6,-3) circle (0.15);

%Root
\draw [fill=blue] (0,4.25) circle (0.15);
\node[above] at (0,4.25) {$X_{t}$}; 

%Left branch
\draw [fill=red] (-2,2.5) circle (0.15);
\draw [fill=white] (-1.5,1.25) circle (0.15);
\draw [fill=white] (-2.5,1.25) circle (0.15);
\draw [fill=white] (-2.75,0) circle (0.15);
\draw [fill=white] (-2.25,0) circle (0.15);
\draw [fill=white] (-1.75,0) circle (0.15);
\draw [fill=white] (-1.25,0) circle (0.15);

%Middle branch
\draw [fill=purple!70!white] (0,2.5) circle (0.15);
\draw [fill=white] (-0.5,1.25) circle (0.15);
\draw [fill=white] (0.5,1.25) circle (0.15);
\draw [fill=white] (0.75,0) circle (0.15);
\draw [fill=white] (0.25,0) circle (0.15);
\draw [fill=white] (-0.25,0) circle (0.15);
\draw [fill=white] (-0.75,0) circle (0.15);

%Right branch
\draw [fill=cyan] (2,2.5) circle (0.15);
\draw [fill=pink] (1.5,1.25) circle (0.15);
\draw [fill=violet] (2.5,1.25) circle (0.15);
\draw [fill=olive] (2.75,0) circle (0.15);
\draw [fill=lime] (2.25,0) circle (0.15);
\draw [fill=white] (1.75,0) circle (0.15);
\draw [fill=white] (1.25,0) circle (0.15);

\node[right] at (2.5,1.25) {$Y_{t}$};

\begin{scope}[shift = {(5, 02.25)}]

%Left branch
\draw [color=gray] (0,2.5) -- (0.5,1.25);
\draw [color=gray] (0,2.5) -- (-0.5,1.25);
\draw [color=gray] (0.5,1.25) -- (0.75,0);
\draw [color=gray] (0.5,1.25) -- (0.25,0);
\draw [color=gray] (-0.5,1.25) -- (-0.25,0);
\draw [color=gray] (-0.5,1.25) -- (-0.75,0);

%Left branch
\draw [color=gray, fill=orange] (0,2.5) circle (0.15);
\draw [color=gray, fill=white] (0.5,1.25) circle (0.15);
\draw [color=gray, fill=white] (-0.5,1.25) circle (0.15);
\draw [color=gray, fill=white] (-0.75,0) circle (0.15);
\draw [color=gray, fill=white] (-0.25,0) circle (0.15);
\draw [color=gray, fill=white] (0.75,0) circle (0.15);
\draw [color=gray, fill=white] (0.25,0) circle (0.15);

\end{scope}

\begin{scope}[shift = {(5, 0)}]

%Left branch
\draw [color=gray] (0,1.25) -- (0.25,0);
\draw [color=gray] (0,1.25) -- (-0.25,0);

%Left branch
\draw [color=gray, fill=yellow] (0,1.25) circle (0.15);
\draw [color=gray, fill=white] (0.25,0) circle (0.15);
\draw [color=gray, fill=white] (-0.25,0) circle (0.15);

\end{scope}

\end{scope}

\end{tikzpicture}

\caption{Visual representation of the transition in step (6-b), where an arrival of both $\hat{\mathfrak{T}}^{\rm dyn}_{z,3}$ and $\hat{\mathfrak{T}}^{\rm dyn}_{v,2}$is portrayed, where $z=X_s$. In this case, the set $V^{-}$ is empty. Edges represented in gray in the second plot are those that are not part of $\cB_{\hslash,t}(X_t)\cup\cB_{\hslash,t}(Y_t) $, where $\hslash = 3$.}
\label{fig:explorationV0}
\end{figure}
%%%%%%%%%%%%%%%%%%%%%%%%%%%%%%%%%%%%%%%%%%%%%%%%%%%%%%%%%%%%	
	
%%%%%%%%%%%%%%%%%%%%%%%%%%%%%%%%%%%%%%%%%%%%%%%%%%%%%%%%%%%%	
\begin{figure}[!ht]
\centering
\begin{tikzpicture}[scale=0.5]

%Stub
\draw [line width=2, green] (0,4.25) -- (1, 27/8);
\draw [line width=2, green] (6/8,173/56) -- (10/8,205/56);
\draw (6/8,173/56) -- (10/8,205/56);
\node[right] at (1, 27/8) {{\tiny $\sigma_{z, 3}$}};

%Left branch
\draw (0,4.25) -- (-2,2.5);
\draw (-2,2.5) -- (-2.5,1.25);
\draw (-2,2.5) -- (-1.5,1.25);
\draw (-1.5,1.25) -- (-1.75,0);
\draw (-1.5,1.25) -- (-1.25,0);
\draw (-2.5,1.25) -- (-2.25,0);
\draw (-2.5,1.25) -- (-2.75,0);

%Middle branch
\draw (0,4.25) -- (0,2.5);
\draw (0,2.5) -- (-0.55,1.25);
\draw (0,2.5) -- (0.5,1.25);
\draw (-0.5,1.25) -- (-0.75,0);
\draw (-0.5,1.25) -- (-0.25,0);
\draw (0.5,1.25) -- (0.25,0);
\draw (0.5,1.25) -- (0.75,0);

%Right branch
\draw (0,4.25) -- (2,2.5);
\draw (2,2.5) -- (2.5,1.25);
\draw (2,2.5) -- (1.5,1.25);
\draw (1.5,1.25) -- (1.75,0);
\draw (1.5,1.25) -- (1.25,0);
\draw (2.5,1.25) -- (2.25,0);
\draw (2.5,1.25) -- (2.75,0);

%Root
\draw [fill=blue] (0,4.25) circle (0.15);
\node[above] at (0,4.25) {$X_{s}$}; 

%Left branch
\draw [fill=red] (-2,2.5) circle (0.15);
\draw [fill=white] (-1.5,1.25) circle (0.15);
\draw [fill=white] (-2.5,1.25) circle (0.15);
\draw [fill=white] (-2.75,0) circle (0.15);
\draw [fill=white] (-2.25,0) circle (0.15);
\draw [fill=white] (-1.75,0) circle (0.15);
\draw [fill=white] (-1.25,0) circle (0.15);

%Middle branch
\draw [fill=purple!70!white] (0,2.5) circle (0.15);
\draw [fill=white] (-0.5,1.25) circle (0.15);
\draw [fill=white] (0.5,1.25) circle (0.15);
\draw [fill=white] (0.75,0) circle (0.15);
\draw [fill=white] (0.25,0) circle (0.15);
\draw [fill=white] (-0.25,0) circle (0.15);
\draw [fill=white] (-0.75,0) circle (0.15);

%Right branch
\draw [fill=orange] (2,2.5) circle (0.15);
\draw [fill=white] (1.5,1.25) circle (0.15);
\draw [fill=white] (2.5,1.25) circle (0.15);
\draw [fill=white] (2.75,0) circle (0.15);
\draw [fill=white] (2.25,0) circle (0.15);
\draw [fill=white] (1.75,0) circle (0.15);
\draw [fill=white] (1.25,0) circle (0.15);

\begin{scope}[shift={(7,0)}]

%Stub
\draw [line width=2, green] (-2.75,0) -- (-23/8, -5/8);
\draw [line width=2, green] (-22/8,-26/40) -- (-24/8,-24/40);
\draw (-22/8,-26/40) -- (-24/8,-24/40);
\draw (-2.75,0) -- (-23/8, -5/8);
\node[left] at (-23/8, -5/8) {{\tiny $\sigma_{v,1}$}};

%Left branch
\draw (0,4.25) -- (-2,2.5);
\draw (-2,2.5) -- (-2.5,1.25);
\draw (-2,2.5) -- (-1.5,1.25);
\draw (-1.5,1.25) -- (-1.75,0);
\draw (-1.5,1.25) -- (-1.25,0);
\draw (-2.5,1.25) -- (-2.25,0);
\draw (-2.5,1.25) -- (-2.75,0);

%Middle branch
\draw (0,4.25) -- (0,2.5);
\draw (0,2.5) -- (-0.55,1.25);
\draw (0,2.5) -- (0.5,1.25);
\draw (-0.5,1.25) -- (-0.75,0);
\draw (-0.5,1.25) -- (-0.25,0);
\draw (0.5,1.25) -- (0.25,0);
\draw (0.5,1.25) -- (0.75,0);

%Right branch
\draw (0,4.25) -- (2,2.5);
\draw (2,2.5) -- (2.5,1.25);
\draw (2,2.5) -- (1.5,1.25);
\draw (1.5,1.25) -- (2.5,1.25);
\draw (1.5,1.25) -- (1.25,0);
\draw (2.5,1.25) -- (2.75,0);

%Root
\draw [fill=violet] (0,4.25) circle (0.15);
\node[above] at (0,4.25) {$Y_{s}$}; 

%Left branch
\draw [fill=cyan] (-2,2.5) circle (0.15);
\draw [fill=white] (-1.5,1.25) circle (0.15);
\draw [fill=pink] (-2.5,1.25) circle (0.15);
\draw [fill=yellow] (-2.75,0) circle (0.15);
\draw [fill=teal] (-2.25,0) circle (0.15);
\draw [fill=white] (-1.75,0) circle (0.15);
\draw [fill=white] (-1.25,0) circle (0.15);

%Middle branch
\draw [fill=lime] (0,2.5) circle (0.15);
\draw [fill=white] (-0.5,1.25) circle (0.15);
\draw [fill=white] (0.5,1.25) circle (0.15);
\draw [fill=white] (0.75,0) circle (0.15);
\draw [fill=white] (0.25,0) circle (0.15);
\draw [fill=white] (-0.25,0) circle (0.15);
\draw [fill=white] (-0.75,0) circle (0.15);

%Right branch
\draw [fill=olive] (2,2.5) circle (0.15);
\draw [fill=white] (1.5,1.25) circle (0.15);
\draw [fill=white] (2.5,1.25) circle (0.15);
\draw [fill=white] (2.75,0) circle (0.15);
\draw [fill=white] (1.25,0) circle (0.15);

\end{scope}

\begin{scope}[shift={(15,0)}]

%Stub
\draw [line width=2, green] (0,4.25) -- (2,2.5);
\draw [line width=2, green] (6/8,173/56) -- (10/8,205/56);
\draw (6/8,173/56) -- (10/8,205/56);
\node[right] at (1, 27/8) {{\tiny $\{\sigma_{z, 3}, \sigma_{v,1}\}$}};

%Left branch
\draw (0,4.25) -- (-2,2.5);
\draw (-2,2.5) -- (-2.5,1.25);
\draw (-2,2.5) -- (-1.5,1.25);
\draw (-1.5,1.25) -- (-1.75,0);
\draw (-1.5,1.25) -- (-1.25,0);
\draw (-2.5,1.25) -- (-2.25,0);
\draw (-2.5,1.25) -- (-2.75,0);

%Middle branch
\draw (0,4.25) -- (0,2.5);
\draw (0,2.5) -- (-0.55,1.25);
\draw (0,2.5) -- (0.5,1.25);
\draw (-0.5,1.25) -- (-0.75,0);
\draw (-0.5,1.25) -- (-0.25,0);
\draw (0.5,1.25) -- (0.25,0);
\draw (0.5,1.25) -- (0.75,0);

%Right branch
\draw (0,4.25) -- (1, 27/8);
\draw (1, 27/8) -- (2,2.5);
\draw [color=brown!90!black] (2,2.5) -- (2.5,1.25);
\draw (2,2.5) -- (1.5,1.25);
\draw (1.5,1.25) -- (1.75,0);
\draw (1.5,1.25) -- (1.25,0);
\draw [color=brown!90!black] (2.5,1.25) -- (2.25,0);
\draw [color=brown!90!black] (2.5,1.25) -- (2.75,0);

\begin{scope}[shift = {(-0.5, -1.25)}]

%Bottom attachments
\draw (1.75, 1.25) -- (1.25,0);
\draw (1.25,0) -- (1,-1.25);
\draw (1.25,0) -- (1.5,-1.25);
\draw (1.5,-1.25) -- (1.5,-2.75);
\draw (1.5,-1.25) -- (2,-2.75);
\draw (2,-2.75) -- (1.5,-2.75);
\draw (1,-1.25) -- (0.5,-2.75);
\draw (1,-1.25) -- (1,-2.75);
\draw (1.5,-2.75) -- (2,-4.25);
\draw (2,-2.75) -- (2.5,-4.25);
\draw (0.5,-2.75) -- (0.35,-4.25);
\draw (0.5,-2.75) -- (0,-4.25);
\draw (1,-2.75) -- (1.25,-4.25);
\draw (1,-2.75) -- (0.75,-4.25);

\draw (1.75, 1.25) -- (2.25,0);
\draw (2.25,0) -- (2, -1.25);
\draw (2.25,0) -- (2.5, -1.25);

\draw [fill=cyan] (1.75,1.25) circle (0.15);
\draw [fill=violet] (1.25,0) circle (0.15);
\draw [fill=olive] (1.5,-1.25) circle (0.15);
\draw [fill=lime] (1,-1.25) circle (0.15);
\draw [fill=white] (1.5,-2.75) circle (0.15);
\draw [fill=white] (2,-2.75) circle (0.15);
\draw [fill=white] (0.5,-2.75) circle (0.15);
\draw [fill=white] (1,-2.75) circle (0.15);
\draw [fill=white] (2,-4.25) circle (0.15);
\draw [fill=white] (2.5,-4.25) circle (0.15);
\draw [fill=white] (0.75,-4.25) circle (0.15);
\draw [fill=white] (1.25,-4.25) circle (0.15);
\draw [fill=white] (0,-4.25) circle (0.15);
\draw [fill=white] (0.35,-4.25) circle (0.15);

\draw [fill=white] (2.25,0) circle (0.15);
\draw [fill=white] (2, -1.25) circle (0.15);
\draw [fill=white] (2.5, -1.25) circle (0.15);

\node[left] at (1.25, 0) {$Y_{t}$};

\end{scope}

%Root
\draw [fill=blue] (0,4.25) circle (0.15);
\node[above] at (0,4.25) {$X_{t}$}; 

%Left branch
\draw [fill=red] (-2,2.5) circle (0.15);
\draw [fill=white] (-1.5,1.25) circle (0.15);
\draw [fill=white] (-2.5,1.25) circle (0.15);
\draw [fill=white] (-2.75,0) circle (0.15);
\draw [fill=white] (-2.25,0) circle (0.15);
\draw [fill=white] (-1.75,0) circle (0.15);
\draw [fill=white] (-1.25,0) circle (0.15);

%Middle branch
\draw [fill=purple!70!white] (0,2.5) circle (0.15);
\draw [fill=white] (-0.5,1.25) circle (0.15);
\draw [fill=white] (0.5,1.25) circle (0.15);
\draw [fill=white] (0.75,0) circle (0.15);
\draw [fill=white] (0.25,0) circle (0.15);
\draw [fill=white] (-0.25,0) circle (0.15);
\draw [fill=white] (-0.75,0) circle (0.15);

%Right branch
\draw [fill=yellow] (2,2.5) circle (0.15);
\draw [fill=pink] (1.5,1.25) circle (0.15);
\draw [color=brown!90!black, fill=white] (2.5,1.25) circle (0.15);
\draw [color=brown!90!black, fill=white] (2.75,0) circle (0.15);
\draw [color=brown!90!black, fill=white] (2.25,0) circle (0.15);
\draw [fill=teal] (1.75,0) circle (0.15);

\draw [decorate,decoration={brace,amplitude=5pt}] (3,2.5) -- (3,0) node[midway, right] { {\tiny $ \,\, V^{-}(\sigma_{z,3}, \sigma_{v,1}, E_{s})$} };

%(1.1,-0.3) -- (2.9,-0.3)

\begin{scope}[shift = {(5, 2.25)}]

%Left branch
\draw [color=gray] (0.2,2.5) -- (0.7,1.25);
\draw [color=gray] (0.2,2.5) -- (-0.3,1.25);
\draw [color=gray] (0.7,1.25) -- (0.95,0);
\draw [color=gray] (0.7,1.25) -- (0.45,0);
\draw [color=gray] (-0.3,1.25) -- (-0.05,0);
\draw [color=gray] (-0.3,1.25) -- (-0.55,0);

%Left branch
\draw [color=gray, fill=orange] (0.2,2.5) circle (0.15);
\draw [color=gray, fill=white] (0.7,1.25) circle (0.15);
\draw [color=gray, fill=white] (-0.3,1.25) circle (0.15);
\draw [color=gray, fill=white] (-0.55,0) circle (0.15);
\draw [color=gray, fill=white] (-0.05,0) circle (0.15);
\draw [color=gray, fill=white] (0.95,0) circle (0.15);
\draw [color=gray, fill=white] (0.45,0) circle (0.15);

\end{scope}

\end{scope}

\end{tikzpicture}

\caption{Visual representation of the transition in step (6-b), where an arrival of both $\hat{\mathfrak{T}}^{\rm dyn}_{z,3}$ and $\hat{\mathfrak{T}}^{\rm dyn}_{v,1}$ is portrayed, where $z=X_s$. In this case $\sigma_{v,1}$ is not matched in $E_s$, i.e., $\sigma_{v,1} \in \bar E_s\setminus E_s$. Edges represented in gray in the second plot are those that are not part of $\cB_{\hslash,t}(X_t)\cup\cB_{\hslash,t}(Y_t)$, and $\hslash = 3$. Edges represented in brown are those in $V^-(\sigma_{z,3},\sigma_{v,1},E_s)$.}
\label{fig:explorationSigma-}
\end{figure}
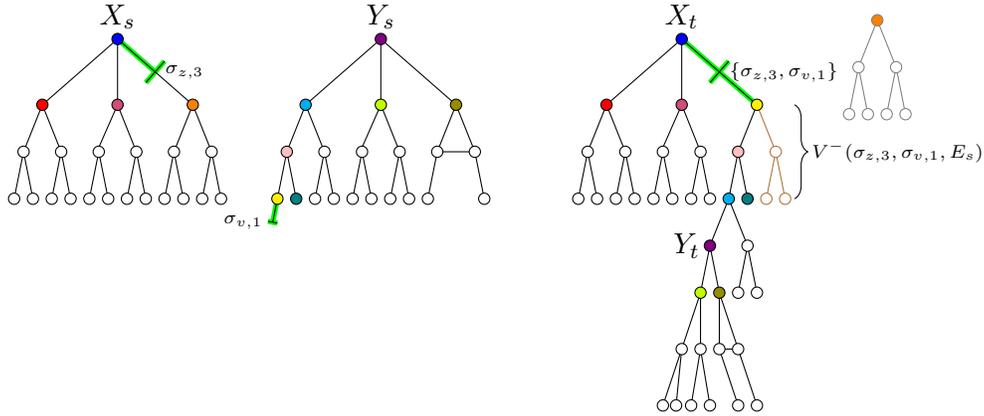
%%%%%%%%%%%%%%%%%%%%%%%%%%%%%%%%%%%%%%%%%%%%%%%%%%%%%%%%%%%	
	
	%%%%%%%%%%%%%%%%%%%%%%%%%%%%%%%%%%%%%%%%%%%%%%%%%%%%%%
	\begin{figure}[!ht]
		\centering
		\begin{tikzpicture}[scale=0.5]
			
			%Stub
			\draw [line width=2, green] (0,4.25) -- (1, 27/8);
			\draw [line width=2, green] (6/8,173/56) -- (10/8,205/56);
			\draw (6/8,173/56) -- (10/8,205/56);
			\node[right] at (1, 27/8) {{\tiny $\sigma_{z, 3}$}};
			
			%Left branch
			\draw (0,4.25) -- (-2,2.5);
			\draw (-2,2.5) -- (-2.5,1.25);
			\draw (-2,2.5) -- (-1.5,1.25);
			\draw (-1.5,1.25) -- (-1.75,0);
			\draw (-1.5,1.25) -- (-1.25,0);
			\draw (-2.5,1.25) -- (-2.25,0);
			\draw (-2.5,1.25) -- (-2.75,0);
			
			%Middle branch
			\draw (0,4.25) -- (0,2.5);
			\draw (0,2.5) -- (-0.55,1.25);
			\draw (0,2.5) -- (0.5,1.25);
			\draw (-0.5,1.25) -- (-0.75,0);
			\draw (-0.5,1.25) -- (-0.25,0);
			\draw (0.5,1.25) -- (0.25,0);
			\draw (0.5,1.25) -- (0.75,0);
			
			%Right branch
			\draw (0,4.25) -- (2,2.5);
			\draw (2,2.5) -- (2.5,1.25);
			\draw (2,2.5) -- (1.5,1.25);
			\draw (1.5,1.25) -- (1.75,0);
			\draw (1.5,1.25) -- (1.25,0);
			\draw (2.5,1.25) -- (2.25,0);
			\draw (2.5,1.25) -- (2.75,0);
			
			%Root
			\draw [fill=blue] (0,4.25) circle (0.15);
			\node[above] at (0,4.25) {$X_{s}$}; 
			
			%Left branch
			\draw [fill=red] (-2,2.5) circle (0.15);
			\draw [fill=white] (-1.5,1.25) circle (0.15);
			\draw [fill=white] (-2.5,1.25) circle (0.15);
			\draw [fill=white] (-2.75,0) circle (0.15);
			\draw [fill=white] (-2.25,0) circle (0.15);
			\draw [fill=white] (-1.75,0) circle (0.15);
			\draw [fill=white] (-1.25,0) circle (0.15);
			
			%Middle branch
			\draw [fill=purple!70!white] (0,2.5) circle (0.15);
			\draw [fill=white] (-0.5,1.25) circle (0.15);
			\draw [fill=white] (0.5,1.25) circle (0.15);
			\draw [fill=white] (0.75,0) circle (0.15);
			\draw [fill=white] (0.25,0) circle (0.15);
			\draw [fill=white] (-0.25,0) circle (0.15);
			\draw [fill=white] (-0.75,0) circle (0.15);
			
			%Right branch
			\draw [fill=orange] (2,2.5) circle (0.15);
			\draw [fill=pink] (1.5,1.25) circle (0.15);
			\draw [fill=yellow] (2.5,1.25) circle (0.15);
			\draw [fill=white] (2.75,0) circle (0.15);
			\draw [fill=white] (2.25,0) circle (0.15);
			\draw [fill=white] (1.75,0) circle (0.15);
			\draw [fill=white] (1.25,0) circle (0.15);
			
			\begin{scope}[shift={(6,0)}]
				
				%Left branch
				\draw (0,4.25) -- (-2,2.5);
				\draw (-2,2.5) -- (-2.5,1.25);
				\draw (-2,2.5) -- (-1.5,1.25);
				\draw (-1.5,1.25) -- (-1.75,0);
				\draw (-1.5,1.25) -- (-1.25,0);
				\draw (-2.5,1.25) -- (-2.25,0);
				\draw (-2.5,1.25) -- (-2.75,0);
				
				%Middle branch
				\draw (0,4.25) -- (0,2.5);
				\draw (0,2.5) -- (-0.55,1.25);
				\draw (0,2.5) -- (0.5,1.25);
				\draw (-0.5,1.25) -- (-0.75,0);
				\draw (-0.5,1.25) -- (-0.25,0);
				\draw (0.5,1.25) -- (0.25,0);
				\draw (0.5,1.25) -- (0.75,0);
				
				%Right branch
				\draw (0,4.25) -- (2,2.5);
				\draw (2,2.5) -- (2.5,1.25);
				\draw (2,2.5) -- (1.5,1.25);
				\draw (1.5,1.25) -- (2.5,1.25);
				\draw (1.5,1.25) -- (1.25,0);
				\draw (2.5,1.25) -- (2.75,0);
				
				%Root
				\draw [fill=violet] (0,4.25) circle (0.15);
				\node[above] at (0,4.25) {$Y_{s}$}; 
				
				%Left branch
				\draw [fill=cyan] (-2,2.5) circle (0.15);
				\draw [fill=white] (-1.5,1.25) circle (0.15);
				\draw [fill=white] (-2.5,1.25) circle (0.15);
				\draw [fill=white] (-2.75,0) circle (0.15);
				\draw [fill=white] (-2.25,0) circle (0.15);
				\draw [fill=white] (-1.75,0) circle (0.15);
				\draw [fill=white] (-1.25,0) circle (0.15);
				
				%Middle branch
				\draw [fill=lime] (0,2.5) circle (0.15);
				\draw [fill=white] (-0.5,1.25) circle (0.15);
				\draw [fill=white] (0.5,1.25) circle (0.15);
				\draw [fill=white] (0.75,0) circle (0.15);
				\draw [fill=white] (0.25,0) circle (0.15);
				\draw [fill=white] (-0.25,0) circle (0.15);
				\draw [fill=white] (-0.75,0) circle (0.15);
				
				%Right branch
				\draw [fill=olive] (2,2.5) circle (0.15);
				\draw [fill=white] (1.5,1.25) circle (0.15);
				\draw [fill=white] (2.5,1.25) circle (0.15);
				\draw [fill=white] (2.75,0) circle (0.15);
				\draw [fill=white] (1.25,0) circle (0.15);
				
			\end{scope}
			
			\begin{scope}[shift={(17, 3.5)}]
				
				%Right branch
				\draw [color=gray] (2,2.5) -- (2.5,1.25);
				\draw [color=gray] (2,2.5) -- (1.5,1.25);
				\draw [color=gray] (1.5,1.25) -- (1.75,0);
				\draw [color=gray] (1.5,1.25) -- (1.25,0);
				\draw [color=gray] (2.5,1.25) -- (2.25,0);
				\draw [color=gray] (2.5,1.25) -- (2.75,0);
				
				%Right branch
				\draw [color=gray, fill=orange] (2,2.5) circle (0.15);
				\draw [color=gray, fill=pink] (1.5,1.25) circle (0.15);
				\draw [color=gray, fill=yellow] (2.5,1.25) circle (0.15);
				\draw [color=gray, fill=white] (2.75,0) circle (0.15);
				\draw [color=gray, fill=white] (2.25,0) circle (0.15);
				\draw [color=gray, fill=white] (1.75,0) circle (0.15);
				\draw [color=gray, fill=white] (1.25,0) circle (0.15);
				
			\end{scope}
			
			\begin{scope}[shift={(15,0)}]

				%Stub
				\draw [line width=2, green] (0,4.25) -- (1, 27/8);
				\draw [line width=2, green] (6/8,173/56) -- (10/8,205/56);
				\draw (6/8,173/56) -- (10/8,205/56);
				\node[right] at (1, 27/8) {{\tiny $\sigma_{z, 3}$}};
				
				%Left branch
				\draw (0,4.25) -- (-2,2.5);
				\draw (-2,2.5) -- (-2.5,1.25);
				\draw (-2,2.5) -- (-1.5,1.25);
				\draw (-1.5,1.25) -- (-1.75,0);
				\draw (-1.5,1.25) -- (-1.25,0);
				\draw (-2.5,1.25) -- (-2.25,0);
				\draw (-2.5,1.25) -- (-2.75,0);
				
				%Middle branch
				\draw (0,4.25) -- (0,2.5);
				\draw (0,2.5) -- (-0.55,1.25);
				\draw (0,2.5) -- (0.5,1.25);
				\draw (-0.5,1.25) -- (-0.75,0);
				\draw (-0.5,1.25) -- (-0.25,0);
				\draw (0.5,1.25) -- (0.25,0);
				\draw (0.5,1.25) -- (0.75,0);
				
				%Right branch
				\draw (0,4.25) -- (1, 27/8);
				\draw [color=brown!90!black] (1, 27/8) -- (2,2.5);
				\draw [color=brown!90!black] (2,2.5) -- (2.5,1.25);
				\draw [color=brown!90!black] (2,2.5) -- (1.5,1.25);
				\draw [color=brown!90!black] (1.5,1.25) -- (1.75,0);
				\draw [color=brown!90!black] (1.5,1.25) -- (1.25,0);
				\draw [color=brown!90!black] (2.5,1.25) -- (2.25,0);
				\draw [color=brown!90!black] (2.5,1.25) -- (2.75,0);
				
				%Root
				\draw [fill=orange] (0,4.25) circle (0.15);
				\node[above] at (0,4.25) {$X_{t}$}; 
				
				%Left branch
				\draw [fill=blue] (-2,2.5) circle (0.15);
				\draw [fill=purple!70!white] (-1.5,1.25) circle (0.15);
				\draw [fill=red] (-2.5,1.25) circle (0.15);
				\draw [fill=white] (-2.75,0) circle (0.15);
				\draw [fill=white] (-2.25,0) circle (0.15);
				\draw [fill=white] (-1.75,0) circle (0.15);
				\draw [fill=white] (-1.25,0) circle (0.15);
				
				%Middle branch
				\draw [fill=pink] (0,2.5) circle (0.15);
				\draw [fill=white] (-0.5,1.25) circle (0.15);
				\draw [fill=white] (0.5,1.25) circle (0.15);
				\draw [fill=white] (0.75,0) circle (0.15);
				\draw [fill=white] (0.25,0) circle (0.15);
				\draw [fill=white] (-0.25,0) circle (0.15);
				\draw [fill=white] (-0.75,0) circle (0.15);
				
				%Right branch
				\draw [color=brown!90!black, fill=white] (2,2.5) circle (0.15);
				\draw [color=brown!90!black, fill=white] (1.5,1.25) circle (0.15);
				\draw [color=brown!90!black, fill=white] (2.5,1.25) circle (0.15);
				\draw [color=brown!90!black, fill=white] (2.75,0) circle (0.15);
				\draw [color=brown!90!black, fill=white] (2.25,0) circle (0.15);
				\draw [color=brown!90!black, fill=white] (1.75,0) circle (0.15);
				\draw [color=brown!90!black, fill=white] (1.25,0) circle (0.15);
				
				\draw [decorate,decoration={brace,amplitude=5pt,mirror}] (1.1,-0.3) -- (2.9,-0.3) node[midway, below] { {\tiny $V^{+}(\sigma_{z,3}, E_{s})$} };
				
				\begin{scope}[shift={(6,0)}]
					
					%Left branch
					\draw (0,4.25) -- (-2,2.5);
					\draw (-2,2.5) -- (-2.5,1.25);
					\draw (-2,2.5) -- (-1.5,1.25);
					\draw (-1.5,1.25) -- (-1.75,0);
					\draw (-1.5,1.25) -- (-1.25,0);
					\draw (-2.5,1.25) -- (-2.25,0);
					\draw (-2.5,1.25) -- (-2.75,0);
					
					%Middle branch
					\draw (0,4.25) -- (0,2.5);
					\draw (0,2.5) -- (-0.55,1.25);
					\draw (0,2.5) -- (0.5,1.25);
					\draw (-0.5,1.25) -- (-0.75,0);
					\draw (-0.5,1.25) -- (-0.25,0);
					\draw (0.5,1.25) -- (0.25,0);
					\draw (0.5,1.25) -- (0.75,0);
					
					%Right branch
					\draw (0,4.25) -- (2,2.5);
					\draw (2,2.5) -- (2.5,1.25);
					\draw (2,2.5) -- (1.5,1.25);
					\draw (1.5,1.25) -- (2.5,1.25);
					\draw (1.5,1.25) -- (1.25,0);
					\draw (2.5,1.25) -- (2.75,0);
					
					%Root
					\draw [fill=violet] (0,4.25) circle (0.15);
					\node[above] at (0,4.25) {$Y_{t}$}; 
					
					%Left branch
					\draw [fill=cyan] (-2,2.5) circle (0.15);
					\draw [fill=white] (-1.5,1.25) circle (0.15);
					\draw [fill=white] (-2.5,1.25) circle (0.15);
					\draw [fill=white] (-2.75,0) circle (0.15);
					\draw [fill=white] (-2.25,0) circle (0.15);
					\draw [fill=white] (-1.75,0) circle (0.15);
					\draw [fill=white] (-1.25,0) circle (0.15);
					
					%Middle branch
					\draw [fill=lime] (0,2.5) circle (0.15);
					\draw [fill=white] (-0.5,1.25) circle (0.15);
					\draw [fill=white] (0.5,1.25) circle (0.15);
					\draw [fill=white] (0.75,0) circle (0.15);
					\draw [fill=white] (0.25,0) circle (0.15);
					\draw [fill=white] (-0.25,0) circle (0.15);
					\draw [fill=white] (-0.75,0) circle (0.15);
					
					%Right branch
					\draw [fill=olive] (2,2.5) circle (0.15);
					\draw [fill=white] (1.5,1.25) circle (0.15);
					\draw [fill=white] (2.5,1.25) circle (0.15);
					\draw [fill=white] (2.75,0) circle (0.15);
					\draw [fill=white] (1.25,0) circle (0.15);
					
				\end{scope}
				
			\end{scope}
			
		\end{tikzpicture}
		
		\caption{Visual representation of the transition in step (6-c), where an arrival of $\hat{\mathfrak{T}}^{\rm dyn}_{z,3}$ is portrayed, where $z=X_s$. Edges represented in gray in the second plot are those that are not part of $\cB_{\hslash,t}(X_t)\cup\cB_{\hslash,t}(Y_t) $. Edges represented in brown are those in $V^+(\sigma_{z,3},E_s)$.}
		\label{fig:explorationV+}
	\end{figure}
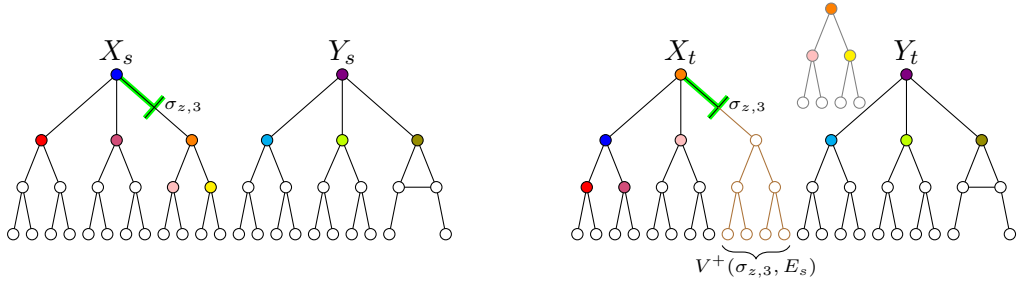
	%%%%%%%%%%%%%%%%%%%%%%%%%%%%%%%%%%%%%%%%%%%%%%%%%%%%%%%%%%

%%%%%%%%%%%%%%%%%%%%%%%%%%%%%%%%%%%%%%%%%%%%%%%%%%%%%%%%%%%	
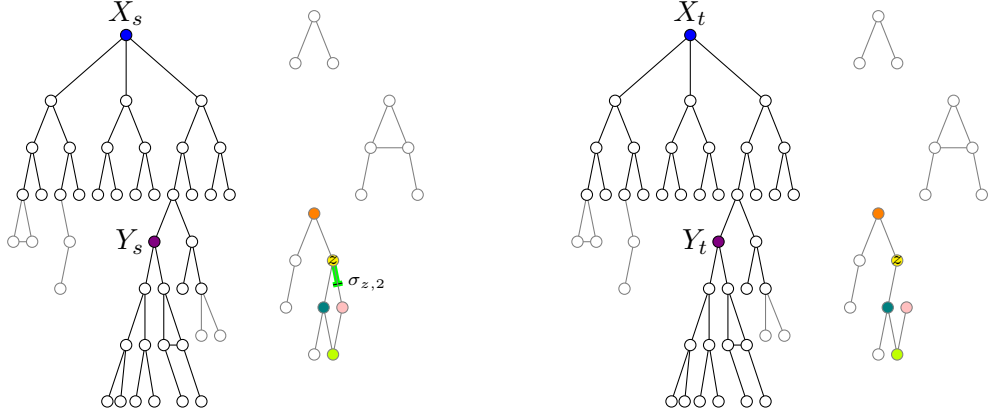
\begin{figure}[!ht]
\centering
\begin{tikzpicture}[scale=0.5]

%Left branch
\draw (0,4.25) -- (-2,2.5);
\draw (-2,2.5) -- (-2.5,1.25);
\draw (-2,2.5) -- (-1.5,1.25);
\draw (-1.5,1.25) -- (-1.75,0);
\draw (-1.5,1.25) -- (-1.25,0);
\draw (-2.5,1.25) -- (-2.25,0);
\draw (-2.5,1.25) -- (-2.75,0);

\draw [color=gray] (-2.75, 0) -- (-3, -1.25);
\draw [color=gray] (-2.75, 0) -- (-2.5, -1.25);
\draw [color=gray] (-3, -1.25) -- (-2.5, -1.25);
\draw [color=gray] (-1.75, 0) -- (-1.5, -1.25);
\draw [color=gray] (-1.5, -1.25) -- (-1.75, -2.5);

%Middle branch
\draw (0,4.25) -- (0,2.5);
\draw (0,2.5) -- (-0.55,1.25);
\draw (0,2.5) -- (0.5,1.25);
\draw (-0.5,1.25) -- (-0.75,0);
\draw (-0.5,1.25) -- (-0.25,0);
\draw (0.5,1.25) -- (0.25,0);
\draw (0.5,1.25) -- (0.75,0);

%Right branch
\draw (0,4.25) -- (1, 27/8);
\draw (1, 27/8) -- (2,2.5);
\draw (2,2.5) -- (2.5,1.25);
\draw (2,2.5) -- (1.5,1.25);
\draw (1.5,1.25) -- (1.75,0);
\draw (1.5,1.25) -- (1.25,0);
\draw (2.5,1.25) -- (2.25,0);
\draw (2.5,1.25) -- (2.75,0);

\begin{scope}[shift = {(-0.5, -1.25)}]

%Bottom attachments
\draw (1.75, 1.25) -- (1.25,0);
\draw (1.25,0) -- (1,-1.25);
\draw (1.25,0) -- (1.5,-1.25);
\draw (1.5,-1.25) -- (1.5,-2.75);
\draw (1.5,-1.25) -- (2,-2.75);
\draw (2,-2.75) -- (1.5,-2.75);
\draw (1,-1.25) -- (0.5,-2.75);
\draw (1,-1.25) -- (1,-2.75);
\draw (1.5,-2.75) -- (2,-4.25);
\draw (2,-2.75) -- (2.5,-4.25);
\draw (0.5,-2.75) -- (0.35,-4.25);
\draw (0.5,-2.75) -- (0,-4.25);
\draw (1,-2.75) -- (1.25,-4.25);
\draw (1,-2.75) -- (0.75,-4.25);

\draw (1.75, 1.25) -- (2.25,0);
\draw (2.25,0) -- (2, -1.25);
\draw (2.25,0) -- (2.5, -1.25);

\draw [color=gray] (2.5, -1.25) -- (3, -2.5);
\draw [color=gray] (2.5, -1.25) -- (2.5, -2.5);

\draw [fill=white] (1.75,1.25) circle (0.15);
\draw [fill=violet] (1.25,0) circle (0.15);
\draw [fill=white] (1.5,-1.25) circle (0.15);
\draw [fill=white] (1,-1.25) circle (0.15);
\draw [fill=white] (1.5,-2.75) circle (0.15);
\draw [fill=white] (2,-2.75) circle (0.15);
\draw [fill=white] (0.5,-2.75) circle (0.15);
\draw [fill=white] (1,-2.75) circle (0.15);
\draw [fill=white] (2,-4.25) circle (0.15);
\draw [fill=white] (2.5,-4.25) circle (0.15);
\draw [fill=white] (0.75,-4.25) circle (0.15);
\draw [fill=white] (1.25,-4.25) circle (0.15);
\draw [fill=white] (0,-4.25) circle (0.15);
\draw [fill=white] (0.35,-4.25) circle (0.15);

\draw [fill=white] (2.25,0) circle (0.15);
\draw [fill=white] (2, -1.25) circle (0.15);
\draw [fill=white] (2.5, -1.25) circle (0.15);

\draw [color=gray, fill=white] (3, -2.5) circle (0.15);
\draw [color=gray, fill=white] (2.5, -2.5) circle (0.15);

\node[left] at (1.25, 0) {$Y_{s}$};

\end{scope}

%Root
\draw [fill=blue] (0,4.25) circle (0.15);
\node[above] at (0,4.25) {$X_{s}$}; 

%Left branch
\draw [fill=white] (-2,2.5) circle (0.15);
\draw [fill=white] (-1.5,1.25) circle (0.15);
\draw [fill=white] (-2.5,1.25) circle (0.15);
\draw [fill=white] (-2.75,0) circle (0.15);
\draw [fill=white] (-2.25,0) circle (0.15);
\draw [fill=white] (-1.75,0) circle (0.15);
\draw [fill=white] (-1.25,0) circle (0.15);

\draw [color=gray, fill=white] (-2.5, -1.25) circle (0.15);
\draw [color=gray, fill=white] (-3, -1.25) circle (0.15);
\draw [color=gray, fill=white] (-1.5, -1.25) circle (0.15);
\draw [color=gray, fill=white] (-1.75, -2.5) circle (0.15);

%Middle branch
\draw [fill=white] (0,2.5) circle (0.15);
\draw [fill=white] (-0.5,1.25) circle (0.15);
\draw [fill=white] (0.5,1.25) circle (0.15);
\draw [fill=white] (0.75,0) circle (0.15);
\draw [fill=white] (0.25,0) circle (0.15);
\draw [fill=white] (-0.25,0) circle (0.15);
\draw [fill=white] (-0.75,0) circle (0.15);

%Right branch
\draw [fill=white] (2,2.5) circle (0.15);
\draw [fill=white] (1.5,1.25) circle (0.15);
\draw [fill=white] (2.5,1.25) circle (0.15);
\draw [fill=white] (2.75,0) circle (0.15);
\draw [fill=white] (2.25,0) circle (0.15);
\draw [fill=white] (1.75,0) circle (0.15);

\begin{scope}[shift = {(5, 2.25)}]

%Left branch
\draw [color=gray] (0,2.5) -- (0.5,1.25);
\draw [color=gray] (0,2.5) -- (-0.5,1.25);

%Left branch
\draw [color=gray, fill=white] (0,2.5) circle (0.15);
\draw [color=gray, fill=white] (0.5,1.25) circle (0.15);
\draw [color=gray, fill=white] (-0.5,1.25) circle (0.15);

\end{scope}

\begin{scope}[shift = {(7, 0)}]

%Left branch
\draw [color=gray] (0,2.5) -- (0.5,1.25);
\draw [color=gray] (0,2.5) -- (-0.5,1.25);
\draw [color=gray] (0.5,1.25) -- (0.75,0);
\draw [color=gray] (-0.5,1.25) -- (-0.75,0);
\draw [color=gray] (-0.5,1.25) -- (0.5,1.25);

%Left branch
\draw [color=gray, fill=white] (0,2.5) circle (0.15);
\draw [color=gray, fill=white] (0.5,1.25) circle (0.15);
\draw [color=gray, fill=white] (-0.5,1.25) circle (0.15);
\draw [color=gray, fill=white] (-0.75,0) circle (0.15);
\draw [color=gray, fill=white] (0.75,0) circle (0.15);

\end{scope}

\begin{scope}[shift={(3,-3)}]

%Stub
\draw [line width=2, green] (2.5,1.25) -- (21/8,25/40);
\draw [line width=2, green] (20/8,24/40) -- (22/8,26/40);
\draw (20/8,24/40) -- (22/8,26/40);
\node[right] at (21/8, 5/8) { {\tiny $\sigma_{z, 2}$} };

%Right branch
\draw [color=gray] (2,2.5) -- (2.5,1.25);
\draw [color=gray] (2,2.5) -- (1.5,1.25);
\draw [color=gray] (1.5,1.25) -- (1.25,0);
\draw [color=gray] (2.5,1.25) -- (2.25,0);
\draw [color=gray] (2.5,1.25) -- (2.75,0);
\draw [color=gray] (2.25,0) -- (2.5,-1.25);
\draw [color=gray] (2.75,0) -- (2.5,-1.25);
\draw [color=gray] (2.25,0) -- (2,-1.25);

%Right branch
\draw [color=gray, fill=orange] (2,2.5) circle (0.15);
\draw [color=gray, fill=white] (1.5,1.25) circle (0.15);
\draw [color=gray, fill=yellow] (2.5,1.25) circle (0.15);
\draw [color=gray, fill=pink] (2.75,0) circle (0.15);
\draw [color=gray, fill=teal] (2.25,0) circle (0.15);
\draw [color=gray, fill=white] (1.25,0) circle (0.15);
\draw [color=gray, fill=lime] (2.5,-1.25) circle (0.15);
\draw [color=gray, fill=white] (2,-1.25) circle (0.15);

\node at (2.5, 1.25) { {\tiny $z$} };

\end{scope}

\begin{scope}[shift={(15,0)}]

%Left branch
\draw (0,4.25) -- (-2,2.5);
\draw (-2,2.5) -- (-2.5,1.25);
\draw (-2,2.5) -- (-1.5,1.25);
\draw (-1.5,1.25) -- (-1.75,0);
\draw (-1.5,1.25) -- (-1.25,0);
\draw (-2.5,1.25) -- (-2.25,0);
\draw (-2.5,1.25) -- (-2.75,0);

\draw [color=gray] (-2.75, 0) -- (-3, -1.25);
\draw [color=gray] (-2.75, 0) -- (-2.5, -1.25);
\draw [color=gray] (-3, -1.25) -- (-2.5, -1.25);
\draw [color=gray] (-1.75, 0) -- (-1.5, -1.25);
\draw [color=gray] (-1.5, -1.25) -- (-1.75, -2.5);

%Middle branch
\draw (0,4.25) -- (0,2.5);
\draw (0,2.5) -- (-0.55,1.25);
\draw (0,2.5) -- (0.5,1.25);
\draw (-0.5,1.25) -- (-0.75,0);
\draw (-0.5,1.25) -- (-0.25,0);
\draw (0.5,1.25) -- (0.25,0);
\draw (0.5,1.25) -- (0.75,0);

%Right branch
\draw (0,4.25) -- (1, 27/8);
\draw (1, 27/8) -- (2,2.5);
\draw (2,2.5) -- (2.5,1.25);
\draw (2,2.5) -- (1.5,1.25);
\draw (1.5,1.25) -- (1.75,0);
\draw (1.5,1.25) -- (1.25,0);
\draw (2.5,1.25) -- (2.25,0);
\draw (2.5,1.25) -- (2.75,0);

\begin{scope}[shift = {(-0.5, -1.25)}]

%Bottom attachments
\draw (1.75, 1.25) -- (1.25,0);
\draw (1.25,0) -- (1,-1.25);
\draw (1.25,0) -- (1.5,-1.25);
\draw (1.5,-1.25) -- (1.5,-2.75);
\draw (1.5,-1.25) -- (2,-2.75);
\draw (2,-2.75) -- (1.5,-2.75);
\draw (1,-1.25) -- (0.5,-2.75);
\draw (1,-1.25) -- (1,-2.75);
\draw (1.5,-2.75) -- (2,-4.25);
\draw (2,-2.75) -- (2.5,-4.25);
\draw (0.5,-2.75) -- (0.35,-4.25);
\draw (0.5,-2.75) -- (0,-4.25);
\draw (1,-2.75) -- (1.25,-4.25);
\draw (1,-2.75) -- (0.75,-4.25);

\draw (1.75, 1.25) -- (2.25,0);
\draw (2.25,0) -- (2, -1.25);
\draw (2.25,0) -- (2.5, -1.25);

\draw [color=gray] (2.5, -1.25) -- (3, -2.5);
\draw [color=gray] (2.5, -1.25) -- (2.5, -2.5);

\draw [fill=white] (1.75,1.25) circle (0.15);
\draw [fill=violet] (1.25,0) circle (0.15);
\draw [fill=white] (1.5,-1.25) circle (0.15);
\draw [fill=white] (1,-1.25) circle (0.15);
\draw [fill=white] (1.5,-2.75) circle (0.15);
\draw [fill=white] (2,-2.75) circle (0.15);
\draw [fill=white] (0.5,-2.75) circle (0.15);
\draw [fill=white] (1,-2.75) circle (0.15);
\draw [fill=white] (2,-4.25) circle (0.15);
\draw [fill=white] (2.5,-4.25) circle (0.15);
\draw [fill=white] (0.75,-4.25) circle (0.15);
\draw [fill=white] (1.25,-4.25) circle (0.15);
\draw [fill=white] (0,-4.25) circle (0.15);
\draw [fill=white] (0.35,-4.25) circle (0.15);

\draw [fill=white] (2.25,0) circle (0.15);
\draw [fill=white] (2, -1.25) circle (0.15);
\draw [fill=white] (2.5, -1.25) circle (0.15);

\draw [color=gray, fill=white] (3, -2.5) circle (0.15);
\draw [color=gray, fill=white] (2.5, -2.5) circle (0.15);

\node[left] at (1.25, 0) {$Y_{t}$};

\end{scope}

%Root
\draw [fill=blue] (0,4.25) circle (0.15);
\node[above] at (0,4.25) {$X_{t}$}; 

%Left branch
\draw [fill=white] (-2,2.5) circle (0.15);
\draw [fill=white] (-1.5,1.25) circle (0.15);
\draw [fill=white] (-2.5,1.25) circle (0.15);
\draw [fill=white] (-2.75,0) circle (0.15);
\draw [fill=white] (-2.25,0) circle (0.15);
\draw [fill=white] (-1.75,0) circle (0.15);
\draw [fill=white] (-1.25,0) circle (0.15);

\draw [color=gray, fill=white] (-2.5, -1.25) circle (0.15);
\draw [color=gray, fill=white] (-3, -1.25) circle (0.15);
\draw [color=gray, fill=white] (-1.5, -1.25) circle (0.15);
\draw [color=gray, fill=white] (-1.75, -2.5) circle (0.15);

%Middle branch
\draw [fill=white] (0,2.5) circle (0.15);
\draw [fill=white] (-0.5,1.25) circle (0.15);
\draw [fill=white] (0.5,1.25) circle (0.15);
\draw [fill=white] (0.75,0) circle (0.15);
\draw [fill=white] (0.25,0) circle (0.15);
\draw [fill=white] (-0.25,0) circle (0.15);
\draw [fill=white] (-0.75,0) circle (0.15);

%Right branch
\draw [fill=white] (2,2.5) circle (0.15);
\draw [fill=white] (1.5,1.25) circle (0.15);
\draw [fill=white] (2.5,1.25) circle (0.15);
\draw [fill=white] (2.75,0) circle (0.15);
\draw [fill=white] (2.25,0) circle (0.15);
\draw [fill=white] (1.75,0) circle (0.15);

\begin{scope}[shift = {(5, 2.25)}]

%Left branch
\draw [color=gray] (0,2.5) -- (0.5,1.25);
\draw [color=gray] (0,2.5) -- (-0.5,1.25);

%Left branch
\draw [color=gray, fill=white] (0,2.5) circle (0.15);
\draw [color=gray, fill=white] (0.5,1.25) circle (0.15);
\draw [color=gray, fill=white] (-0.5,1.25) circle (0.15);

\end{scope}

\begin{scope}[shift = {(7, 0)}]

%Left branch
\draw [color=gray] (0,2.5) -- (0.5,1.25);
\draw [color=gray] (0,2.5) -- (-0.5,1.25);
\draw [color=gray] (0.5,1.25) -- (0.75,0);
\draw [color=gray] (-0.5,1.25) -- (-0.75,0);
\draw [color=gray] (-0.5,1.25) -- (0.5,1.25);

%Left branch
\draw [color=gray, fill=white] (0,2.5) circle (0.15);
\draw [color=gray, fill=white] (0.5,1.25) circle (0.15);
\draw [color=gray, fill=white] (-0.5,1.25) circle (0.15);
\draw [color=gray, fill=white] (-0.75,0) circle (0.15);
\draw [color=gray, fill=white] (0.75,0) circle (0.15);

\end{scope}

\begin{scope}[shift={(3,-3)}]

%Right branch
\draw [color=gray] (2,2.5) -- (2.5,1.25);
\draw [color=gray] (2,2.5) -- (1.5,1.25);
\draw [color=gray] (1.5,1.25) -- (1.25,0);
\draw [color=gray] (2.5,1.25) -- (2.25,0);
\draw [color=gray] (2.25,0) -- (2.5,-1.25);
\draw [color=gray] (2.75,0) -- (2.5,-1.25);
\draw [color=gray] (2.25,0) -- (2,-1.25);

%Right branch
\draw [color=gray, fill=orange] (2,2.5) circle (0.15);
\draw [color=gray, fill=white] (1.5,1.25) circle (0.15);
\draw [color=gray, fill=yellow] (2.5,1.25) circle (0.15);
\draw [color=gray, fill=pink] (2.75,0) circle (0.15);
\draw [color=gray, fill=teal] (2.25,0) circle (0.15);
\draw [color=gray, fill=white] (1.25,0) circle (0.15);
\draw [color=gray, fill=lime] (2.5,-1.25) circle (0.15);
\draw [color=gray, fill=white] (2,-1.25) circle (0.15);

\node at (2.5, 1.25) { {\tiny $z$} };

\end{scope}

\end{scope}

\end{tikzpicture}

\caption{Visual representation of the transition in step (6-d), where an arrival of $\hat{\mathfrak{T}}^{\rm dyn}_{z,2}$. Edges represented in gray in the first (respectively, second) plot are those that are not part of $\cB_{\hslash,s}(X_s)\cup\cB_{\hslash,s}(Y_s)$ (respectively, $\cB_{\hslash,t}(X_t)\cup\cB_{\hslash,t}(Y_t) $).}
\label{fig:explorationDec}
\end{figure}	
%%%%%%%%%%%%%%%%%%%%%%%%%%%%%%%%%%%%%%%%%%%%%%%%%%%%%%%%

The rest of this section will be devoted to proving the following proposition, which shows that the event ``$|E_t|$ is small'' is \emph{typical}, in the sense of Section \ref{ss.typev}.
\begin{proposition}{\bf [Control on the size of $E_t$]} 	
\label{lemma:E-far}
Let
		\begin{equation}
		\label{def-H}
		\cH \coloneqq \{ |E_{t}| \leq n^{12\delta}\,\,\forall\, t \in [0, n^{3/2}] \},
		\end{equation}
where $\delta$ is given by~\eqref{def-sigma}. For every $d \geq 3$ and $\nu>0$,
		\begin{equation}
		\label{eq:Efar}
		\lim_{n \to \infty} \P(\cH)= 1.
		\end{equation}
\end{proposition}

\begin{proof}
For every $t \geq 0$, partition the set $E_t$ in two parts, $E^{\rm near}_t$ and $E^{\rm far}_t$. The stubs in $E^{\rm far}_t$ are those that are at (graph) distance larger than $\hslash$ in $E_t$ from both $X_t$ and $Y_t$. To ease the visualisation, note that edges in $E^{\rm far}_t$ are those reported in gray in Figures \ref{fig:explorationRW}--\ref{fig:explorationDec}. Clearly, for $n$ large enough, $|E_t^{\rm near}| \leq 2 d^\hslash$, for every $t \geq 0$. We are therefore mainly interested in providing a bound on the size of $E^{\rm far}_t$. More precisely, \eqref{eq:Efar} follows at once as soon as we verify that
		\begin{equation}
		\label{goalfar}
		\lim_{n\to\infty} \P\big( |E_t^{\rm far}| \leq n^{10\delta}, \text{ for all } 0 \leq t  \leq n^{3/2} ) = 1,
		\end{equation}
since $2 d^\hslash \leq n^{2\delta}$, for all $n$ large enough. To simplify the reading, we set $\varepsilon = 10\delta$, and stress that all the inequalities below are true only for values of $n$ large enough. 
		
We organise the proof of~\eqref{goalfar} in five steps.
		
\medskip\noindent
{\bf 1.}
Fix some $s\ge 0$ and consider an arbitrary realisation of $(E_r,X_r,Y_r)_{r\in[0,s]}$. Let $t>s$ the first arrival after time $s$. We aim at controlling the difference $|E_t^{\rm far}|-|E_s^{\rm far}|$.

		\begin{itemize}
		\item[\bf (A)] $|E_t^{\rm far}|-|E_s^{\rm far}|>0$ if one of the following hold:
		\begin{itemize}
			\item[\bf (a)] 
			One of the two random walks does a step (as in step (6-a)): in this case the increase can be bounded by $|E^{\rm far}_t|-|E^{\rm far}_s|< d^{\hslash}$ a.s.\ (see Figure \ref{fig:explorationRW}).
			\item[\bf (b)] If a rewiring as in step (6-b) happens, that is, if $t$ is the arrival time of two Poisson processes associated to stubs that are within distace $\hslash$ from $\{X_{s}, Y_{s}\}$. In this case, the increase in $|E_t^{\rm far}|-|E_s^{\rm far}|$ can be bounded by $d^{[\hslash-L]_+}$ a.s., where 
			\begin{equation}
			L = 1+\min \{ {\rm dist}_s(X_s,z), {\rm dist}_s(Y_s,z), {\rm dist}_s(X_s,v), {\rm dist}_s(Y_s,v ) \}, 
			\end{equation}
			the worst case being e.g.\ $(z,v)=(X_s,Y_s)$ (see Figures~\ref{fig:explorationV-},~\ref{fig:explorationV0} and~\ref{fig:explorationSigma-}).
			\item[\bf (c)]  If a rewiring as in step (6-c) takes place, that is, only one arrival for a Poisson process associated to a stub within distance $\hslash$ from $\{X_{s}, Y_{s}\}$. The increase can be bounded by $|E^{\rm far}_t|-|E^{\rm far}_s| < d^{\hslash}$ a.s., the worst case being $z \in \{X_s,Y_s\}$ (see Figure \ref{fig:explorationV+}).
		\end{itemize}
		\item[\bf (B)]  
		$|E_t^{\rm far}| - |E_s^{\rm far}| = -1$ if $t$ is an arrival time from a stub $\sigma_{z, i} \in E_s^{\rm far}$, i.e., as in step (6-d) (see Figure \ref{fig:explorationDec}).
	\end{itemize} 

Note that the rate of the events in the previous list can be controlled as follows:
		\begin{itemize}
		\item 
		The event in {\bf (A-a)} occurs at rate $2$.
		\item
		The event in {\bf (A-b)} with $L = \ell \in \{1, \dots, \hslash-1\}$ occurs at rate bounded by $2\nu d^{\ell}$.
		\item
		The event in {\bf (A-c)} occurs at rate bounded from above by $2\nu d^\hslash$.
		\item
		The event in {\bf (B)} occurs at rate at least $\nu |E^{\rm far}_s| \big(1-\frac{4d^{\hslash} + 2|E^{\rm far}_s|}{dn}\big)$.
		\end{itemize} 

\medskip\noindent
{\bf 2.}
By the estimates above, the evolution of $|E_t^{\rm far}|$ can be stochastically dominated by a random process $(I_t)_{t \ge 0}$ that evolves as follows:
		\begin{itemize}
		\item $I_t \to I_t +d^\hslash$ at rate $2$.
		\item $I_t \to I_t +d^{\hslash-\ell}$ at rate $2\nu d^\ell$ for all $\ell \in \{1,\dots,\hslash\}$.
		\item $I_t \to I_t-1$ at rate $\nu \, I_t\,(1-\frac{4d^{\hslash}+2I_t}{dn})$.
		\end{itemize}
In other words, the process $(I_t)_{t\ge 0}$ can be naturally coupled with the exploration process so as to have $I_t\ge |E_t^{\rm far}|$ for all $t \geq 0$, almost surely. It will be convenient to consider another process, denoted by $(\bar I_t)_{t\ge 0}$, that has the same upward transitions as those of the process $(I_t)_{t \ge 0}$, but decreases as follows:
		\begin{itemize}
		\item $\bar I_t \to \bar I_t-1$ at rate $\frac{\nu}{2} \, (\bar I_t \wedge n^{2\varepsilon})$.
		\end{itemize}
		
Define 
		\begin{equation}
		\tau_{\rm big} \coloneqq \inf\{t \geq 0\colon I_t> n^{2\varepsilon}\},
		\end{equation}
and note that it is possible to couple the two processes so as to have $I_t \leq \bar I_t$ for all $0 \leq t \leq \tau_{\rm big}$. In light of these domination,~\eqref{goalfar} follows as soon as we verify that
		\begin{equation}
		\label{goalfar2}
		\lim_{n \to \infty} \P \Big( \sup_{t \in [0, n^{3/2}]} \bar I_t > n^\varepsilon \Big) = 0.
		\end{equation}
		
\medskip\noindent
{\bf 3.}
To simplify notation, in what follows we use that, by the definition of $\hslash$ in~\eqref{def-sigma} and the fact that $\nu>0$ is a fixed constant,
		\begin{equation}
		\label{totrate}
		\bigg( 2\nu \sum_{\ell=0}^\hslash d^\ell \bigg) \, \vee \, d^\hslash \leq n^{2\delta}, \text{ for all } n \text{ large enough},
		\end{equation}
that is, $n^{2\delta}$ serves as an upper bound for both the rate of increase of $\bar I$ and the magnitude of an upward jump. Moreover, by the definition of $(\bar I_t)_{t\ge0}$, within time $\tau_{\rm big}$ the total jump rate (including also the downward jumps) is uniformly bounded by $n^{3\varepsilon}$. Therefore, defining
		\begin{equation}
		\label{eq:def-cR} 
		\cR \coloneqq \big\{(\bar I_t)_{0 \leq t \leq n^{3/2}} \text{ makes at most } T
		\coloneqq n^{3/2+4\varepsilon} \text{ jumps} \big\},
		\end{equation}
we have that $\P(\cR^\complement)$ is bounded by the probability that a Poisson random variable of mean $n^{3/2+3\varepsilon}$ is larger than $n^{3/2+4\varepsilon}$. By Markov's inequality,
		\begin{equation}
		\label{eq:est.R}
		\P(\cR^\complement) \leq \P \Big( {\rm Poisson}(n^{3/2+3\varepsilon}) > n^{3/2+4\varepsilon} \Big) \to 0.
		\end{equation}
	
\medskip\noindent
{\bf 4.}
Let $(\varkappa_i)_{i \in \N_0} \subset \R_+$ denote the sequence of jump times of the process $(\bar I_t)_{t\ge 0}$. Then
		\begin{equation}
		\label{eq:bound-bar-I}
		\begin{split}
		\P \Big( \sup_{t \in [0, n^{3/2}]} \bar I_t > n^\varepsilon \Big) 
		& \leq \P \big( \big\{\bar I_{\varkappa_j}> n^\varepsilon, \text{ for some } j \leq T \big\} \cap \cR \big) 
		+ \P(\cR^\complement) \\
		& \leq T \max_{1 \leq j \leq T} \P(\bar I_{\varkappa_j} > n^\varepsilon) + \P(\cR^\complement).
		\end{split}
		\end{equation}
By \eqref{totrate}, in order for the event $\{\bar I_{\varkappa_j}>n^\varepsilon\}$ to occur there must exist a jump time $\varkappa_r$, with $ r<j$, such that 
		\begin{equation}
		\bar I_{\varkappa_r} \in A \coloneqq [n^{\varepsilon/2}, n^{\varepsilon/2}+n^{2\delta}]
		\end{equation}
and $\bar I_{\varkappa_i} > n^{\varepsilon/2}$ for all $i \in \{r,\dots ,j\}$. Therefore, Markov's property yields
		\begin{equation}
		\label{eq:bound-bar-I-2}
		\begin{split}
		\P(&\bar I_{\varkappa_j}> n^\varepsilon) \\
		\quad & = \P\big( \{\bar I_{\varkappa_j} > n^{\varepsilon}\} \cap \big\{ \exists\, r < j 
		\text{ s.t. } \bar I_{\varkappa_r} \in A, \bar I_{\varkappa_i} > n^{\varepsilon/2}\,\, 
		\forall\, i \in \{r,\dots,j\} \big\} \big) \\
		& \leq j \max_{r < j} \P \big( \bar I_{\varkappa_r} \in A ) \P(\{\bar I_{\varkappa_j} > n^\varepsilon\}
		\cap \big\{	\bar I_{\varkappa_i} > n^{\varepsilon/2}\,\, \forall\, i \in \{r, \dots, j \} \big\} \mid
		\bar I_{\varkappa_r} \in A \big) \\
		& \leq j \max_{r < j} \P(\cQ_j^r | \bar{I}_{\varkappa_r} \in A),
		\end{split}
		\end{equation}
where
		\begin{equation}
		\cQ_j^r \coloneqq \{ \bar I_{\varkappa_j} > n^\varepsilon \} \cap \big\{ \bar I_{\varkappa_i} 	> n^{\varepsilon/2}, \text{ for all } i \in \{r, \dots, j\} \big\}.
		\end{equation}

We now prove that the conditional  probability in~\eqref{eq:bound-bar-I-2} can be bounded as follows:
		\begin{equation}
		\label{eq:bound-upward}
		\P(\cQ_j^r \mid \bar{I}_r\in A) \leq \P\bigg( {\rm Bin} \bigg( j-r, \frac{n^{2\delta}}{n^{2\delta} + \frac{\nu}{2} n^{\varepsilon/2}} \bigg)
		\geq \frac{n^\varepsilon + (j-r) - n^{\varepsilon/2}-n^{2\delta}}{n^{2\delta}+1} \bigg).
		\end{equation}
Indeed, in order for the process to be above $n^\varepsilon$ after $j-r$ jumps, it must make $j_\downarrow$ jumps downwards and $j_\uparrow$ jumps upwards such that $j_\downarrow + j_\uparrow = j-r$. Since downward jumps have size $1$ and upward jumps have size at most $n^{2\delta}$ (cf.~\eqref{totrate}), it follows that
		\begin{equation}
		n^{2\delta} j_\uparrow - j_\downarrow \geq n^{\varepsilon} - n^{\varepsilon/2} - n^{2\delta} 
		\quad \Longrightarrow \quad j_\uparrow
		\geq \frac{n^{\varepsilon}+(j-r) - n^{\varepsilon/2} - n^{2\delta}}{n^{2\delta}+1}. 
		\end{equation}
Moreover, under the event that the process does not go below $n^{\varepsilon/2}$ in the next $j-r$ jumps, the probability of an upward jump is at most $\frac{n^{2\delta}}{n^{2\delta}+\frac\nu2 n^{\varepsilon/2}}$  (cf.\ again~\eqref{totrate}). These two facts imply~\eqref{eq:bound-upward}. 

Furthermore, note that, once again using that upward jumps are bounded by $n^{2\delta}$, we get
\begin{equation}
\P(\cQ_j^r \mid \bar{I}_r \in A) =  0
\end{equation}
if $j-r \leq \frac{1}{n^{2\delta}}(n^{\varepsilon} - n^{\varepsilon/2} - n^{2\delta})$. By choosing $n$ large enough, the probability above is zero if $j-r \leq n^{7\delta}$.

\medskip\noindent
{\bf 5.}
To simplify the reading, note that for all $n$ large enough the preceding discussion yields the bound
		\begin{equation}
		\label{eq:bound-upward-2}
		\P(\cQ_j^r \mid \bar{I}_r \in A) \leq \ind_{ \big\{ j-r > n^{7\delta} \big\}} \P\bigg( {\rm Bin}\bigg(j-r, \frac{n^{2\delta}}{n^{2\delta}+\frac{\nu}{2} n^{\varepsilon/2}}\bigg)
		\geq \frac {j-r}{n^{2\delta}} \bigg).
		\end{equation}
Since the expectation of the latter Binomial random variable is of order $(j-r)n^{-3\delta}$, Chernoff's bound yields
		\begin{equation}
		\P \bigg( {\rm Bin} \bigg( j-r, \frac{n^{2\delta}}{n^{2\delta}+\frac\nu2 n^{\varepsilon/2}} \bigg) 
		\geq (j-r) n^{-2\delta} \bigg) \leq \exp\Big( -\tfrac{1}{2} (j-r)n^{-4\delta} \Big).
		\end{equation}
Hence, we get
		\begin{equation}
		\label{eq:bound-upward-final}
		\P(\cQ_j^r \mid \bar{I}_r\in B) \leq \ind_{ \big\{ j-r > n^{7\delta} \big\}} \exp\Big( - \tfrac{1}{4} (j-r) n^{-4\delta}\Big).
		\end{equation}
Substituting \eqref{eq:est.R},~\eqref{eq:bound-bar-I-2} and~\eqref{eq:bound-upward-final} into~\eqref{eq:bound-bar-I}, and recalling the definition of $T$ in~\eqref{eq:def-cR}, implies
		\begin{align*}
		\P\bigg( \sup_{t \in [0, n^{3/2}]} \bar{I}_t > n^\varepsilon \bigg) 
		& \leq T \max_{1 \leq j \leq T} \P( \bar I_{\varkappa_j} > n^\varepsilon) + \P(\cR^\complement) \\
		& \leq T \max_{1 \leq j \leq T} j \max_{r<j} \P( \cQ_j^r \mid \bar{I}_r\in A) + \P(\cR^\complement)
		\to 0,
		\end{align*}
concluding the proof.
\end{proof}
	
%%%
	
\subsection{Coupling}
\label{ss.coup}
	
We now introduce an explicit coupling between the exploration process discussed in Section~\ref{ss.expl} and the toy model analysed in Section~\ref{sec:meet}. With the help of this coupling we provide the proof of Theorem \ref{th:meeting} at the end of this section. 
	
Recall $\Prob$ denotes the probability law of the toy model and $\P$ stands for the probability law of the graphical construction of $(E_t,X_t,Y_t)_{t\ge 0}$. As shown in Proposition \ref{prop:nice-events}, at time $t=0$ w.h.p.\ the partial matching $E_0$ will be made by two disjoint trees having height $\hslash$ and rooted at $X_0$ and $Y_0$, respectively, so we work under this initial event. Recall the definition of the process in Section \ref{sec:meet}, and of the marked Poisson processes $({\mathfrak{T}}^W_{z,k})_{W\in\{X,Y\}, z\in\cT^W,1\le k\le d}$. 

To present the coupling between $\Prob$ and $\P$ it is convenient to first establish it in the first phase of the process with law $\Prob$, and consider the coupling of the second phase. 

%%%

\subsection{First phase coupling}
	
We start the process at some $s\ge 0$ and assume that at that time the exploration process $(E_s,X_s,Y_s)$ is such that
	\begin{itemize}
	\item[(C1)] $\cB_{\hslash,s}(X_s)$ and $\cB_{\hslash,s}(Y_s)$  are two disjoint trees, 
	\item[(C2)] ${\rm dist}_s(X_s,Y_s) = \infty$,
	\item[(C3)] $|E_s| \leq n^{12\delta}$,
	\end{itemize}
while the process $\Prob$ starts from the first phase. We proceed with the coupled construction of the first phase as follows:

	\begin{itemize}
		\item[(I-1$^{\rm st}$)] 
		Construct a distance preserving bijection $\varphi$ between the stubs in $\cB_{\hslash,s}(X_s) \cup \cB_{\hslash,s}(Y_s)$ and the stubs in $\cT^X \cup \cT^Y$. 
		\item[(II-1$^{\rm st}$)] 
		Look at the collections of (unmarked) processes 
		\begin{equation}
		(\mathfrak{T}^{\rm rw}_{z,k})_{z \in \{X_s,Y_s\}, 1 \leq k \leq d}, 
		\qquad (\hat{\mathfrak{T}}^{\rm dyn}_{z,k})_{\{z \in [n], 1 \leq k \leq d\colon \sigma_{z,k} \in E_s\}},
		\end{equation}
		and let $t>s$ be the time of the first arrival among the two collections.
		\begin{itemize}
			\item[(rw)] 
			If the arrival at $t$ comes from the former collection, then construct the process $(E_r,X_r,Y_r)_{r \in [s,t]}$ as explained in point (6-a) of the procedure in Section \ref{ss.expl}. Moreover, 
			\begin{itemize}
				\item [(i)]	
				if $(E_t,X_t,Y_t)$ satisfies (C1), (C2), and (C3), then restart the procedure with $t$ in place of $s$;
				\item[(ii)] otherwise, declare the coupling \emph{failed}.
			\end{itemize}
			\item[(dyn)] 
			\begin{itemize}
				\item[(i)]
			If the arrival at $t$ comes from the latter collection, say from $\hat{\mathfrak{T}}^{\rm dyn}_{z,i}$ for a unique $z \in [n]$, $1 \leq i \leq d$ such that $\sigma_{z,i} \in \bar E_s$, then construct the process $(E_r,X_r,Y_r)_{r \in [s,t]}$ as explained in point (6-c) of the procedure in Section \ref{ss.expl}. In particular
				\begin{itemize}
					\item [(A)]	
					if $(E_t,X_t,Y_t)$ satisfies (C1), (C2), and (C3), then restart the procedure with $t$ in place of $s$;
					\item[(B)] 
					otherwise, declare the coupling \emph{failed}.
				\end{itemize}
				\item[(ii)]
			If the arrival at $t$ comes from the latter collection, for two processes $\hat{\mathfrak{T}}^{\rm dyn}_{z,i}$ and $\hat{\mathfrak{T}}^{\rm dyn}_{v,j}$ for $z, v \in [n]$, $1 \leq i, j \leq d$ such that $\sigma_{z,k} \in E_s$, then construct the process $(E_r,X_r,Y_r)_{r \in [s,t]}$ as in point (6-b) of the procedure in Section \ref{ss.expl}. Furthermore:
				\begin{itemize}
					\item [(A)]
					if $(\varphi(\sigma_{z,i}),\varphi(\sigma_{v,j}))$ is a \emph{nice} pair of stubs for the toy model, and $\cB_{\hslash,t}(X_t)$ and $\cB_{\hslash,t}(Y_t)$ are two (overlapping) trees, then declare the coupling \emph{successful} and the toy process enters the second phase at time $t$ at distance $\ell={\rm dist}_t(X_t,Y_t)$;
					\item[(B)] 
					if $(\varphi(\sigma_{z,i}),\varphi(\sigma_{v,j}))$ is \emph{not a nice} pair of stubs for the toy model, and $(E_t,X_t,Y_t)$ satisfies (C1), (C2), and (C3), restart the procedure at time $t$;
					\item[(C)]
					otherwise, declare the coupling \emph{failed}.
				\end{itemize}
			\end{itemize}
		\end{itemize}
	\end{itemize}
If the coupling fails, then the realisations of the two processes are continued independently, and it is immediate to check that the two marginals indeed coincide with the two processes in Sections~\ref{sec:meet} and~\ref{ss.expl}.
	
In words, starting at time $s \geq 0$ with $(E_s,X_s,Y_s)$ satisfying (C1), (C2) and (C3), and using the procedure just described, we produce an attempt to couple the first phase of the toy model in Section  \ref{sec:meet} with the exploration process. If this attempt \emph{succeeds}, then the toy model enters the second phase at some $1 \leq \ell \leq 2\hslash-1$ that coincides with the distance between $X_t$ and $Y_t$ in the exploration process. 
	
Denote by
	\begin{equation}
	\label{def-F}
	\cF_t\coloneqq\{{\rm dist}_{t}(X_t,Y_t) > 2\hslash \text{ and } {\rm dist}_{t}(X_t,Y_t) < \infty\}
	\end{equation}
the event in which at time $t$ the two random walks are on the same connected component of $E_t$ but more than $2\hslash$ apart. Note that, in order for the coupling to fail at time $t$, one of the following three events must occur:
	\begin{itemize}
	\item[(fail-1$^{\rm st}$-a)] $[\cE_t^{\rm trees}]^\complement$ (defined as in \eqref{events});
	\item[(fail-1$^{\rm st}$-b)] $\cE_t^{\rm trees}\cap \cF_t$;
	\item[(fail-1$^{\rm st}$-c)] $\{|E_t|>n^{12\delta}\}$ (recall Proposition \ref{lemma:E-far}).
	\end{itemize}

%%%

\subsection{Second phase coupling}
	
Given that the first coupling succeeds and the toy model enters the second phase at some $1 \leq \ell \leq 2\hslash-1$, we next explain how to couple the second phase with the evolution of the exploration process. Consider the evolution of the exploration process started at time $t > s$ at some $(E_t,X_t,Y_t)$ such that there exists a unique path joining $X_t$ and $Y_t$ in $E_t$ and $|E_t| \le n^{12\delta}$.
	\begin{itemize}
	\item[(I-2$^{\rm nd}$)] 
	Construct a distance preserving injection $\varphi$ between the stubs of the vertices along the unique path joining $X_t$ to $Y_t$ and the stubs of the vertices along the unique path of $\cT^{\rm joint}$ (cf.\ Section \ref{sec:meet}) joining the two random walks.
	\item[(II-2$^{\rm nd}$)] 
	Proceed with the construction of the exploration process as explained in Section \ref{ss.expl}, and let $r>t$ be the first time when the exploration process evolves:
	\begin{itemize}
	\item[(rw)] If at time $r$ one of the two random walks moves following a stub $\sigma$, then construct the process $(E_u,X_u,Y_u)_{u \in [t,r]}$ as explained in point (6-a) of the procedure in Section \ref{ss.expl}, let the associated random walk in the toy model move following the stub $\varphi(\sigma)$, and do as follows:
	\begin{itemize}
	\item[(i)] 
	if $X_r = Y_r$, then declare the coupling \emph{successful}, set the length of the second phase to $r-t$ and continue to construct the two processes independently;
	\item[(ii)] 
	if $X_r \neq Y_r$, do as follows:
	\begin{itemize}
	\item if there is a unique path joining $X_r$ to $Y_r$ and $|E_t| \leq n^{12\delta}$, then restart from (I-2$^{\rm nd}$) with $r$ instead of $t$;
	\item otherwise, declare the coupling \emph{failed}.
	\end{itemize}
	\end{itemize}
	\item[(dyn)] If at time $r$ there is a rewiring, then construct the process $(E_u,X_u,Y_u)_{u \in [t,r]}$ as explained in points (6-b) and (6-c) of the procedure in Section~\ref{ss.expl}, and do as follows:
	\item[(i)] 
	if the rewiring takes place along the unique path joining $X_t$ and $Y_t$ in $E_t$, then stop the coupling, set the length of the second phase to $r-t$, and do as follows:
	\begin{itemize}
	\item[(A)]
	if as a consequence of the rewiring $(E_r,X_r,Y_r)$ satisfies properties (C1), (C2), and (C3), then say that the second phase ended at a \emph{good point} and restart the coupling of the first phase at $(E_r,X_r,Y_r)$;
	\item[(B)] 
	otherwise, say that the second phase ended at a \emph{bad point}, and declare the coupling \emph{failed}.
	\end{itemize}
	\item[(ii)] 
	if the rewiring does not take place along the unique path joining $X_t$ and $Y_t$ in $E_t$, then do as follows:
	\begin{itemize}
	\item[(A)] 
	if as consequence of the rewiring there exists a unique path joining $X_r$ and $Y_r$ in $E_r$ and $|E_r| \leq n^{12\delta}$, restart from (I-2$^{\rm nd}$) with $r$ instead of $t$;
	\item[(B)] otherwise, declare the coupling \emph{failed}.
	\end{itemize}
	\end{itemize}
	\end{itemize}

For the second phase, the coupling can fail at time $t$ only by realising either one of the following events:
	\begin{itemize}
	\item[(fail-2$^{\rm nd}$-a)] 
	the event in step (dyn-ii-B), which implies that either $[\cE_t^{\rm trees}]^\complement \cap [\cE_t^{\rm far}]^\complement$ or $|E_t| > n^{12\delta}$. 
	\item[(fail-2$^{\rm nd}$-b)] 
	the event in step (dyn-i-B), i.e., ending at a \emph{bad point}.
	\end{itemize}  
Let us now introduce $t_{\rm renew}^{(0)}(t) = t$,
	\begin{equation}
	\label{renew}
	\tau_{\rm renew}^{\ell+1}(t) \coloneqq \inf \Big\{s > \tau_{\rm renew}^{\ell}(t) \colon \begin{array}{c} {\text{every process in } ( \mathfrak{T}^{\rm dyn}_{z,k})_{z \in [n], 1 \leq k \leq d}} \\ {\text{registered at least one arrival}} \end{array} \Big\}
	\end{equation}
and
	\begin{equation}
	\label{renew+}
	K(t) = \inf\left\{ \ell \geq 1 \colon (E_{\tau_{\rm renew}^{\ell}(t)}, X_{\tau_{\rm renew}^{\ell}(t)}, 
	Y_{\tau_{\rm renew}^{\ell}(t)}) \text{ satisfies (C1), (C2), and (C3)}\right\}.
	\end{equation}
Finally, set $\tau_{\rm renew}^+(t) \coloneqq \tau_{\rm renew}^{(K(t))}(t)$. Observe that, by its very construction, $\big( \tau_{\rm renew}^{\ell+1}(t)-\tau_{\rm renew}^{\ell}(t) \big)_{\ell \geq 0}$ is an i.i.d.\ sequence and $K(t)$ is a Geometric random variable with parameter given by the probability that a random graph distributed according to $\mu_d$ satisfies (C1), (C2), and (C3). This probability can be made arbitrarily close to one, provided $n$ is large enough.
	
If the coupling fails at some time $t>0$, then regardless of the phase in which this happens, we declare it \emph{inactive} in the time interval $[t,\tau_{\rm renew}^+(t))$, and restart the coupling from the first phase at $\tau_{\rm renew}^+(t)$. This concludes the definition of the coupling. We start the exploration process at $(E_0,X_0,Y_0)$, assuming that $\cB_{\hslash,0}(X_0)$ and $\cB_{\hslash,0}(Y_0)$ are two disjoint trees (i.e., (C1)-(C3) are satisfied), and use the coupling just described up to time $\tau_{\rm meet}^{\pi \otimes \pi}$. Note that the meeting time between the two random walks might happen when the coupling is active of inactive.

We next state three lemmas and conclude the proof of Theorem~\ref{th:meeting}, postponing the verification of the lemmas until the end of the section. Our main technical result states that the meeting of the two random walks happens w.h.p.\ when the coupling is in its active phase. This is done with the aid of two main steps. We first prove that w.h.p.\ the coupling can fail only during the first phase (see Lemma~\ref{lemma:set-fail} below). Afterwards in Lemma~\ref{lemma:set-B}, we check that it actually stays active for most of the time (Lemma~\ref{lemma:set-B}).
	
\begin{lemma}{\bf [Failures occur only during the first phase]} 
\label{lemma:set-fail}
Consider the event
		\begin{equation}
		\label{eq:fail}
		\cW \coloneqq \{\text{\em The coupling fails during the second phase before time $n^{3/2}$}\}.
		\end{equation}
Then
		\begin{equation}
	 	\lim_{n \to \infty} \P(\cW) = 0.
		\end{equation}
\end{lemma}
	
Consider the random subset of times at which the coupling is not active, that is, the random set
	\begin{equation}
	\label{def-B}
	B \coloneqq \{s \in [0,n^{3/2}] \colon \text{ the coupling is inactive at time } s\}.
	\end{equation}
	
\begin{lemma}{\bf [The coupling is active for most of the time]} 
\label{lemma:set-B}
	We have
	\begin{equation}
	\label{eq:B-size}
	\lim_{n \to \infty} \P(|B| >  n^{1/2 + 28\delta}) = 0.
\end{equation}
Moreover,
	\begin{equation}
	\label{eq:B-cc}
	\lim_{n \to \infty} \P \big( B\text{ \em has more than $n^{1/2+27\delta}$ connected components} \big) = 0.
	\end{equation}
\end{lemma}
	
Finally, the next lemma states that w.h.p.\ the meeting occurs when the coupling is active.
\begin{lemma}{\bf [No meetings when the coupling is inactive]}
\label{lemma:meet-A}
For $B$ as in \eqref{def-B},
	\begin{equation}
	\lim_{n \to \infty} \P(\tau_{\rm meet}^{\pi\otimes\pi} \in B)=0.
	\end{equation}
\end{lemma}

We are now in position to prove Theorem~\ref{th:meeting}.
\begin{proof}[Proof of Theorem~\ref{th:meeting}]
Using the coupling described in Section~\ref{ss.coup}, and Lemmas~\ref{lemma:set-B} and~\ref{lemma:meet-A}, we see that, for all $0 \leq  t \leq n^{3/2}$,
	\begin{equation}	
	\P(\tau^{\pi \otimes \pi}_{\rm meet} > t) = \P\big( \{\tau^{\pi \otimes \pi}_{\rm meet} > t\} \cap
	\{|B| \leq n^{\frac12+16\delta} \} \cap \{\tau^{\pi \otimes \pi}_{\rm meet} \not \in B\} \cap \cW^\complement \big) + o(1).
	\end{equation}
Recalling the definition of $\tau_{\rm final}$ in \eqref{eq:def-tau-final}, and using the above described coupling, we see that the latter probability can be bounded above by
		\begin{equation}
		\label{bound1}
		\P(\tau^{\pi\otimes\pi}_{\rm meet} > t) \leq \P({\tau_{\rm final}} > t-n^{\frac12+28\delta}) + o(1)
		\end{equation}
and from below by
		\begin{equation}
		\label{bound2}
		\P(\tau^{\pi\otimes\pi}_{\rm meet} > t) \geq \P({\tau_{\rm final}} > t) - o(1),
		\end{equation}
so \eqref{bound-prob} follows from Proposition \ref{lem:exponential-tau} by taking $0<\delta<\tfrac{1}{32}$.
		
To control the expectation of $\tau_{\rm meet}^{\pi \otimes \pi}$, we start by fixing $C_4>0$ and using~\eqref{bound1},~\eqref{bound2} and Proposition~\ref{lem:exponential-tau} to obtain
		\begin{align}
		\frac{\E[\tau_{\rm meet}^{\pi \otimes \pi}]}{n}
		& = \frac{1}{n} \int_{0}^\infty \P(\tau_{\rm meet}^{\pi \otimes \pi} > t)\, {\rm d}t \\
		& = (1+o(1)) \frac{1}{n} \int_{0}^{C_4n} \P(\tau_{\rm final} > t) \, {\rm d}t
		+ \frac{1}{n} \int_{C_4n}^\infty \P(\tau_{\rm meet}^{\pi \otimes \pi}>t) \, {\rm d}t.
		\end{align}
We are left to show that the latter integral is $o(n)$. It suffices to realise that, for any $t \geq 0$ and regardless of $(E_t,X_t,Y_t)$, the rate at which a rewiring occurs that puts the two random walks at distance $1$ is $d \nu$, and under the realisation of this event there is a non-negligible probability, say $p = p(d,\nu) > 0$, that the two random walks meet after a bounded amount of time. Therefore, for some $C_5 = C_5(d,\nu)>0$,
		\begin{equation}
		\P(\tau_{\rm meet}^{\pi \otimes \pi}>t) \leq \ee^{-C_5\,t/n},
		\end{equation}
which yields
		\begin{equation}
		\int_{C_4n}^\infty \P(\tau_{\rm meet}^{\pi \otimes \pi} > t) \, {\rm d}t \leq \frac{n}{C_5} e^{-C_5 C_4}.
		\end{equation}
At this point, given $\epsilon > 0$, choose $C_4 = C_4(d, \nu, \epsilon)>0$ sufficiently large as to have 
		\begin{equation}
		\frac{1}{2\vartheta_{d,\nu}} - \epsilon \leq \frac{1}{n} \int_{0}^{C_4n} \P(\tau_{\rm final} > t) \, {\rm d}t
		\leq \frac{1}{2\vartheta_{d,\nu}} + \epsilon,
		\end{equation}
and
		\begin{equation}
		\frac{1}{n} \int_{C_4n}^\infty \P(\tau_{\rm meet}^{\pi \otimes \pi}>t) \, {\rm d}t \leq \epsilon,
		\end{equation}
from which the desired claim follows immediately.
\end{proof}
	
\begin{proof}[Proof of Lemma~\ref{lemma:set-fail}]
Recall that, in order for the coupling to fail during the second phase, one of the events on the list immediately above \eqref{renew} must occur. 

If during the second phase of the coupling at some time $r \in [0,n^{3/2}]$ a rewiring occurs that does not involve the unique path between the two random walks and such that $\cB_{\hslash,r}(X_r) \cup \cB_{\hslash,r}(Y_r)$ is not a tree (cf.\ step (II-2$^{\rm nd}$-dyn-ii-B)), then we would have $[\cE^{\rm tree}_r]^\complement \cap [\cE^{\rm far}_r]^\complement$, and hence this kind of failure can be neglected thanks to Proposition \ref{prop:typical}. Similarly, the case when the coupling fails because $|E_{t}| > n^{12\delta}$ can be discarded with the aid of Proposition~\ref{lemma:E-far}. Lastly, if the coupling is active and in the second phase, in order for the event in (II-2$^{\rm nd}$-dyn-i-B) to occur at the next step, it is necessary to have a rewiring involving one of the (less than $4\hslash$) stubs in the unique path between the two random walks \emph{and} another stub in $\bar E_{s}$. Note that rewirings happen with rate bounded by $\frac{\nu}{4}|\bar{E}_{s}|$ and the probability that such a rewiring results in the event above can be bounded by
		\begin{equation}
		\frac{4\hslash \frac{\nu}{4} \frac{1}{dn-1} \times |\bar E_{s}|} {2|\bar E_{s}| \frac{\nu}{4} \frac{dn-2}{dn-1}} \leq c n^{-1+\delta}.
		\end{equation}
Finally, the time spent in the second phase of the coupling up to time $n^{3/2}$ is w.h.p\ at most $n^{3/2}n^{1-4\delta}$ due to Lemma~\ref{lemma:2nd-phase-new}. In particular, if $\cH$ denotes the event in~\eqref{def-H}, we can bound
\begin{equation}
\begin{split}
\P(\cW) & \leq \Prob_{d,\nu} ( \bar \tau_{\rm 2^{nd}}(n^{3/2}) > n^{3/2} n^{-1+4\delta}) + \P(\cH^\complement) \\
& \qquad + \P \Big( {\rm Poisson} \Big( \frac{\nu}{2} n^{3/2} n^{-1+4\delta} n^{12\delta} n^{-1+\delta} \Big) \geq 1 \Big) \\
& \leq \frac{\nu}{2}n^{-\frac{1}{2}+17\delta} + o(1),
\end{split}
\end{equation}
which converges to zero as soon as $\delta < \frac{1}{34}$, concluding the proof.
\end{proof}

\begin{proof}[Proof of Lemma \ref{lemma:set-B}]
Recall that, by Lemma \ref{lemma:set-fail}, w.h.p.\ no failures of the coupling will be registered during the second phase. Recall also that a failure occurring during the first phase must occur as a consequence of one of the events reported in the list immediately below \eqref{def-F}. 
		
We divide the proof into two steps. First we show that, under the event that the coupling is active and in the first phase, the probability that a failure occurs at the next jump is sufficiently small to take a union bound over the number of jumps. In this way we prove \eqref{eq:B-cc}. Afterwards we show that, uniformly in the realisation of $(E_t,X_t,Y_t)$ after the failure, w.h.p.\ the coupling will be reactivated within a short amount of time.
		
In what follows all inequalities hold for values of $n$ large enough. 
		
\medskip\noindent
{\bf 1.}
Let 
	\begin{equation}
	\label{eq:def-J}
	\cJ\coloneqq\big\{\text{$B$ has more than $n^{1/2+27\delta}$ connected components}\big\}\,.
	\end{equation}
We show that the probability of this event tends to zero. Indeed, under the event $\cH$ in \eqref{lemma:E-far}, the number of jumps of the process within time $n^{3/2}$ can be bounded above by the number of arrivals within the same time of a Poisson process with rate $2+dn^{12\delta}$. Hence, by Markov's inequality, w.h.p.\ within time $n^{3/2}$ there are at most $n^{3/2+13\delta}$ jumps. Assume that, at some time $0 \leq s \leq n^{3/2}$, the coupling is active and in the first phase, $\cB_{\hslash,s}(X_s)$ and $\cB_{\hslash,s}(Y_s)$ are two non-overlapping trees, ${\rm dist}_s(X_s,Y_s)=\infty$ and $|E_s| \leq n^{12\delta}$. Then the probability that after the next jump (say at time $r>s$) either the event $[\cE_r^{\rm trees}]^{\complement}$ or the event $\cF_r$ occurs can be bounded by
		\begin{equation}
		\label{eq:bound-bin}
		\begin{split}
		\P\left({\rm Bin} \Big( 2d^\hslash, \frac{|E_{s}|}{dn-|E_{s}|-1} \Big) > 0 \right)
		\leq n^{-1+13\delta}. 
		\end{split}
		\end{equation}
Therefore, if we call $J$ the number of connected components of $B$ that are generated by the occurrence of one of the events above, then we have
		\begin{equation}
		\P(\{J>n^{1/2+27\delta}\} \cap \cH) \leq \P\big( {\rm Bin}(n^{3/2+13\delta},n^{-1+13\delta})>n^{1/2+27\delta}\big) \to 0.
		\end{equation}
In conclusion, using Proposition \ref{lemma:E-far} and Lemma \ref{lemma:set-fail}, we have
		\begin{equation}
		\label{eq:bound-cc}
		\P(\cJ) \leq \P(\{J > n^{1/2+27\delta}\} \cap \cH \cap \cW^\complement ) + \P(\cH^\complement) + \P(\cW) \to 0,
		\end{equation}
which yields~\eqref{eq:B-cc}.

\medskip
We next control the size of the connected components of the set $B$. We first show that, for any $\varepsilon>0$, the probability that there exist a connected component of $B$ having length larger than $n^\varepsilon$ tends to zero sufficiently fast in $n$ for any $\delta>0$.
	
\medskip\noindent
	{\bf 2a.}
		Consider a given connected component $\hat B\subset B$ starting at some $0 \leq s< n^{3/2}$, and recall the definition of $\tau_{\rm renew}(s)$, $K(s)$, and $\tau_{\rm renew}^+(s)$ in \eqref{renew} and \eqref{renew+}. We work on the event $\cH$ in \eqref{def-H}. In order for the component to have size larger than a given $\theta>0$ (possibly depending on $n$), it must be the case that $\tau_{{\rm renew}}^+(s)>\theta$. In particular, conditioning on the value of $K(s) \in \N$, we get, for all $T \in \N$,
		 \begin{equation}
		 \begin{split}
		 	\P( \{|\hat B|>\theta\} \cap \cH) & \leq \P( \{K(s) \leq T\} \cap \{|\hat B|>\theta\} \cap \cH) + \P(K(s) > T) \\ 
		 	& \leq \P(\{\tau_{\rm renew}^{(T)}(s)>\theta\} \cap \cH) + \P(K(s) > T).
		 \end{split}
		 \end{equation}
As pointed out immediately after~\eqref{renew+}, $K(s)$ has Geometric distribution with rate arbitrarily close to one, provided $n$ is large enough. In particular, this implies
		 \begin{equation}\label{eq:exp-est}
		 	\P(K(s) \geq \ell) \leq 2^{-\ell} \quad \forall\, \ell \in \N\,,
		 \end{equation}
We now prove that there exists a universal constant $c>0$ such that, for $\theta \geq 3\log n$ and $T = \lceil \frac{\theta}{3\log n } \rceil$,
		 \begin{equation}\label{eq:max-of-exp-est}
		\P(\{\tau_{\rm renew}^{(T)}(s)>\theta\} \cap \cH) \leq \ee^{-c \frac{\theta}{\log n}},
		 \end{equation}
		 and therefore
		 \begin{equation}\label{eq:B-last-est}
		 	\P(\{|\hat B|>\theta\} \cap \cH) \leq \ee^{-c \frac{\theta}{\log n}}
			+ 2^{-\frac{\theta}{3\log n}} \qquad \forall\, T \in \N,\,\theta \geq 3 \log n.
		 \end{equation}
		 
		 \medskip\noindent
		 {\bf 2b.}
		 To prove \eqref{eq:max-of-exp-est}, note that, on the event $\cH$, the number of renewals before time $\theta$ is stochastically dominated by the sum of $T$ i.i.d.\ random variables with the same law as
		\begin{equation}
		Z = \max_{1\leq j \leq dn} E_j,
		\end{equation}
		where $(E_j)_{j \in \N}$ is a collection of i.i.d.\ exponential random variables with rate $\nu/4$. Therefore
		\begin{equation}\label{eq:max-of-exp-est-2}
		\P( \{ \tau_{\rm renew}^{(T)}(s)>\theta \} \cap \cH) \leq  \P\bigg( \sum_{i=1}^T Z_i > \theta\bigg) 
		\qquad \forall\,T \in \N, \theta>0.
		\end{equation}
		Note that
		\begin{equation}
		\E[Z] \leq 2\log n 
		\end{equation}
		and that $Z/\E[Z]$ satisfies a large deviation principle. Therefore, choosing $T= \frac{\theta}{3 \log n}$, we see that there exists a constant $c>0$ (independent of $n$) such that
		 \begin{equation}\label{eq:max-of-exp-est-3}
			 \P \bigg(\sum_{i=1}^T Z_i > \theta\bigg) \leq \P\bigg(\sum_{i=1}^T \frac{Z_i}{\E[Z]} 
			 > \frac{\theta}{2\log n}\bigg) \leq \ee^{-c T},\qquad \forall\,\theta>0.
		\end{equation}
Estimate~\eqref{eq:max-of-exp-est} follows by combining the above with~\eqref{eq:max-of-exp-est-2}.
	
		\medskip\noindent
		{\bf 2c.}
		We are left to show how the argument leading to \eqref{eq:B-last-est}, which is concerned with a single connected component, can be adapted to control the tail probability of multiple connected components, and how it eventually leads to the desired conclusion in \eqref{eq:B-size}. In particular, since all the estimates required to prove \eqref{eq:B-last-est} are uniform with respect to the history of the process, it follows that, for any $K \in \N$ (possibly depending on $n$) and any sequence $\theta_1, \dots, \theta_K$, where each element in the sequence is larger than $3\log n$,
		\begin{equation}\label{eq:bound_components}
			\P(\{|\hat B_i|>\theta_i\,\, \forall\, 1 \leq i \leq K\} \cap \cH) \leq \prod_{i=1}^{K} \ee^{-c \frac{\theta_i}{\log n}},
		\end{equation}		
for some absolute constant $c>0$, where $|\hat B_i|$ represents the $i$-th connected component in $B$ (with the convention that, if $B$ has $k$ connected component, then $|\hat B_i|=0$ for all $i>k$). Therefore, recalling the event $\cJ$ in \eqref{eq:def-J}, we see that \eqref{eq:B-size} follows at once via the estimate
		\begin{equation}
			\P(|B| >  n^{1/2+28\delta}) \leq \P(|B| >  n^{1/2+28\delta} \cap \cH \cap \cJ^\complement)
			+\P(\cH^\complement) + \P(\cJ).
		\end{equation}
Indeed, the last two terms on the right-hand side vanish thanks to Proposition \ref{lemma:E-far} and \eqref{eq:bound-cc}. As for the first term, note that, on the event $cJ^\complement$, $B$ has at most $n^{1/2+28\delta}$ connected components. In particular, if $|B| >  n^{1/2+28\delta}$, there exists at least one connected component with size at least $n^{\delta}$. Union bound combined with~\eqref{eq:bound_components} now yields
\begin{equation}
\P(|B| >  n^{1/2+28\delta} \cap \cH \cap \cJ^\complement) \leq n^{1/2+27\delta} \P(\{|\hat B_1| > n^{\delta}\} \cap \cH) \leq n^{1/2+27\delta} \ee^{-c\frac{n^{\delta}}{\log n}},
\end{equation}
which converges to zero as $n$ grows, concluding the proof.
\end{proof}

\begin{proof}[Proof of Lemma \ref{lemma:meet-A}]
By Corollary~\eqref{coro} and the definition of the event $\cW$ in~\eqref{eq:fail}, we have
		\begin{equation}
		\label{eq:first-bound}
		\P(\tau^{\pi\otimes\pi}_{\rm meet} \in B) = \P \Big( \{\tau^{\pi\otimes\pi}_{\rm meet} \in B \}
		\cap \{\cE^{\rm trees}_s \cup \cE_s^{\rm far}\,\, \forall\, 0 \leq s \leq n^{3/2} \}
		\cap \cW^\complement \Big) + o(1).
		\end{equation}
Under these events, there are only two ways in which the event $\{\tau_{\rm meet}^{\pi\otimes\pi} \in B\}$ can occur:
		\begin{enumerate}
		\item[\bf Case (1):] 
		for some time $s \in B$ the two random walks are a distance $2\hslash$ apart, each having a tree-like neighbourhood up to distance $\hslash$, and they meet before time $s + n^{\delta}$ by traversing the unique path of length $2\hslash$ that joins them before it gets destroyed.
		\item[\bf Case (2):] 
		for some time $s \in B$ there is a rewiring that puts the two random walks at distance $\ell < 2\hslash$, each having a tree-like $\hslash$-neighbourhood, and the two random walks meet before time $n^{\delta}$ by traversing the unique path of length $\ell < 2\hslash$ that joins them before this path gets destroyed.
		\end{enumerate}
Note that the probability of the event in {\bf Case (2)} can be bounded from above by the probability that $\tau_{1^{\rm st}}^{\rm tot} \leq n^\delta$ for the toy model in Section~\ref{sec:meet}. Hence, by Lemma~\ref{le:exponential-tau},
		\begin{equation}
		\label{eq:case2}
		\P(\textbf{Case (2)}) \to 0.
		\end{equation}
On the other hand, in order for the event in {\bf Case (1)} to occur, there must exist a time $0\le s\le n^{3/2}$ at which the two random walks are at distance exactly $2\hslash$ from each other. Under the event $\cH$ in Proposition~\ref{lemma:E-far}, the number of jumps of the process within time $n^{3/2}$ is w.h.p.\ at most $n^{3/2+12\delta}$. Taking a union bound over the jump times, we get
		\begin{equation}
		\P(\textbf{Case (1)}) \leq n^{\frac{3}{2}+12\delta} q(\hslash)
		\leq n^{\frac{3}{2}+12\delta} \times \prod_{i=0}^{\hslash-1}\frac{\Delta(i)}{\beta},
		\end{equation}
where $q$ is defined as in~\eqref{def-q}. Recall that $\Delta(i) \to 0$ as $i$ grows. Hence, provided $n$ is taken large enough, $\Delta(i) \leq d^{-\frac{4}{\delta}}$ for all $i \geq \hslash/2$. Moreover, for the same reason there exists a constant $C_6=C_6(d,\nu)>0$ such that $\prod_{i=0}^{\hslash/2} \Delta(i) \leq C_6$.Using that $\beta > 1$ and $\delta < \tfrac{1}{25}$ immediately implies
		\begin{equation}
		\label{eq:case1}
		\P(\textbf{Case (1)}) \leq C_6 n^{\frac{3}{2}+12\delta} d^{-\frac{4}{\delta} \frac{\hslash}{2}} \to 0.
		\end{equation}
Hence, the desired conclusion follows from~\eqref{eq:first-bound},~\eqref{eq:case2}, and~\eqref{eq:case1}.
\end{proof}
	
Before concluding this section, we prove the following adaptation of Theorem~\ref{th:meeting}, which deals with the case where the two random walks start at the extremes of a randomly sampled edge and will be of use later.	
\begin{proposition}{\bf [Tail of $\tau_{{\rm meet}}^{\rm edge}$ on an intermediate time scale]}
\label{lemma-new}
For every $\nu > 0$, $d \geq 3$ and any positive sequence $(s_n)_{n \in \N}$ such that $1 \ll s_n \ll n$,
		\begin{equation}
		\lim_{n \to \infty} \P_{\mu_d}(\tau_{{\rm meet}}^{\rm edge} > s_n) = \vartheta_{d,\nu}.
		\end{equation}
\end{proposition}
	
\begin{proof}
Note that the result can be obtained by means of the same coupling as above, by starting the coupling from the second phase with a realisation of $(E_0,X_0,Y_0)$ such that $X_0$ and $Y_0$ are neighbouring vertices, and their (almost overlapping) $\hslash$-neighbourhoods are trees. The two random walks are initially coupled to the process $(\hat Z_t)_{t\ge 0}$ as in Section \ref{sec:RW} with $\hat{Z}_0=1$. The probability that the process is absorbed before hitting $0$ is given by
		\begin{equation}
		\Pr(H_\dagger < H_0 \mid \hat{Z}_0 = 1 ) = \frac{1}{2R_{d,\nu}(0)} = \vartheta_{d,\nu},
		\end{equation}
where we simply use the transience of the process $(\hat Z_t)_{t\ge 0}$, the fact that the process jumps uniformly with rate at least $2$, and Proposition~\ref{lemma-R}. Clearly, the stopping time $H_{\{0, \dagger\}}$ is bounded independently of $n$, and the probability that the coupling fails in the second phase vanishes. If the two random walks do not meet during such a second phase of the coupling, then the process is reinitialised with high probability in the first phase, and then the probability of a meeting within time $o(n)$ is asymptotically negligible.
\end{proof}

\begin{remark}
\label{rmk:heuristic}
In this section we have made rigorous the heuristics mentioned in Section~\ref{sec:RW}, which can be summarised by saying that the dynamics of two random walks on the finite dynamic graph is well approximated by the dynamics of two random walks on the infinite dynamic tree. In particular, Proposition~\ref{lemma-new} can be translated into the claim that $1-\vartheta_{d, \nu}$ is the probability that two independent random walks ever meet when running on a regular tree with disappearing edges, starting from neighbouring vertices. With this idea in mind, we get a heuristic argument for the second limit in~\eqref{eq:per-theta-a} and the limit~\eqref{eq:per-theta-b}. Indeed, when either $d$ or $\nu$ is large, the only way in which a meeting can occur is when one of the two random walks moves before the edge joining them disappears, and moves exactly along that edge. In this way we get the factor $\frac1d\times \frac{2}{2+\nu}$, which is in line with Proposition~\ref{pr:der-theta-0}. A similar argument applies to the first limit in~\eqref{eq:per-theta-a}. Note that the probability that the two random walks ever meet in the dynamic setting is smaller than in the static setting, since an edge along the unique path joining the two random walks may be rewired.
\end{remark}

%%%%%%%%%%%%%%% SECTION 6 %%%%%%%%%%%%%%%%%%%%%%%
	
\section{Convergence to the Fisher-Wright diffusion}
\label{sec:CCC}

This section is devoted to the proof of Theorem~\ref{thm:disc} by showing that the condition in \eqref{mfc-new} holds when
\begin{equation}
\gamma_n = \alpha_n \text{ with } \alpha_n \coloneqq \frac{n}{2\vartheta_{d,\nu}}.
\end{equation} 
The proof is based on an adaptation of the arguments developed in \cite[Section 6]{CCC16} to the dynamic set-up, with some simplifications due to the availability of Theorem \ref{th:meeting}. Along the way, we prove another key result, namely Proposition \ref{separate2-new}, which is needed to control the expected number of discordant edges on short time scales, i.e., of order $o(n)$. The proof of this proposition can be found in Section \ref{suse:lemma-disc}. 

%%%

\subsection{Proof of convergence}
\label{suse:WF}
	
In this section we prove Theorem~\ref{thm:disc}. Recall that, by Theorem~\ref{th:meeting}, $\alpha_n\sim \E_{\mu_d}[\tau_{{\rm meet}}^{\pi\otimes\pi}]$ as $n\to\infty$. We verify that the convergence in \eqref{mfc-new} holds in mean for $\gamma_n= \alpha_n$. The convergence in \eqref{mfc-new} follows from Markov's inequality, and we deduce the validity of Theorem~\ref{thm:disc} as a consequence.

Let us look at the expectation of the left-hand side of~\eqref{mfc-new} for $T>0$ with $\gamma_n= \alpha_n$. Define $\delta_n = \frac{1}{n} \log^2 n$, and split the integral over $[0,T]$ in~\eqref{mfc-new} into the subintervals $[0,\delta_n]$ and $[\delta_n,T]$ as follows:
		\begin{equation}
		\label{split}
		\begin{split}
		I(0,T)
		& \coloneqq \E_{\mu_d,u} \left[ \left| \frac{\alpha_n}{n}\int_{0}^{T} \mathcal{D}_{\alpha_n s} \, \dd s 
		- \int_{0}^{T} \mathcal{O}_{\alpha_n s} \big( 1-\mathcal{O}_{\alpha_n s}\big) \, \dd s \right| \right] \\ 
		& \,\leq I(0,\delta_n)+ I(\delta_n,T).
		\end{split}
		\end{equation}
In the following, we prove that both terms in the right-hand side of the expression above vanish as $n$ grows.

For the first term, since $\cO_t,\cD_t \in [0,1]$ for all $t \geq 0$,
		\begin{equation}
		\label{shortTime}
		\begin{split}
		I(0,\delta_n) & = \E_{\mu_d,u} \left[ \left| \frac{\alpha_n}{n} \int_{0}^{\delta_n} 
		\mathcal{D}_{\alpha_n s} \, \dd s 
		- \int_{0}^{\delta_n} \mathcal{O}_{\alpha_n s} \big( 1-\mathcal{O}_{\alpha_n s}\big) \, 
		\dd  s \right| \right] \\
		& \leq \frac{\alpha_n}{n}\delta_{n}+\delta_{n}.
		\end{split}
		\end{equation}
Once one observes that $\alpha_n \asymp n$ and $\delta_n \to 0$, it follows that 
	\begin{equation}
	\lim_{n \to \infty} I(0,\delta_n) = 0.
	\end{equation}
Let us now turn our attention to $I(\delta_n,T)$. We abbreviate 
		\begin{equation}
		\label{conditioning}
 		Q_n(s) \coloneqq \frac{\alpha_n}{n}\left( \mathcal{D}_{\alpha_n s}
		- \E_{\mu_d,u}\left[\mathcal{D}_{\alpha_n s} | \mathcal{F}_{s-\delta_n}\right]\right),
		\end{equation}
with $(\mathcal{F}_{t})_{t\geq 0}$ the joint filtration in Proposition~\ref{prop:martingales}, and estimate
			\begin{align}
			&I(2\delta_n,T)\\
			\label{split2.1} 
			&\quad \leq \sqrt{\E_{\mu_d,u}\left[ \left(\int_{\delta_n}^T Q_n(s) \, \dd s \right)^2 \right]}\\ 
			\label{split2.2}
			&\quad + \E_{\mu_d,u}\left[\int_{\delta_n}^T \left| \frac{\alpha_n}{n} \E_{\mu_d,u}\left[ 
			\mathcal{D}_{\alpha_n s}| \mathcal{F}_{s-\delta_n}\right]
			- \mathcal{O}_{\alpha_n(s-\delta_n)} \big( 1-\mathcal{O}_{\alpha_n(s-\delta_n)}\big) \right| 
			\dd s \right]\\ 
			\label{split2.3} 
			&\quad + \E_{\mu_d,u}\left[\left|\int_{\delta_n}^T \left[\mathcal{O}_{\alpha_n(s-\delta_n)} 
			\big(1-\mathcal{O}_{\alpha_n(s-\delta_n)}\big) 
			- \mathcal{O}_{\alpha_n s} \big(1-\mathcal{O}_{\alpha_n s}\big)\right] \dd s \right|\right].
			\end{align}
We deal with the three expressions in the right-hand side separately. 

Consider~\eqref{split2.3}. Since the integrand is bounded, a simple time-shift yields that the integral is bounded by $2\delta_n$, which converges to zero as $n$ grows. Consider next~\eqref{split2.1}. Observe that, for any $r \in [\delta_n, s-\delta_n)$, by the tower property and the definition in \eqref{conditioning}, we have
		\begin{equation}
		\E_{\mu_d,u}\left[Q_n(s) \mid \mathcal{F}_{r}\right] = 0,
		\end{equation}
from which we get
		\begin{equation}  
		\E_{\mu_d,u}\left[Q_n(s)Q_n(r)\right]
		= \E_{\mu_d,u}\left[Q_n(r)\, \E_{\mu_d,u}\left[Q_n(s) \mid \mathcal{F}_{r}\right]\right] = 0.
		\end{equation}
Therefore we can write 
		\begin{align}
		\E_{\mu_d,u} \Bigg[ \bigg(\int_{\delta_n}^T Q_n(s) \, {\rm d} s \bigg)^2 \Bigg]
		= 2\, \E_{\mu_d,u} \Bigg[\int_{\delta_n}^T Q_n(s) \bigg(\int_{s}^{T \wedge (s+\delta_{n})} 	Q_n(r)\,\dd r \bigg) \dd s \Bigg].
		\end{align}
Expand $Q_{n}(s)Q_{n}(r)$ to obtain
		\begin{align}
		\label{eq:big_bound}
		&\E_{\mu_d,u}\Bigg[\int_{\delta_n}^T \bigg( \int_{s}^{T \wedge (s+\delta_{n})}Q_n(r)\, \dd r \bigg)  Q_n(s) \,\dd s \Bigg] \\
		\label{split3.1}
		&= \E_{\mu_d,u} \left[ \int_{\delta_n}^T \left(\int_{s}^{T \wedge (s+\delta_{n})} \Big( \frac{\alpha_n}{n} \Big)^{2}
		 \mathcal{D}_{\alpha_n{s}} \mathcal{D}_{\alpha_n{r}} \, \dd r\right) \, \dd s\right] \\
		\label{split3.2}
		&-\int_{\delta_n}^T \left(\int_{s}^{T \wedge (s+\delta_{n})} \Big( \frac{\alpha_n}{n} \Big)^{2} \E_{\mu_d,u}
		 \left[ \mathcal{D}_{\alpha_n{s}} \E_{\mu_d,u}\left[\mathcal{D}_{\alpha_n r} 
		 \middle| \mathcal{F}_{r-\delta_n} \right] \right] \, \dd r\right) \, \dd s \\
		\label{split3.3}
		&-\int_{\delta_n}^T \left(\int_{s}^{T \wedge (s+\delta_{n})} \Big( \frac{\alpha_n}{n} \Big)^{2} \E_{\mu_d,u} 
		\left[ \mathcal{D}_{\alpha_n{r}} \E_{\mu_d,u}\left[\mathcal{D}_{\alpha_n s} 
		\middle| \mathcal{F}_{s-\delta_n} \right] \right] \, \dd r\right) \, \dd s \\
		\label{split3.4}
		&- \int_{\delta_n}^T \left(\int_{s}^{T \wedge (s+\delta_{n})} \Big( \frac{\alpha_n}{n} \Big)^{2} \E_{\mu_d,u}
		 \left[\E_{\mu_d,u}\left[\mathcal{D}_{\alpha_n r} \middle| \mathcal{F}_{r-\delta_n} \right] \E_{\mu_d,u}
		 \left[\mathcal{D}_{\alpha_n s} \middle| \mathcal{F}_{s-\delta_n} \right] \right] \, \dd r\right) \,\dd s.
		\end{align}
We next show that the right-hand side tends to zero by considering each term in the right-hand side of \eqref{eq:big_bound} separately. To this aim, a key observation is that for all $r>s\ge 0$, conditionally on $\cF_{s\alpha_n}$,  the expected fraction of discordant edges at time $r\alpha_n$ can be bounded, via \eqref{exp-D-sup} and Markov's inequality, by
			\begin{equation}
			\label{eq:goodbnd}
			\E_{\mu_d,u} \left[ \mathcal{D}_{\alpha_n r} \middle| \mathcal{F}_{s\alpha_n} \right]
			\leq 2\P_{\mu_d} \big(\tau_{{\rm meet}}^{\rm edge} > \alpha_{n} (r-s) \big).
			\end{equation}
It will be convenient to write
		\begin{equation}
		\label{eq:at}
		A(t)\coloneqq\limsup_{n\to\infty}\frac{\alpha_{n}}{n} \int_0^t \P_{\mu_d} 
		\big(\tau_{{\rm meet}}^{\rm edge} > \alpha_{n} s \big) \, \dd s \le \frac{t}{2\vartheta_{d,\nu}}.
		\end{equation}

Consider \eqref{split3.1}. By conditioning at time $0<s<r$ and using \eqref{eq:goodbnd} and \eqref{exp-D-u}, we obtain
		\begin{equation}
		\begin{split}
		\Big( \frac{\alpha_n}{n} \Big)^{2}& \int_{\delta_n}^T \bigg( \int_{s}^{T \wedge (s+\delta_{n})}   
		\E_{\mu_d,u}  \left[ \mathcal{D}_{\alpha_n s} \mathcal{D}_{\alpha_n r} \right]\,\dd r \bigg) \, \dd s\\ 
		& = \Big( \frac{\alpha_n}{n} \Big)^{2}\int_{\delta_n}^T \bigg(  \int_{s}^{T \wedge (s+\delta_{n})} 
		\E_{\mu_d,u} \Big[\mathcal{D}_{\alpha_n s} \E_{\mu_d,u} \big[ \mathcal{D}_{\alpha_n r} \big| 
		\mathcal{F}_{\alpha_n s} \big] \Big]\, \dd r \bigg) \, \dd s\\
		& \leq 2 \Big( \frac{\alpha_n}{n} \Big)^{2} \int_{\delta_n}^T \bigg( \int_{s}^{T \wedge (s+\delta_{n})} 
		\E_{\mu_d,u} \big[\mathcal{D}_{\alpha_n s} \big]  \P_{\mu_d}\big( \tau_{\rm meet}^{\rm edge} 
		> \alpha_n (r-s) \big) \, \dd r \bigg)\,\dd s \\
		& \leq 2\frac{\alpha_n}{n} \int_{\delta_{n}}^{T} \P_{\mu_d} \big(\tau_{\rm meet}^{\rm edge} 
		> \alpha_{n} s\big) \, \dd s \times \frac{\alpha_n}{n} \int_{0}^{\delta_{n}} \P_{\mu_d} 
		\big(\tau_{\rm meet}^{\rm edge}> \alpha_n r \big)\, \dd r,
		\end{split}
		\end{equation}
and the latter tends to zero as $n$ grows, thanks to~\eqref{eq:at} combined with the facts that $\alpha_n \asymp n$ and $\delta_n \to 0$.
		
The arguments for \eqref{split3.2} and \eqref{split3.3} are similar and for this reason we consider only the first. By \eqref{eq:goodbnd} and \eqref{exp-D-u}, we obtain
		\begin{equation}
		\begin{split}
		\Big( \frac{\alpha_n}{n} \Big)^{2} & \int_{\delta_n}^T \bigg( \int_{s}^{T \wedge (s+\delta_{n})} 
		\E_{\mu_d,u} \left[  \mathcal{D}_{\alpha_n s} \E_{\mu_d,u}\left[\mathcal{D}_{\alpha_n r} 
		\middle| \mathcal{F}_{r-\delta_n} \right] \right] \, \dd r \bigg) \, \dd s \\
		& \leq 4\Big( \frac{\alpha_n}{n} \Big)^{2} \int_{\delta_n}^T \P_{\mu_d} 
		\big(\tau_{{\rm meet}}^{\rm edge}> \alpha_n s \big) \bigg( \int_{s}^{T \wedge (s+\delta_{n})} 
		\P_{\mu_d} \big(\tau_{{\rm meet}}^{\rm edge} > 2\alpha_n  \delta_{n} \big) \, \dd r \bigg) \, \dd s \\
		& \leq 4 \delta_n \Big( \frac{\alpha_n}{n} \Big)^{2} \int_{\delta_n}^T \P_{\mu_d} 
		\big(\tau_{{\rm meet}}^{\rm edge}> \alpha_n s \big) \, \dd s \\
		& \leq 4 \delta_n \Big( \frac{\alpha_n}{n} \Big)^{2} T,
		\end{split}
		\end{equation}
and the latter vanishes as $n$ grows. 

Finally, to control \eqref{split3.4}, we once again apply~\eqref{eq:goodbnd}, to obtain
		\begin{equation}
		\begin{split}
		\Big( \frac{\alpha_n}{n} \Big)^{2}  &\int_{\delta_n}^T \int_{s}^{T \wedge (s+\delta_{n})} 
		\E_{\mu_d,u} \left[\E_{\mu_d,u}\left[\mathcal{D}_{\alpha_n r} \middle| \mathcal{F}_{r-\delta_n} \right] 
		\E_{\mu_d,u}\left[\mathcal{D}_{\alpha_n s} \middle| \mathcal{F}_{s-\delta_n} \right] \right] \, \dd r\, \dd s \\ 
		& \leq 4 \Big( \frac{\alpha_n}{n} \Big)^{2} \int_{\delta_n}^T \bigg( \int_{s}^{T \wedge (s+\delta_{n})} 
		\P_{\mu_d} \big(\tau_{{\rm meet}}^{\rm edge} > 2\alpha_{n} \delta_{n} \big)^{2} \, \dd r \bigg) \, \dd s \\
		& \leq 4 \delta_n T \Big( \frac{\alpha_n}{n}\,\P_{\mu_d} \big(\tau_{{\rm meet}}^{\rm edge}> \alpha_{n} 
		\delta_{n} \big) \Big)^{2},
		\end{split}
		\end{equation}
and the latter tends to zero in the limits as $n$ grows. 

It remains to control \eqref{split2.2}, for which the following lemma is needed.

\begin{proposition}{\bf [Expected discordances on short time scales]}
\label{separate2-new} 
Fix a sequence $\delta_n$ such that $\frac{\log n}{n} \ll \delta_n \ll 1$, and define
		\begin{equation}
		\label{epsilon-new}
		\epsilon_n \coloneqq \sup_{\xi\in\{0,1\}^{n}} 
		\Big| \E_{\mu_d, \xi} \big[ \mathcal{D}_{\delta_n\alpha_n} \big]
		- 2\vartheta_{d,\nu} \mathcal{O}^n(\xi) \big(1-\mathcal{O}^n(\xi) \big) \Big|,
		\end{equation}
where $\mathcal{O}^n(\xi)=\frac{1}{n} \sum_{x \in [n]} \xi(x)$. Then
		\begin{equation}
		\label{Hammer2}
		\lim_{n \to \infty} \epsilon_n = 0.
		\end{equation}
\end{proposition}

\noindent
Proposition~\ref{separate2-new} says that, regardless of the initial opinion configuration, on short time scales the expected density of discordant edges is proportional to product of the initial densities of the two opinions up to multiplication by $2\vartheta_{d,\nu}$. This will be of use later on in the proof of Theorem~\ref{thm:disc}.
	
We postpone the proof of Proposition~\ref{separate2-new} to the end of this section. Note that the integrand in~\eqref{split2.2} is uniformly bounded by $\frac{1}{2\vartheta_{d,\nu}}\epsilon_{n}$ defined in \eqref{epsilon-new}, and hence the integral tends to zero by Proposition~\ref{separate2-new}. This concludes the proof of Theorem~\ref{thm:disc}.
\qed

%%%	

\subsection{Discordances on short time scales}
\label{suse:lemma-disc}
	
In this section we prove Proposition \ref{separate2-new}. Before presenting the proof, we state two intermediate results that will be of use later. The first is a control on the mixing time of the random walk on the dynamic random graph.

\begin{proposition}{\bf [Random walk mixing time]}
\label{pr:mixing}
\label{DynamicMixingMarginal} 
For $t\ge 0$ consider the worst-case total variation distance on the product space
		\begin{equation}
		\label{eq:mixing_distance}
		d_{\rm TV}(t) \coloneqq \max_{(x,G)\in E} \max_{(A,B) \subset[n] \times \cG_n(d)} 
		\big| \P_{G}(X^x_t \in A, G_t \in B) - \pi(A) \mu_{d}(B) \big|,
		\end{equation}
and let 
		\begin{equation}
		\label{eq:mixing}
		t_{\rm mix} \coloneqq \inf\{t \geq 0 \colon d(t) \leq (2\ee)^{-1}\}.
		\end{equation}	
There exists a positive constant $C_0=C_0(d,\nu)>0$ such that
		\begin{equation}
		\label{MixBound}
		t_{\rm mix} \leq C_0 \log n.
		\end{equation}
\end{proposition}
	
\begin{proof}
Consider the stopping time $\tau_{\rm env}$ at which every edge in $G_0=G$ has been rewired. It is immediate that $\tau_{\rm env}$ is a strong stationary time for the Markov process obtained by projecting $(G_t,X^x_t)_{t\ge 0}$ onto the first coordinate, i.e.,
		\begin{equation}
		\P_G(G_t=G'\mid\tau_{\rm env}\le t)=\mu_{d}(G'), \qquad \forall\, G,G' \in \cG_d(n).
		\end{equation}
Consider next the stopping time $\tau_{\rm rw}$ as the first time after $\tau_{\rm env}$ at which the following events are all satisfied:  
\begin{itemize}
	\item the random walk arrives in a vertex $v$;
	\item all the stubs of $v$ have been rewired before the random walks does a single step;
	\item the random walk does a step.
\end{itemize}
It is immediate to check that
		\begin{equation}
		\P_G(G_t=G', X_t^x=y \mid \tau_{\rm rw} \leq t)=\mu_{d}(G') \pi(y)
		\qquad \text{for all } G, G' \in \cG_d(n) \text{ and all } x, y \in [n].
		\end{equation}
Therefore $\tau_{\rm rw}$ is a strong stationary time for the joint Markov chain, and hence
		\begin{equation}
		\label{eq:sst}
		t_{\rm mix} \leq \inf \{t \geq 0 \colon \P_G(\tau_{\rm rw}>t) \leq (2\ee)^{-1}\}.
		\end{equation}
To control the tail of $\tau_{\rm rw}$ it is enough to note that $\tau_{\rm rw}=\tau_{\rm env}+Z$, where $Z$ is distributed as the sum of two independent random variables having law ${\rm Exp}(1)$. On the other hand, $\tau_{\rm env}$ is stochastically dominated by the maximum of $dn$ exponential random variables of rate $\nu/4$. Hence, $d$ and $\nu$ being bounded, for every $\varepsilon>0$ there exists some $C_0=C_0(d,\nu)>0$ such that, for $n$ large enough,
		\begin{equation}
		\label{eq:bound-sst}
		\P_G(\tau_{\rm rw} > C_0 \log n) \leq (2\ee)^{-1}.
		\end{equation}
The claim follows via \eqref{eq:sst} and \eqref{eq:bound-sst}.
\end{proof}

The next result clarifies how we can make use the bound on the mixing time provided by Lemma~\ref{pr:mixing} to prove Proposition~\ref{separate2-new}. The proof follows the lines of \cite[Proposition 6.1]{CCC16}, but we repeat the full argument to show the reader how to adapt it to our dynamic set-up.

\begin{lemma}{\bf [Expected discordances on arbitrary timescale]}
\label{lemma:bound_TV}
For every $0 < s < t < \infty$,
		\begin{equation}
		\label{eq:lemmaboundTV}
		\begin{split}
		& \sup_{\xi \in \{0,1\}^{n}} \Big| \E_{\mu_d,\xi}[\mathcal{D}_{t}] 
		- 2\P_{\mu_d}( \tau_{{\rm meet}}^{\rm edge}>s) \cO^n(\xi) \big(1-\cO^n(\xi) \big) \Big| \\
		& \qquad \qquad\qquad\leq 2\,\P_{\mu_d} \big( \tau_{{\rm meet}}^{\rm edge}\in (s,t] \big) 
		+ 4\,\P_{\mu_d} \big( \tau_{{\rm meet}}^{\rm edge} >s \big)\, d_{\rm TV}(t-s).
		\end{split}
		\end{equation}
\end{lemma}
	
\begin{proof}
By the definition of $\E_{\mu_d,\xi}[\mathcal{D}_{t}]$, duality, and the reversibility of the dynamic random environment with respect to $\mu_d$, we obtain
		\begin{equation}
		\label{eq:giant3}
			\begin{aligned}
			& \E_{\mu_d,\xi}[\mathcal{D}_{t}] \\
			& = \frac{2}{dn} \sum_{G, G' \in \cG(d)} \mu_d(G) \sum_{\substack{x, y \in [n] \\ x \neq y}}
			\sum_{1 \leq i,j \leq d} \ind_{\sigma_{x,i}\leftrightarrow_{G'}\sigma_{y,j}}
			\E_{G,\xi} [\xi(\hat{X}^x_{t,t})(1-\xi(\hat{X}^y_{t,t})) \ind_{G_t=G'} \ind_{\hat{\tau}_{{\rm meet},t}^{x,y}>t} ] \\
			& = \frac{2}{dn}\sum_{G, G' \in \cG(d)} \mu_d(G') \sum_{\substack{x, y \in [n] \\ x \neq y}} \sum_{1 \leq i,j \leq d}
			\ind_{\sigma_{x,i} \leftrightarrow_{G'} \sigma_{y,j}} \E_{G',\xi}[\xi({X}^x_{t})(1-\xi({X}^y_{t})) \ind_{G_t=G} \ind_{{\tau}_{{\rm meet}}^{x,y}>t} ] \\
			& = \frac{2}{dn} \sum_{G \in \cG(d)} \sum_{\substack{ x, y \in [n] \\ x \neq y}} \sum_{1 \leq i,j \leq d}
			\E_{\mu_d,\xi}[ \xi({X}^x_{t})(1-\xi({X}^y_{t})) \ind_{G_t=G} \ind_{\sigma_{x,i}\leftrightarrow_{0}
			\sigma_{y,j}} \ind_{{\tau}_{{\rm meet}}^{x,y}>t}] \\
		    & = \frac{2}{dn} \sum_{\substack{ x, y \in [n] \\ x \neq y}} \sum_{1 \leq i,j \leq d}
		    \E_{\mu_d,\xi} [\xi({X}^x_{t})(1-\xi({X}^y_{t})) \ind_{\sigma_{x,i}\leftrightarrow_{0}\sigma_{y,j}} \ind_{{\tau}_{{\rm meet}}^{x,y}>t} ]\\
			& \leq \frac{2}{dn} \sum_{\substack{x, y \in [n] \\ x \neq y}} \sum_{1 \leq i,j \leq d}	
			\E_{\mu_d,\xi} [\xi({X}^x_{t})(1-\xi({X}^y_{t})) \ind_{\sigma_{x,i}\leftrightarrow_{0}\sigma_{y,j}} \ind_{{\tau}_{{\rm meet}}^{x,y}>s}] + 2\P_{\mu_d} (\tau_{\rm meet}^{\rm edge} \in (s,t]),
		\end{aligned}
		\end{equation}
where in the last inequality we split according to $s$, bound $\xi({X}^x_{t})(1-\xi({X}^y_{t}))$ by $1$, perform the sums, and use the definition of $\tau_{\rm meet}^{\rm edge}$.

We next apply Markov's property for the joint evolution of $(X_r^x,X_r^y,G_r)_{r\ge 0}$ and the random graph dynamics to obtain
			\begin{equation}
			\label{eq:giant2}
			\begin{aligned}
				&\E_{\mu_d,\xi}[\xi({X}^x_{t})(1-\xi({X}^y_{t}))\ind_{\sigma_{x,i}\leftrightarrow_0\sigma_{y,j}}
				\ind_{{\tau}_{{\rm meet}}^{x,y}>s}] \\
				&\quad=\E_{\mu_d,\xi}\bigg[\E_{\mu_d,\xi}[\xi({X}^x_{t})(1-\xi({X}^y_{t}))\ind_{\sigma_{x,i}
				\leftrightarrow_0\sigma_{y,j}}\ind_{{\tau}_{{\rm meet}}^{x,y}>s}\mid (X_r^x,X_r^y,G_r)_{r\in [0,s]}]\bigg] \\
				&\quad=\E_{\mu_d,\xi}\bigg[\E_{\mu_d,\xi}[\xi({X}^x_{t})(1-\xi({X}^y_{t}))\mid 
				(X_r^x,X_r^y,G_r)_{r\in [0,s]}]\ind_{\sigma_{x,i}\leftrightarrow_0\sigma_{y,j}}
				\ind_{{\tau}_{{\rm meet}}
				^{x,y}>s}\bigg] \\
				&\quad=\E_{\mu_d,\xi}\bigg[\E_{\mu_d,\xi}[\xi({X}^{X_s^x}_{t-s})(1-\xi({X}^{X^y_s}_{t-s}))\mid 
			        (X_r^x,X_r^y,G_r)_{r\in [0,s]}]\ind_{\sigma_{x,i}\leftrightarrow_0\sigma_{y,j}}
			        \ind_{{\tau}_{{\rm meet}^{x,y}>s}}\bigg] \\
				&\quad=\E_{\mu_d,\xi}\bigg[\E_{\mu_d,\xi}[\xi({X}^{X_s^x}_{t-s})(1-\xi({X}^{X^y_s}_{t-s}))
				-\cO^n (\xi)(1-\cO^n(\xi))\mid (X_r^x,X_r^y,G_r)_{r\in [0,s]}] \\
				&\qquad\times\ind_{\sigma_{x,i}\leftrightarrow_0\sigma_{y,j}}\ind_{{\tau}_{{\rm meet}}^{x,y}>s}\bigg]
				+\cO^n(\xi)(1-\cO^n(\xi))\P_{\mu_d}\Big(\{\tau_{\rm meet}^{x,y}>s\}\cap\{\sigma_{x,i}
				\leftrightarrow_0\sigma_{y, j} \}\Big).
			\end{aligned}
			\end{equation}
At this point, the proof is complete as soon as we bound the difference in the first term on the right-hand side of~\eqref{eq:giant2} by 
\begin{equation}
2\P_{\mu_d}\Big( \{\tau_{\rm meet}^{x,y}>s\} \cap \{\sigma_{x,i}\leftrightarrow_0\sigma_{y, j} \} \Big) d_{\rm TV}(t-s).
\end{equation}
To do so, observe that
		\begin{equation}
		\label{eq:giant}
		\begin{aligned}
				\Big| \E_{\mu_d,\xi}& \Big[ \E_{\mu_d,\xi} \big[  \xi(X^{X^{x}_{s}}_{t-s}) (1-\xi(X^{X^y_{s}}_{t-s})) 
				- \mathcal{O}^n(\xi)(1-\mathcal{O}^n(\xi)) \big| (X_r^x,X_r^y,G_r)_{r\in [0,s]} \big]\\
				&\qquad\qquad\qquad\qquad\qquad\qquad\qquad\qquad\qquad\quad
				\times\ind_{\sigma_{x,i}\leftrightarrow_0\sigma_{y,j}}\ind_{{\tau}_{{\rm meet}}^{x,y}>s} \Big] \Big| \\
				&= \Big| 
					\E_{\mu_d,\xi} \Big[\Big( 
					\E_{\mu_d,\xi} \big[  \xi(X^{X^{x}_{s}}_{t-s})  \mid (X_r^x,G_r)_{r\in [0,s]}\big]
					\E_{\mu_d,\xi} \big[1-\xi(X^{X^y_{s}}_{t-s})
				 \mid (X_r^y,G_r)_{r\in [0,s]}\big]\\
				 &\qquad\qquad\qquad\qquad\qquad\qquad\qquad - \mathcal{O}^n(\xi)(1-\mathcal{O}^n(\xi))
				 \Big)
				 \ind_{\sigma_{x,i}\leftrightarrow_0\sigma_{y,j}}
				 \ind_{{\tau}_{{\rm meet}}^{x,y}>s}
				  \Big] \Big| \\
				& \leq \E_{\mu_d,\xi} \Big[ \mathcal{O}^n(\xi) \Big| \mathcal{O}^n(\xi) 
				- \E_{\mu_d,\xi} \big[\xi(X^{X^{y}_{s}}_{t-s}) \big| (G_r,X_r^y)_{r\in [0,s]} \big] \Big|\ind_{\sigma_{x,i}
				\leftrightarrow_0\sigma_{y,j}}
				\ind_{{\tau}_{{\rm meet}}^{x,y}>s} \Big] \\
				& \quad + \E_{\mu_d,\xi} \Big[ \E_{\mu_d,\xi} \big[ 1-\xi(X^{X^{y}_{s}}_{t-s})
				 \big| (G_r,X_r^y)_{r\in [0,s]}  \big]\\
				&\qquad\qquad\quad\times \Big| \mathcal{O}^n(\xi) 
				- \E_{\mu_d,\xi} \big[\xi(X^{X^{x}_{s}}_{t-s}) \big| (G_r,X_r^x)_{r\in [0,s]} \big] \Big| 
				\ind_{\sigma_{x,i}\leftrightarrow_0\sigma_{y,j}}
				\ind_{{\tau}_{{\rm meet}}^{x,y}>s} \Big] \\
				& \leq 2\, \P_{\mu_d}\Big(\{\tau_{\rm meet}^{x,y}>s\}\cap\{\sigma_{x,i}
				\leftrightarrow_0\sigma_{y, j} \}\Big)\, d_{\rm TV}(t-s),
		\end{aligned}
		\end{equation}
which concludes the proof.
\end{proof}
	
\begin{proof}[Proof of Proposition~\ref{separate2-new}]
From \eqref{eq:lemmaboundTV} we obtain
		\begin{equation}
		\label{giant4}
		\begin{aligned}
				\sup_{\xi \in \{0,1\}^n} & \left| \E_{\mu_d, \xi} \left[\mathcal{D}_{t_n}\right]
				- 2 \P_{\mu_d}\big( \tau_{{\rm meet}}^{\rm edge}>\tfrac12 t_n \big)
				\mathcal{O}^n(\xi) \big(1-\mathcal{O}^n(\xi)\big) \right| \\
				& \qquad \leq 2 \P_{\mu_d} \big( \tau_{\rm meet}^{\rm edge} \in \big( \tfrac12 t_n,t_{n} \big] \big) 
				+ 4\P_{\mu_d} \big( \tau_{\rm meet}^{\rm edge} > \tfrac12 t_n \big) d_{\rm TV}\big( \tfrac12 t_n \big).
		\end{aligned}
		\end{equation}
The desired result follows by choosing $t_n = \delta_n\alpha_n$. Note that Proposition~\ref{lemma-new} yields 
\begin{equation}
\P_{\mu_d}\big( \tau_{{\rm meet}}^{\rm edge}>\tfrac12 t_n \big) \to \vartheta_{d, \nu}\,
\end{equation}
because $\frac{\log n}{n} \ll \delta_{n} \ll 1$. The same Proposition~\ref{lemma-new} can be used to conclude that
\begin{equation}
\P_{\mu_d} \big( \tau_{\rm meet}^{\rm edge} \in \big( \tfrac12 t_n,t_{n} \big] \big) \to 0\,,
\end{equation}
while Proposition~\ref{DynamicMixingMarginal} yields $d_{\rm TV}(\delta_n \alpha_n) \to 0$, and concludes the proof.
\end{proof}

%%%%%%% SECTION 7 %%%%%%%%%%%%%%%%%%%%%%%%%%%%%%%%
	
\section{Properties of the diffusion constant}
\label{sec:theta}

In this section we prove Propositions~\ref{lem:limit-theta}--\ref{pr:der-theta-0}. Recall that $\beta_d = \sqrt{d-1}$ and $\rho_d = \frac{2}{d}\sqrt{d-1} = \frac{2}{d}\beta_d$. 

\medskip\noindent
{\bf 1.}
It follows from \eqref{def-Delta} that
\begin{equation}
\Delta_{d,\nu} \sim \frac{\rho_d}{\nu}\,,\qquad \nu\to\infty\,,
\end{equation}
which together with \eqref{def-theta} proves the second limit in \eqref{eq:limit-theta-a} and \eqref{eq:per-theta-a}. It also
follows from \eqref{def-Delta} that
\begin{equation}
\Delta_{d,\nu} \sim \frac{\rho_d}{\nu+2}\,,\qquad d\to\infty\,,
\end{equation}
which together with \eqref{def-theta} proves the limit in \eqref{eq:limit-theta-b} and \eqref{eq:per-theta-b}. 

\medskip\noindent		
{\bf 2.}
It follows from \eqref{def-Delta} that 
\begin{equation}
\lim_{\nu\downarrow 0} \Delta_{d,\nu} = \Delta_{d,0} = \mathfrak{d}_d\,, 
\end{equation}
with $\mathfrak{d}_d$ the solution of the equation
\begin{equation}
\mathfrak{d}_d = \(\frac2{\rho_d}-\mathfrak{d}_d\)^{-1}\,.
\end{equation}
The latter gives $\rho_d \mathfrak{d}_d^2-2\mathfrak{d}_d+\rho_d=0$, which yields, via \eqref{def-beta} and \eqref{def-rho},
\begin{equation}
\label{eq:sol-C}
\mathfrak{d}_d \in\left\{ \frac{1}{\beta_d},\beta_d\right\}.
\end{equation}
Consequently,
\begin{equation}
\vartheta_{d,0} \in\left\{\frac{d-2}{d-1},0\right\}\,.
\end{equation}	
Note that \eqref{def-Ri} and \eqref{def-Delta-i} show that $\Delta_{d,\nu}<\beta_d$, so that the second solution in \eqref{eq:sol-C} can be discarded. We conclude that $\vartheta_{d,0} = \frac{d-2}{d-1}$, which proves the first limit in \eqref{eq:limit-theta-a}. 

\medskip\noindent
{\bf 3.}
To prove the first limit in \eqref{eq:per-theta-a} we return to the recursion in \eqref{eq:recursion-Delta}. Differentiating this recursion with respect to $\nu$, we get
\begin{equation}
\frac{\partial}{\partial\nu} \Delta_{d,\nu}(i) 
=  - \Delta_{d,\nu}(i)^2 \left[\frac{i+1}{\rho_d} - \frac{\partial}{\partial\nu} \Delta_{d,\nu}(i+1)\right]\,,
\qquad i \in \N_0\,,
\end{equation}
which can be iterated to obtain
\begin{equation}
\label{eq:derDel}
\frac{\partial}{\partial\nu} \Delta_{d,\nu}(0) = - \frac{1}{\rho_d} \sum_{i=0}^\infty (i+1) 
\prod_{j=0}^i \Delta_{d,\nu}(j)^2\,.
\end{equation}
By \eqref{def-theta},
\begin{equation}
\frac{1}{\nu} (\vartheta_{d,\nu}-\vartheta_{d,0}) = - \frac{1}{\beta_d}\, 
\frac{1}{\nu} (\Delta_{d,\nu}-\Delta_{d,0})\,.
\end{equation}
Passing to the limit $\nu \downarrow 0$, we get
\begin{equation}
\lim_{\nu \downarrow 0} \frac{1}{\nu} (\vartheta_{d,\nu}-\vartheta_{d,0})
= \frac{1}{\beta_d} \frac{1}{\rho_d} \sum_{i=0}^\infty (i+1) 
\prod_{j=0}^i \Delta_{d,0}(j)^2\,.
\end{equation}
Since $\Delta_{d,0}(j) = \Delta_{d,0}(0) = \Delta_{d,0} = \mathfrak{d}_d = 1/\beta_d$ for all $j \in \N$, we find
\begin{equation}
\begin{aligned}
\lim_{\nu \downarrow 0} \frac{1}{\nu} (\vartheta_{d,\nu}-\vartheta_{d,0})
&= \frac{d}{2\beta_d^2} \sum_{i=0}^\infty (i+1) \mathfrak{d}_d^{2(i+1)}
= \frac{d}{2\beta_d^2} \frac{\mathfrak{d}_d^2}{(1-\mathfrak{d}_d^2)^2}\\ 
&=  \frac{d}{2\beta_d^2} \frac{\beta_d^{-2}}{(1-\beta_d^{-2})^2} 
= \frac{d}{2} \frac{1}{(\beta_d^2-1)^2} = \frac{d}{2(d-2)^2}\,,
\end{aligned}
\end{equation}
which completes the proof of the first limit in \eqref{eq:per-theta-a}.

\medskip\noindent
{\bf 4.}
It is immediate from \eqref{eq:derDel} that $\nu \mapsto \Delta_{d,\nu}$ is strictly decreasing. Hence, by \eqref{def-theta}, $\nu \mapsto \vartheta_{d,\nu}$ is strictly increasing. Put $\chi_{d,\nu} = \Delta_{d,\nu}/\beta_d$. Then the recursion in \eqref{eq:recursion-Delta} reads
\begin{equation}
\chi_{d,\nu}(i) = \Big(d(1+\tfrac12(i+1)\nu) - (d-1)\chi_{d,\nu}(i+1)\Big)^{-1}\,.
\end{equation} 
Differentiating this recursion with respect to $d$, we get
\begin{equation}
\begin{aligned}
\frac{\partial}{\partial d}\,\chi_{d,\nu}(i) 
&= - \chi_{d,\nu}(i)^2\left[(1+\tfrac12(i+1)\nu) - \chi_{d,\nu}(i+1) 
- (d-1) \frac{\partial}{\partial d}\,\chi_{d,\nu}(i+1)\right]\\
&= - \chi_{d,\nu}(i)^2\left[\frac{1}{d}\left(\frac{1}{\chi_{d,\nu}(i)} - \chi_{d,\nu}(i)\right) 
- (d-1) \frac{\partial}{\partial d}\,\chi_{d,\nu}(i+1) \right]\,,
\end{aligned}
\end{equation}
which can be iterated to obtain
\begin{equation}
\frac{\partial}{\partial d}\,\chi_{d,\nu}(0) 
= - \sum_{i=0}^\infty \frac{1}{d}\left(\frac{1}{\chi_{d,\nu}(i)} - \chi_{d,\nu}(i)\right) (d-1)^i 
\prod_{j=0}^i \chi_{d,\nu}(j)^2\,.
\end{equation}
Since the right-hand side is negative (recall that $\chi_{d,\nu}<1$), we see that $d \mapsto \chi_{d,\nu}(0)$ is strictly decreasing. Hence, by \eqref{def-theta}, $d \mapsto \vartheta_{d,\nu}$ is strictly increasing.

%%%%%%% REFERENCES %%%%%%%%%%%%%%%%%%%%%%%%%%%%%%%%

%%%%%%%%%%%%%%%%%%%%%%%%%%%%%%%%%%%%%%%%%%%%%%%%%%
	

\begin{thebibliography}{}
		
		\bibitem{A82}
		D.\ Aldous.
		\newblock Markov chains with almost exponential hitting times.
		\newblock {\em Stoch.\ Proc.\ Appl.}, 13(3):305--310, 1982.
		
		\bibitem{A12}
		D.\ Aldous.
		\newblock Lecture notes of the course {\em Finite Markov Information Exchange Processes}.
		\newblock \url{https://www.stat.berkeley.edu/~aldous/FMIE/warwick_3.pdf}, 2012.

		\bibitem{A13}
		D.\ Aldous.
		\newblock {\em Probability Approximations via the Poisson Clumping Heuristic}.
		\newblock Series: Springer Science \& Business Media, Vol.\ 77.
		\newblock {\em Springer}, 2013.		
			
		\bibitem{AB92}
		D.\ Aldous and M.\ Brown.
		\newblock Inequalities for rare events in time-reversible Markov chains I.
		\newblock Stochastic Inequalities.
		\newblock {\rm Lecture Notes Monograph Series}, 1--16, 1992.
		
		\bibitem{AB93}
		D.\ Aldous and M.\ Brown.
		\newblock Inequalities for rare events in time-reversible Markov chains II.
		\newblock {\em Stoch.\ Proc.\ Appl.}, 44(1):15--25, 1993.
		
		\bibitem{AF02}
		D.\ Aldous and J.A.\ Fill.
		\newblock {\em Reversible Markov Chains and Random Walks on Graphs}.
                 \newblock Unfinished monograph, 2014.
                 \newblock \url{https://www.stat.berkeley.edu/users/aldous/RWG/book.pdf}
		
		\bibitem{ABHdHQ24}
		L.\ Avena, R.\ Baldasso, R.S.\ Hazra, F.\ den Hollander, and M.\ Quattropani.
		\newblock Discordant edges for the voter model on regular random graphs.
		\newblock {\em ALEA Lat.\ Am.\ J.\ Probab.\ Math.\ Stat.}, (21):431--464, 2024.
		
		\bibitem{ACHQpr}
		L.\ Avena, F.\ Capannoli, R.S.\ Hazra, and M.\ Quattropani.
		\newblock Meeting, coalescence and consensus time on directed random graphs.
		\newblock {\em Ann.\ Appl.\ Probab.}, 34(5):4940--4997, 2024.
		
		\bibitem{AGHH18}
		L.\ Avena, H.\ Guldas, R.\ van der Hofstad, and F.\ den Hollander.
		\newblock Mixing times of random walks on dynamic configuration models.
		\newblock {\em Ann.\ Appl.\ Probab.}, 28(4):1977--2022, 2018.
		
		\bibitem{AGHH19}
		L.\ Avena, H.\ Guldas, R.\ van der Hofstad, and F.\ den Hollander.
		\newblock Random walks on dynamic configuration models:	A trichotomy.
		\newblock {\em Stoch.\ Proc.\ Appl.}, 129(9):3360--3375, 2019.
		
		\bibitem{AGHHN22}
		L.\ Avena, H.\ Guldas, R.\ van der Hofstad, F.\ den Hollander, and O.\ Nagy.
		\newblock Linking the mixing times of random walks on static and dynamic random graphs.
		\newblock {\em Stoch.\ Proc.\ Appl.}, 153:145--182, 2022.
		
		\bibitem{BDZ15}
		A.\ Basak, R.\ Durrett, and Y. Zhang.
		\newblock The evolving voter model on thick graphs.
		\newblock {\em arXiv preprint}, arXiv:1512.07871, 2015.
	
		\bibitem{BS17}
		R.\ Basu and A.\ Sly.
		\newblock Evolving voter model on dense random graphs.
		\newblock {\em Ann.\ Appl.\ Probab.}, 27(2):1235--1288, 2017.
		
		%\bibitem{B80}
		%B.\ Bollob\`as.
		%\newblock A probabilistic proof of an asymptotic formula for the number of labelled regular graphs.
		%\newblock {\em European J.\ Combin.}, 1(4):311--316, 1980.
		
		\bibitem{B20}
	    	C.\ Bordenave.
		\newblock A new proof of Friedman's second eigenvalue theorem and its extension to random lifts. 
		\newblock {\em Ann.\ Sci.\ \'Ecole Norm.\ Sup.}, 4(6):1393--1439, 2020.
		
		\bibitem{BCS18}
		C.\ Bordenave, P.\ Caputo and J. Salez.
		\newblock Random walk on sparse random digraphs.
		\newblock {\em Probab.\ Theory Relat.\ Fields}, 170(3):933--960, 2018.
		
		\bibitem{BdHM22pr}
		P.\ Braunsteins, F. den Hollander and  M.\ Mandjes.
		\newblock Graphon-valued processes with vertex-level fluctuations.
		\newblock Preprint at arXiv:2209.01544 (2022).
		
		\bibitem{C24pr}
		F.\ Capannoli. 
		\newblock Evolution of discordant edges in the voter model on random sparse digraphs.
		\newblock To appear in {\em Electron.\ J.\ Probab.}
		
		\bibitem{CQ21}
		P.\ Caputo and M.\ Quattropani. 
		\newblock Mixing time trichotomy in regenerating dynamic digraphs.
		\newblock {\em Stoch.\ Proc.\ Appl.}, 137:222--251, 2021.
				
		\bibitem{C21}
		Y.-T.\ Chen.
		\newblock  Precise asymptotics of some meeting times arising from the voter model on large random regular graphs.
		\newblock {\em Electron. Commun. Probab.}, 26:1--13, 2021.
				
		\bibitem{CCC16}
		Y.-T.\ Chen, J.\ Choi, and J.T.\ Cox.
		\newblock On the convergence of densities of finite voter models to the Wright-Fisher diffusion.
		\newblock {\em Ann.\ Inst.\ Henri Poincar\'e Probab.\ Stat.}, 52(1):286--322, 2016.
		
		\bibitem{CF05}
		C.\ Cooper and A.\ Frieze.
		\newblock The cover time of random regular graphs. 
		\newblock {\em SIAM J.\ Discr.\ Math.}, 18(4):728--740, 2005.
		
		\bibitem{CFR10}
		C.\ Cooper, A.\ Frieze, and T.\ Radzik.
		\newblock Multiple random walks in random regular graphs.
		\newblock {\em SIAM J.\ Discr.\ Math.}, 23(4):1738--1761, 2010.
			
		\bibitem{C89}
		J.T.\ Cox.
		\newblock Coalescing random walks and voter model consensus times on the torus in $\mathbb{Z}^d$.
		\newblock {\em Ann. Probab.}, 17(4):1333--1366, 1989.
		
		\bibitem{CG90}
		J.T.\ Cox and A.\ Greven.
		\newblock On the long term behavior of some finite particle systems.
		\newblock {\em Probab.\ Theory Relat.\ Fields}, 85:195--237, 1990.
	
	        \bibitem{DHJN17}
	        S.\ Dommers, F.\ den Hollander, O.\ Jovanovski and F.\ R.\ Nardi
	        \newblock Metastability for Glauber dynamics on random graphs.
	        \newblock {\em Ann. Appl. Probab.}, 27(4):2130--2158, 2017.
	
		\bibitem{D08}
		R.\ Durrett.
		\newblock Probability Models for DNA Sequence Evolution (2nd edition).
		\newblock Series: Probability and Its Applications.
		\newblock {\em Springer}, 2008.
		
		\bibitem{D10}
		R.\ Durrett.
		\newblock {\em Random Graph Dynamics}.
		\newblock Cambridge University Press.
		\newblock Cambridge.
		\newblock 2010.
		
		\bibitem{DGLMSSSV12}
		R.\ Durrett, J.P.\ Gleeson, A.L.\ Lloyd, P.J.\ Mucha, F.\ Shi, D.\ Sivakoff, J.E.S.\ Socolar, and C.\ Varghese.
		\newblock Graph fission in an evolving voter model.
		\newblock {\em Proc.\ Natl.\ Acad.\ Sci.\ USA}, (109):3682--3687, 2012.
		
		\bibitem{Fernley24}
		R.\ Fernley.
		\newblock The Phase Transition of the Voter Model on Evolving Scale-Free Networks.
		\newblock {\em arXiv preprint}, arXiv:2406.03037, 2024.
		
		\bibitem{F03}
		J.\ Friedman.
		\newblock A proof of Alon's second eigenvalue conjecture.
		\newblock Proceedings of the thirty-fifth annual ACM symposium on Theory of computing.
		\newblock 2003.
		
		\bibitem{HHK21}
		V.\ Hao Cai, R.\ van der Hofstad and T.\ Kumagai
		\newblock Glauber dynamics for Ising models on random regular graphs: cut-off and metastability
		\newblock {\em ALEA Lat.\ Am.\ J. Probab.\ Math.\ Stat.}, 18:1441--1482, 2021.
			
		\bibitem{HS20}
		J.\ Hermon, P. Sousi.
		\newblock	A comparison principle for random walk on dynamical percolation.
		\newblock {\em Ann.\ Probab.\ } 48, 2952–2987, 2020.
	
		\bibitem{RvdH17}
    		R.\ van der Hofstad.
     		\newblock {\em Random Graphs and Complex Networks}. Volume 1.
    		\newblock Cambridge Series in Statistical and Probabilistic Mathematics.
		\newblock Cambridge University Press, Cambridge.
    		\newblock 2017.

		\bibitem{RvdH24}
    		R.\ van der Hofstad.
     		\newblock {\em Random Graphs and Complex Networks}. Volume 2.
    		\newblock Cambridge Series in Statistical and Probabilistic Mathematics.
		\newblock Cambridge University Press, Cambridge.
    		\newblock 2024.
		
		\bibitem{HL75}
		R.\ A.\ Holley and T.\ M.\ Liggett
		\newblock  Ergodic Theorems for Weakly Interacting Infinite Systems and the Voter Model.
		\newblock {\em Ann.\ Probab.}, 3(4):643--663, 1975.
		
		\bibitem{HN06} 
		P.\ Holme, M.\ Newman. 
		\newblock Nonequilibrium phase transition in the coevolution of networks and opinions.
		\newblock {\em Phys. Rev.\ E} 74, 056108, 2006.
		
		\bibitem{JM17}
		E.\ Jacob, P.\ M\"{o}rters.
		\newblock The contact process on scale-free networks evolving by vertex updating.
		\newblock {\em R.\ Soc.\ Open Sci.} 4, 170081, 2017.
		
		\bibitem{JLM24pr}
		E.\ Jacob, A.\ Linker, P.\ M\"{o}rters.		
		\newblock Metastability of the contact process on slowly evolving scale-free networks.
		\newblock {\em arXiv preprint}, arXiv:2309.17040, 2024.

		\bibitem{LS10}
		E.\ Lubetzky and A.\ Sly.
		\newblock Cutoff phenomena for random walks on random regular graphs.
		\newblock {\em Duke Math. J.}, 153(3):475--510, 2010.
				
		\bibitem{MQS21}
		F.\ Manzo, M.\ Quattropani, and E.\ Scoppola.
		\newblock A probabilistic proof of Cooper \& Frieze’s ``First Visit Time Lemma''.
		\newblock {\em ALEA Lat.\ Am.\ J. Probab.\ Math.\ Stat.}, 18(2):1739--1758, 2021.
		
		\bibitem{M24}
		M.\ Markering.
		\newblock Cover times for random walk on dynamical percolation.
		\newblock {\em ALEA Lat.\ Am.\ J. Probab.\ Math.\ Stat.}, 21:907--921, 2024.
		
		\bibitem{MV16}
		J.-C.\ Mourrat and D.\ Valesin.
		\newblock Phase transition of the contact process on random regular graphs.
		\newblock {\em Electron.\ J.\ Probab.}, 21:1--17, 2016.
		
		\bibitem{O13}
		R.I.\ Oliveira
		\newblock  Mean field conditions for coalescing random walks. 
		\newblock {\em Ann.\ Probab.}, 41(5):3420--3461, 2013.
		
		\bibitem{PSS15}
		Y.\ Peres, A.\ Stauffer, J.E.\ Steif.
		\newblock Random walks on dynamical percolation: mixing times, mean squared displacement and hitting times
		\newblock {\em Probab.\ Theory Relat.\ Fields} 162, 487--530, 2015.

		\bibitem{PSS20}
		Y.\ Peres, P.\ Sousi, J.E.\ Steif.
		\newblock Mixing time for random walk on supercritical dynamical percolation.
		\newblock {\em Probab.\ Theory Relat.\ Fields} 176, 809--849, 2020

		\bibitem{P1913}
		O.\ Perron. 
		\newblock {\em Die Lehre von den Kettenbr\"uchen}.
		\newblock Teubner, Leipzig, 1913.
		
		\bibitem{QS23}
		M.\ Quattropani and F.\ Sau. 
		\newblock On the meeting of random walks on random DFA.
		\newblock {\em Stoch.\ Proc.\ Appl.}, 166, 104225, 2023.
		
		\bibitem{SV23pr}
		B.\ Shapira and D.\ Valesin.
		\newblock The contact process on dynamic regular graphs: monotonicity and subcritical phase.
		\newblock {\em arXiv preprint}, arXiv:2309.17040, 2023.
	
		\bibitem{SOV21pr}
		G.L.B.\ da Silva, R.I.\ Oliveira, and D.\ Valesin.
		\newblock The contact process over a dynamical $d$-regular graph.
		\newblock {\em arXiv preprint}, arXiv:2111.11757, 2021.
		
		\bibitem{ST20}
		P.\ Sousi, S.\ Thomas.
		\newblock Cutoff for random walk on dynamical {E}rd{\H o}s--{R}{\'e}nyi graph.
		\newblock {\em Ann.\ Inst.\ Henri Poincar{\'e}, Prob.\ Stat.} 56, 2745--2773, 2020.
			
\end{thebibliography}
\end{document}